\documentclass[11pt]{article}

\usepackage{tikz}
\usepackage{listings}
\usepackage{pdflscape}
\usepackage{epsfig}
\usepackage{enumitem}
\usepackage{framed}
\usepackage{lscape}
\usepackage{longtable}
\usepackage{multirow}
\usepackage{amsmath, amssymb, amsthm}
\usepackage{mathtools}
\usepackage{graphicx}
\usepackage{subfigure}
\usepackage{subcaption}
\usepackage{color}
\usepackage{placeins}
\usepackage{url}
\usepackage{cases}
\usepackage{hyperref}
\usepackage{verbatim}
\usepackage{booktabs}
\usepackage{arydshln}

\usepackage{algorithm}
\usepackage[noend]{algorithmic}
\oddsidemargin
0pt
\evensidemargin
0pt
\marginparwidth
10pt
\marginparsep
10pt
\topmargin
-20pt

\textwidth
6.5in
\textheight
8.5in
\parindent
=
20pt

\newcommand{\B}{\mathbb{B}}
\newcommand{\N}{\mathbb{N}}
\newcommand{\D}{\mathbb{D}}

\newcommand{\R}{\mathbb{R}}

\newcommand{\bx}{\bar{x}}

\newcommand{\Min}{\operatornamewithlimits{minimize}}

\newcommand{\Max}{\operatornamewithlimits{maximize}}
\newcommand{\MaxVal}{\operatornamewithlimits{maximum}}

\newcommand{\Tr}{\operatorname{Tr}}

\newcommand{\Hausd}{\mathbb{H}}
\newcommand{\conv}{\operatorname{conv}}

\DeclareMathOperator*{\argmin}{argmin}
\DeclareMathOperator*{\argmax}{argmax}

\def\approxleq{ \kern3pt \mbox{\raisebox{.6ex}{$<$}} \kern-8pt \mbox{\raisebox{-.6ex}{$\sim$}} \kern5pt}

\def\dom{\textup{dom}}

\def\bx{\bar{x}}
\def\hx{\hat{x}}

\def\by{\bar{y}}

\def\dist{{\textup{dist}}}

\newlength{\len}
\usepackage{thmtools}
\declaretheoremstyle[bodyfont=\sl]{normalbody}

\theoremstyle{definition}
\newtheorem{theorem}{Theorem}
\newtheorem{proposition}{Proposition}
\newtheorem{lemma}{Lemma}
\newtheorem{corollary}{Corollary}
\newtheorem{remark}{Remark}
\newtheorem{assumption}{Assumption}
\newtheorem{definition}{Definition}
\newtheorem{example}{Example}
\allowdisplaybreaks[4]
\title{Subgradient Regularization: A Descent-Oriented\\
Subgradient Method for Nonsmooth Optimization}

\author{Hanyang Li
\thanks{Department of Industrial Engineering and Operations Research, University of California, Berkeley, CA 94720, USA (\href{mailto:hanyang_li@berkeley.edu}{hanyang\_li@berkeley.edu})}
\and Ying Cui
\thanks{Department of Industrial Engineering and Operations Research, University of California, Berkeley, CA 94720, USA  (\href{mailto:yingcui@berkeley.edu}{yingcui@berkeley.edu})}
}

\begin{document}
\bibliographystyle{plain}
\maketitle

\begin{abstract}
In nonsmooth optimization, a negative subgradient is not necessarily a descent direction, making the design of convergent descent methods based on  zeroth-order and first-order information a challenging task. The well-studied bundle methods and gradient sampling algorithms construct descent directions by aggregating subgradients at nearby points in  seemingly different ways, and are often complicated or lack deterministic guarantees. 
In this work, we identify a unifying principle behind these approaches, and develop a general framework of descent methods under the abstract principle that provably converge to stationary points. Within this framework, we introduce a simple yet effective technique, called subgradient regularization, to generate stable descent directions for a broad class of nonsmooth marginal functions, including finite maxima or minima of smooth functions. 
When applied to the composition of a convex function with a smooth map, the method naturally recovers the prox-linear method and, as a byproduct, provides  a new dual interpretation of this classical algorithm.
Numerical experiments demonstrate the effectiveness of our methods on several challenging classes of nonsmooth optimization problems, including  the minimization of Nesterov's nonsmooth Chebyshev-Rosenbrock function.


\end{abstract}

\section{Introduction}

We consider the minimization of a nonconvex, nonsmooth and locally Lipschitz continuous function $f: \R^n \to \R$. The Clarke subdifferential of $f$, denoted by $\partial f$, generalizes the concept of the subdifferential  in convex analysis and reduces to the gradient when $f$ is smooth. A point $\bx$ is said to be a stationary point of $f$ if $0 \in \partial f(\bx)$. 
In nonsmooth optimization, it is well known that an arbitrary negative subgradient $v\in \partial f(x)$ may fail to yield a descent direction at $x$. 
While a descent direction can be obtained by using the minimal norm subgradient
\begin{equation}\label{eq:Steepest_descent}
   g_x \triangleq -\argmin_{v \in \partial f(x)} \|v\|,
\end{equation}
the resulting steepest descent method
often suffers from  zigzagging behavior  and may even converge to non-stationary points; see, e.g.,~\cite[Section 2.2]{hiriart1996convex}. These issues arise from the discontinuity of the mapping $x \mapsto g_x$, posing a fundamental challenge in the design of efficient descent methods for nonsmooth minimization.

To develop fast descent methods  with convergence guarantees, significant efforts have been devoted to constructing stable descent directions using enlarged subdifferentials,  defined in an extended space $(x, \epsilon) \in \R^n \times (0, +\infty)$. The auxiliary parameter $\epsilon$ plays a crucial role in ensuring the continuity of the subdifferential mapping. Existing  algorithms can be broadly classified into two categories based on how the enlarged subdifferentials are constructed: Goldstein-type methods and bundle-type methods.

\vspace{0.1in}
\noindent
\textbf{Goldstein-type methods}. The seminal work of Goldstein~\cite{goldstein1977optimization} introduced  the concept of the \textsl{Goldstein $\epsilon$-subdifferential}:
\[
    \partial^{\text{G}}_{\epsilon}f(x) \triangleq \conv\Big( \bigcup 
    \left\{\partial f(z) \,\middle|\, z \in \overline{\B}_{\epsilon}(x) \right\} \Big), \quad \epsilon>0,
\]
where $\overline{\B}_{\epsilon}(x) \triangleq \{y \in \R^{n}\mid \|y - x\| \leq \epsilon\}$ and $\conv(S)$ represents the convex hull of a set $S$. Goldstein's subgradient method updates the iterates as
\begin{equation}\label{eq:Goldstein_direc}
    \hspace{0.5in}
    x^{k+1}= x^k- \epsilon_k \, \frac{g^{k}}{\|g^{k}\|}
    \quad\text{with}\quad g^{k} = \argmin_{v \in \partial^{\text{G}}_{\epsilon_k} f(x^k)}\|v\| (\neq 0),
\tag{Goldstein}
\end{equation}
which guarantees descent in the objective: $f(x^{k+1}) \leq f(x^k) - \epsilon_k \|g^{k}\|$. With properly chosen stepsizes $\{\epsilon_k\}$, the  sequence $\{x^k\}$ converges  subsequentially  to stationary points of $f$.

However, computing the direction in \eqref{eq:Goldstein_direc} is often challenging in practice, as it requires access to all subgradients within a neighborhood of  $x^k$.
The gradient sampling algorithm~\cite{burke2002approximating,burke2005robust,burke2020gradient} circumvents this by approximating the Goldstein $\epsilon$-subdifferential using randomly sampled points in $\overline{\B}_{\epsilon}(x)$. 
Recently, several variants of Goldstein-type methods have been proposed in~\cite{zhang2020complexity,davis2022gradient,tian2022finite}, focusing on efficient random approximation of the direction in \eqref{eq:Goldstein_direc}  to find a point $x$ satisfying $\dist\big(0, \partial^{\text{G}}_{\epsilon}f(x)\big) \leq \delta$ with complexity guarantees.
The recent work \cite{davis2024local} proposed a randomized scheme of the Goldstein method that is locally nearly linearly convergent by exploring underlying smooth substructures. Without considering the complexity inherent in the minimal norm oracle for finding Goldstein directions, \cite{kong2024lipschitz} identifies abstract properties that facilitate the near-linear convergence in Goldstein-type methods. Those properties involve a new modulus measuring the growth of a function through the Goldstein $\epsilon$-subdifferential. 

\vspace{0.1in}
\noindent
\textbf{Bundle-type methods}. Bundle methods, originating from the works~\cite{wolfe1975method,lemarechal1975extension},
are based on a different notion of enlarged subdifferentials, particularly suited for convex nonsmooth functions. Specifically, for a convex function $f$, the \textsl{$\epsilon$-subdifferential} at $x$ is defined as
\[
    \partial_{\epsilon}f(x) \triangleq \left\{ v \,\middle|\, f(z) \geq f(x) + v^{\top}(z-x) - \epsilon \text{ for all }z \right\}, \qquad\epsilon > 0.
\]
This definition allows subgradients computed at one point to be effectively reused at another. In particular, if $g \in \partial f(\hx)$ at some $\hx$, then $g$ can serve as an inexact subgradient at another point $x$, satisfying $g \in \partial_{\epsilon} f(x)$ with $\epsilon = f(x) - f(\hx) - g^\top(x - \hx)$. This property, known as the {transportation formula}~\cite[Chapter XI, Section 4.2]{hiriart1996convex}, allows one to iteratively build up a convex polyhedron $S_{k} \subset \partial_{\epsilon_k}f(x^k)$ by aggregating objective values and subgradients from past iterations. In the dual interpretation of bundle methods~\cite[Lemma 3.1 and equation (3.20)]{lemarechal2000nonsmooth}, the so-called serious step with the stepsize $\mu_k > 0$ corresponds to the update
\begin{equation}\label{eq:Bundle_direc}
    \hspace{0.5in}
    x^{k+1} = x^k- {\mu_k} \, \tilde g^{k}
    \quad\text{with}\quad \tilde g^{k} = \argmin_{v \in S_k}\|v\| (\neq 0), 
\tag{Bundle}
\end{equation}
ensuring a decrease in the objective value if $S_k$  sufficiently approximates $\partial_{\epsilon_k}f(x^k)$. Between serious steps, null steps are used to refine the approximation $S_k$ by incorporating more objective values and subgradient information at trial points. Bundle methods have also been extended to certain  classes of nonconvex functions; see, for example,~\cite{Mifflin1982modification,hare2010redistributed,de2019proximal}. 

In terms of the convergence rate, the pioneering work~\cite{robinson1999linear} shows that serious steps converge at a linear rate for convex nonsmooth functions whose subdifferential is metrically regular. This analysis has been recently extended to weakly convex functions~\cite{atenas2023unified}. For convex functions admitting $\mathcal{VU}$-structures~\cite{mifflin2000mathcalvu}, a variant of bundle methods that incorporates second-order information relative to a smooth subspace enjoys superlinear convergence in terms of serious steps~\cite{mifflin2005algorithm}. Accounting for both serious and null steps, \cite{diaz2023optimal} shows that the bundle method achieves linear convergence under the more restrictive sharp growth condition, alongside convergence rate analysis under various other smoothness and growth conditions.

\vspace{0.1in}
Observe that the frameworks in \eqref{eq:Goldstein_direc} and \eqref{eq:Bundle_direc}  differ primarily in how they enlarge the Clarke subdifferential $\partial f(x)$ to stabilize the steepest descent direction in \eqref{eq:Steepest_descent}: The Goldstein $\epsilon$-subdifferential $\partial^{\text{G}}_\epsilon f(x)$ expands the set by incorporating subgradients from a neighborhood of  $x$, whereas the $\epsilon$-subdifferential $\partial_\epsilon f(x)$ is defined based on inexact first-order expansions at $x$. If $f$ is convex and $L$-Lipschitz continuous, one can show that $\partial^{\text{G}}_\epsilon f(x) \subset \partial_{2\epsilon L} f(x)$, though the reverse inclusion  does not generally hold.

\subsection{Contributions}

The aim of this paper is two-fold: to establish an abstract guiding principle for the design of convergent descent methods in nonsmooth optimization, and to construct   descent directions for a broad subclass of nonsmooth functions from the angle of subgradient regularization.  

\vskip 0.1in
\noindent
{\bf Descent-oriented subdifferentials.}
Despite the success of Goldstein-type and Bundle-type methods, little research has been devoted to identifying abstract properties of enlarged subdifferentials, defined jointly over $(x, \epsilon)$, that can be systematically employed to design descent algorithms with convergence guarantees. 
Both \eqref{eq:Goldstein_direc} and \eqref{eq:Bundle_direc} follow a common two-step procedure to construct directions: first (approximately) expand subdifferentials via an auxiliary parameter $\epsilon$, then select  the minimal norm element. This process yields a descent direction while avoiding the discontinuity issue of the steepest descent direction \eqref{eq:Steepest_descent}.

In the first part of our work, we formalize this idea through the concept of \textsl{descent-oriented subdifferentials}, a guiding principle for constructing stable descent directions in nonsmooth optimization. Notably,  directions  in  \eqref{eq:Goldstein_direc} and \eqref{eq:Bundle_direc} can be interpreted as particular realizations within our framework. Building on this concept, we further develop a general descent  method (Algorithm \ref{alg1}) that ensures subsequential convergence to stationary points. 

Beyond offering a principled theoretical framework,  descent-oriented subdifferentials also open the door to develop new algorithms for structured nonsmooth problems, where additional problem structure can be leveraged. 
For example, it is known that to minimize a nonsmooth composite function $f(x) \triangleq h (c(x))$, where $c: \R^{n}\to \R^{m}$ is a $C^{1}$-smooth mapping and $h: \R^{m}\to \R$ is a convex function, the prox-linear method~\cite{fletcher1982model,burke1985descent,burke1995gauss,lewis2016proximal,drusvyatskiy2018error} guarantees  descent via  linearizing only the inner smooth map $c$ at each iteration while keeping the outer convex function $h$ unchanged. This naturally motivates our next contribution.

\vskip 0.1in
\noindent
{\bf Descent directions via subgradient regularization. }
For a broad class of nonsmooth marginal functions, including finite maxima or minima of smooth functions,  we consider a simple yet powerful  technique for constructing descent-oriented subdifferentials via  \textsl{subgradient regularization}.  When Algorithm~\ref{alg1} is equipped with this construction, we refer to the resulting method as \textsl{Subgradient-Regularized Descent} (SRDescent).
To illustrate the idea, consider  $f(x) = \max\{f_1(x), f_2(x)\}$ for some smooth functions $f_1$ and $f_2$. One has
\[
    \partial f(x) = \left\{\by_1 \nabla f_1(x) + \by_2 \nabla f_2(x) \,\middle|\, \by \in \argmax_{y \in \Delta^{2}} \big[\,y_1 f_1(x) + y_2 f_2(x)\,\big] \right\},
\]
where $\Delta^{d}\triangleq \big\{ y \in \R^{d}\mid \sum^{d}_{i=1}y_{i}= 1, y \geq 0 \big\}$. Our subgradient regularization approach constructs a descent-oriented subdifferential as follows:
\[
    \left\{
    \by^{\epsilon}_1 \nabla f_1(x) + \by^{\epsilon}_2 \nabla f_2(x)\,\middle|\,
    \by^{\epsilon} \in \argmax\limits_{y \in \Delta^2} \left[
    \begin{array}{rl}
    & y_1 f_1(x) + y_2 f_2(x) \\
    &- \frac{\epsilon}{2} \underbrace{\|y_1 \nabla f_1(x) + y_2 \nabla f_2(x)\|^2}_{\text{subgradient regularization}}
    \end{array}\right]
    \right\},
\]
where the regularization term penalizes the norm of a perturbed subgradient. When $\epsilon>0$, the above set is  outer semicontinuous in $x$; and when $\epsilon \downarrow 0$, it converges to the minimal norm subgradient at $x$.
SRDescent then proceeds by moving along the negative of the regularized subgradient. Unlike the directions used in \eqref{eq:Goldstein_direc} and \eqref{eq:Bundle_direc} that require information from a neighborhood of the current iterate, this approach generates descent directions using  objective values and subgradients at a single reference point. Figure~\ref{fig:illustrate_combine_grad} visualizes the directions generated by gradient sampling, a bundle-type method, and SRDescent. Figure~\ref{fig:illustrate_performance} previews the empirical performance of these approaches. 

\begin{figure}[ht]
    \centering
    \begin{minipage}[b]{0.33\textwidth}
        \centering
        \includegraphics[width=0.9\textwidth]{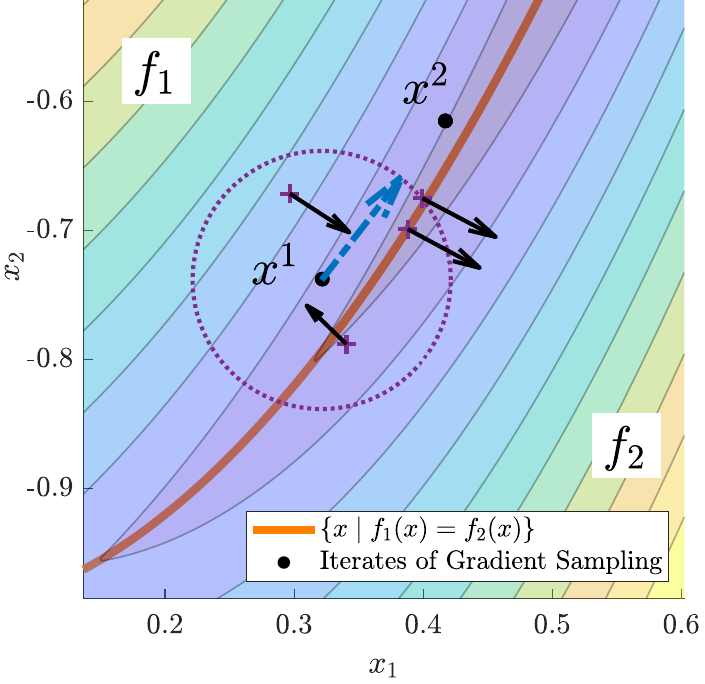}
    \end{minipage}
    \hfill
    \begin{minipage}[b]{0.315\textwidth}
        \centering
        \includegraphics[width=0.9\textwidth]{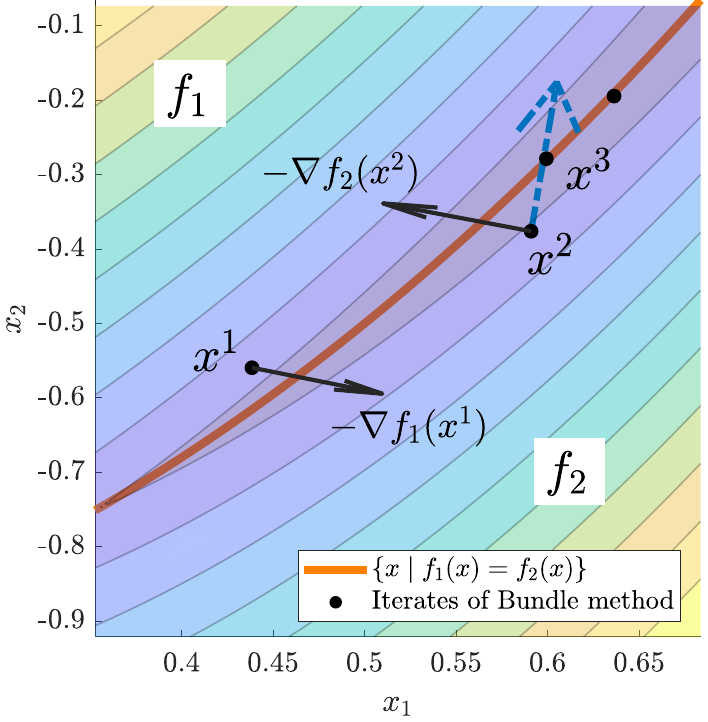}
    \end{minipage}
    \begin{minipage}[b]{0.33\textwidth}
        \centering
        \includegraphics[width=0.9\textwidth]{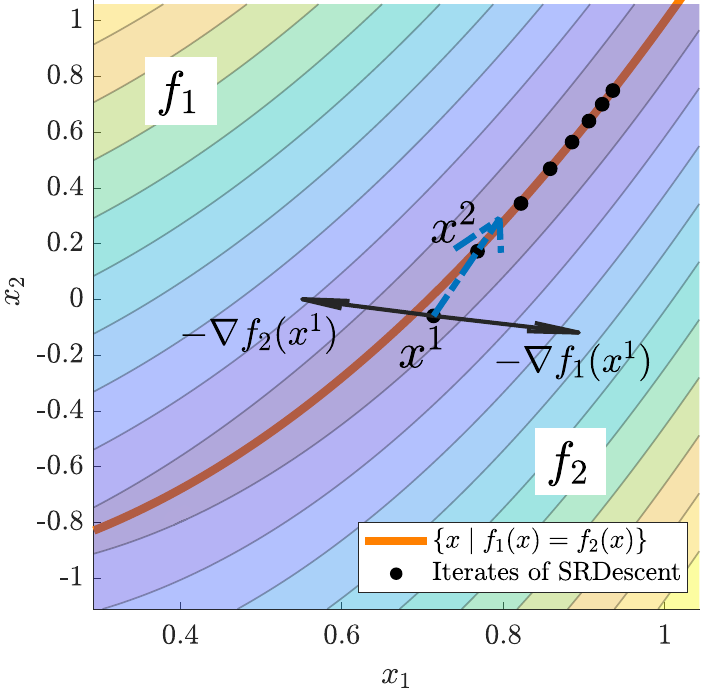}
    \end{minipage}
    \caption{\small Contour plots and descent directions of Nesterov's nonsmooth Chebyshev-Rosenbrock function $f(x) = \frac{1}{4}(x_{1}-1)^{2} + \big|x_{2}-2(x_{1})^{2}+1\big|$, which can be expressed as the maximum of two smooth functions $f_{1}$ and $f_{2}$. In each plot, dashed arrows show the descent direction computed by the respective methods, while different markers indicate the points used to construct these directions. 
    \textbf{Left}: In gradient sampling algorithm, four nearby sample points (crosses) around $x^{1}$ are used to compute a minimal norm convex combination of gradients. 
    \textbf{Middle}: In bundle method, gradients from previous iterates $\nabla f_{1}(x^{1})$ and $\nabla f_{2}(x^{2})$ are aggregated to form a descent direction at $x^{2}$.
    \textbf{Right}: SRDescent directly combines gradients $\nabla f_{1}(x^{1})$ and $\nabla f_{2}(x^{1})$ and their function values at the current iterate $x^{1}$.}
    \label{fig:illustrate_combine_grad}
\end{figure}

\begin{figure}[!ht]
    \centering
    \begin{minipage}[b]{0.32\textwidth}
        \centering
        \includegraphics[width=0.96\textwidth]{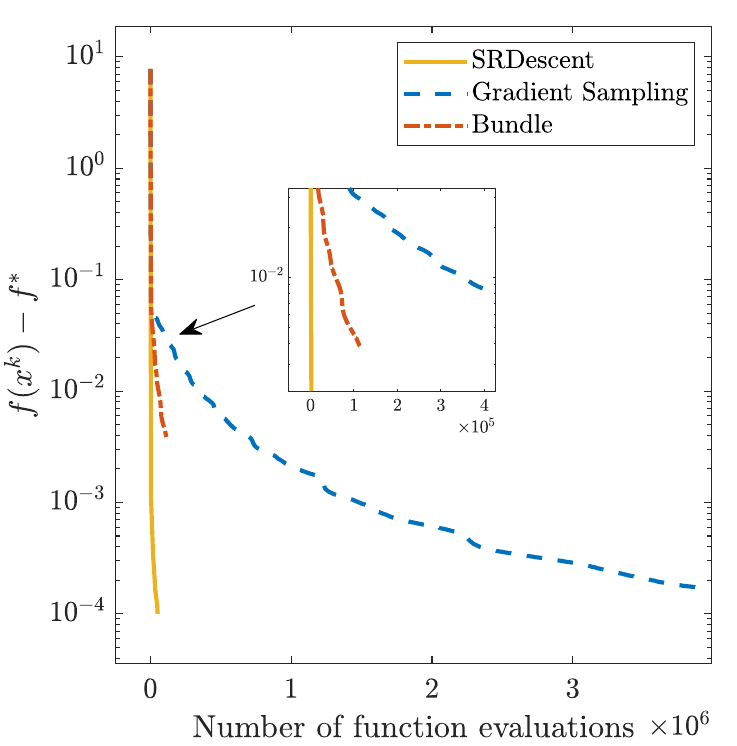}
    \end{minipage}
    \begin{minipage}[b]{0.32\textwidth}
        \centering
        \includegraphics[width=0.96\textwidth]{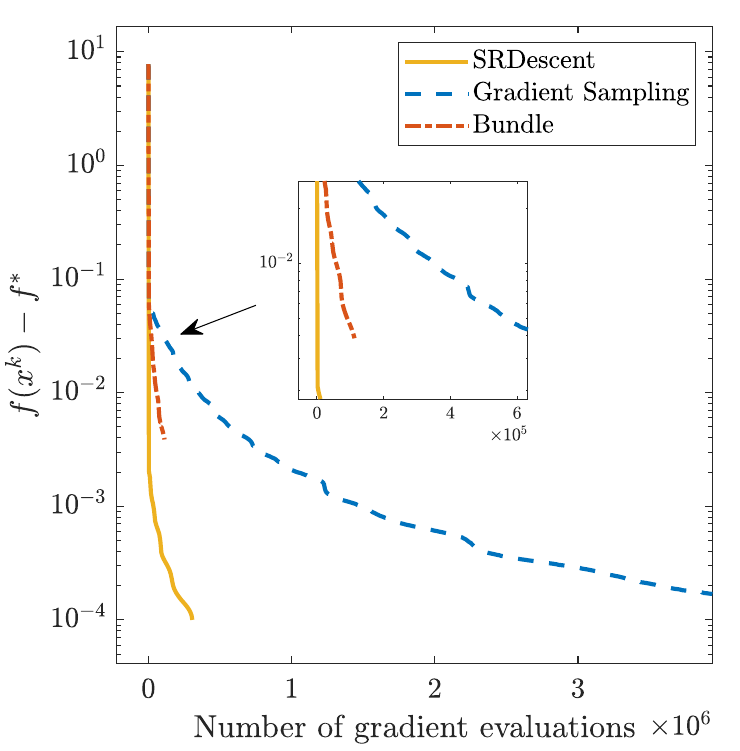}
    \end{minipage}
    \begin{minipage}[b]{0.32\textwidth}
        \centering
        \includegraphics[width=0.97\textwidth]{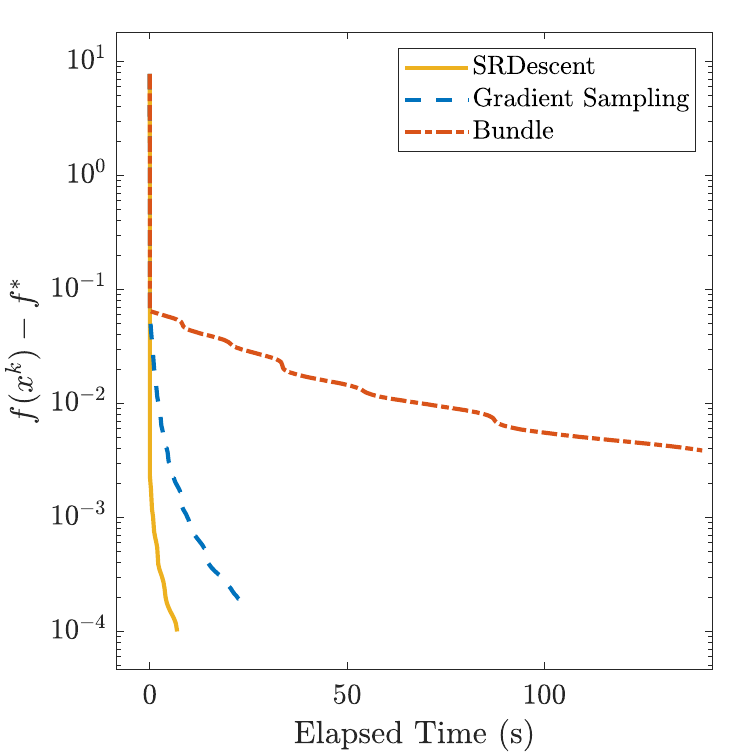}
    \end{minipage}
    \caption{\small Performance of SRDescent, gradient sampling~\cite{burke2005robust}, and a bundle method specially designed for difference-of-convex functions~\cite{de2019proximal} on Nesterov's nonsmooth Chebyshev-Rosenbrock function $f(x) = \frac{1}{4}(x_{1}- 1)^{2}+ \sum^{6}_{i=1} |x_{i+1}-2(x_{i})^2+1|$. The gradient sampling algorithm terminates due to the failure of line search.}
\label{fig:illustrate_performance}
\end{figure}

Interestingly, when  $f$ is a composition of a convex function with a smooth map,  the subgradient regularization reduces to the prox-linear update  (see Proposition \ref{prop:equivalent_prox-linear} in Section \ref{sec5.1:dual_prox-linear}). As a byproduct,  this connection offers a new dual perspective on the prox-linear method, previously understood primarily from a model-based viewpoint~\cite{drusvyatskiy2018error}.

Motivated by this relationship,  we further refine Algorithm \ref{alg1} by introducing a dynamic adjustment of the regularization parameter $\epsilon$. We show that this adaptive variant preserves subsequential convergence to stationary points for locally Lipschitz continuous functions. Moreover, similar to the prox-linear method, under regularity conditions,
the algorithm achieves local linear convergence when applied to the composition of a convex function with a smooth map.

\subsection{Notation and preliminaries}
For any vector $x\in \mathbb{R}^n$, we use $\|x\|$ to denote its Euclidean norm.
The open ball in $\R^{n}$ centered at $\bx \in \R^{n}$ with radius $\delta > 0$ is denoted by $\mathbb{B}_{\delta}(\bx) \triangleq \{x \in \R^{n}\mid \|x - \bx\| < \delta\}$, and its closure is $\overline{\B}_{\delta}(\bx)$. 
For nonempty closed sets $A , B \subset \R^{n}$, the one-sided deviation of $A$ from $B$ is defined as $\D (A, B) \,\triangleq\, \sup_{x \in A}\dist(x, B)$, where $\dist(x, B) \triangleq \inf_{y \in B}\| y - x \|$. The Hausdorff distance is defined as $\Hausd (A, B ) \triangleq \max\{\D(A, B), \,\D(B, A)\}$. We denote $\operatorname{cl}(A)$ and $\conv(A)$ as the closure and convex hull of a set $A$, respectively. The set of nonnegative integers is $\N$. We write $\N_{\infty}^{\sharp}\triangleq \{N \subset \N \mid N \mbox{ infinite}\}$ and $\N_{\infty}\triangleq \{ N \subset \N \mid \N\,\backslash N \text{ finite}\}$. Notation $\{t^{k}\}$ is used to simplify the expression of any sequence $\{ t^{k}\}_{k \in \N}$, where the elements can be points or sets. The convergence of a sequence $\{ t^{k}\}$ and a subsequence $\{t^{k}\}_{k \in N}$ to $t $ are written as $t^{k}\to t $ and $t^{k}\to_{N}t $. The Clarke directional derivative of a locally Lipschitz function $f$ at $x$ in the direction $d$ is $f^\circ(x; d) = \limsup_{y \to x,\, t \downarrow 0} \frac{f(y+td) - f(y)}{t}$, and recall~\cite[Proposition 2.1.2]{clarke1990optimization} that $f^\circ(x; d) = \MaxVal\{d^\top v \mid v \in \partial f(x)\}$.
For a sequence of sets $\{C^{k}\}$, the outer limit and inner limit are defined as
\[
\begin{array}{rl}
    \displaystyle\limsup_{k \rightarrow +\infty} \, C^k \triangleq & \{ u \mid \exists \, N \in \mathbb{N}_\infty^\sharp,\, u^k \to_N u \text{ with } u^k \in C^k \}, \\
    \displaystyle\liminf_{k \rightarrow +\infty} \; C^k \triangleq & \{ u \mid \exists \, N \in \mathbb{N}_\infty,\, u^k \to_N u \text{ with } u^k \in C^k \}.
\end{array}
\]
When the outer and inner limits coincide, the sequence $\{C^{k}\}$ is said to converge, and the set limit is denoted by $\lim_{k \to +\infty}C^{k}$. For a set-valued mapping $S: \R^{n}\rightrightarrows \R^{m}$, the outer and inner limits of $S$ at $\bx$ are
\[
    \limsup_{x \rightarrow \bar x} S(x) \triangleq \bigcup_{x^k \rightarrow \bar x}\limsup_{k \rightarrow +\infty}S(x^k), \qquad \liminf_{x
    \rightarrow \bar x}S(x) \triangleq \bigcap_{x^k \rightarrow \bar x}\liminf_{k \rightarrow +\infty}S(x^k).
\]
The mapping $S$ is \textsl{outer semicontinuous (osc)} at $\bar x \in \R^{n}$ if $\limsup_{x \rightarrow \bar x} S(x) \subset S(\bar x)$, equivalently, if $\limsup_{x \to \bx} S(x) = S(\bx)$. It is \textsl{continuous} at $\bx$ if $S(\bar x) = \limsup_{x \rightarrow \bar x} S(x) = \liminf_{x \rightarrow \bar x} S(x)$. The mapping $S$ is \textsl{locally bounded} at $\bx \in \R^{n}$ if there exists a neighborhood $V$ of $\bx$ such that $\bigcup \{S(x) \mid x \in V\}$ is bounded.

The trace of a square matrix $X$ is denoted by $\Tr(X)$. Let $\mathcal{S}_{m}$ be the space of $m \times m$ symmetric matrices and define the Frobenius scalar product of $X, Y \in \mathcal{S}_{m}$ as $\langle X, Y \rangle = \Tr(X Y)$.

\subsection{Outline of the paper}
The  rest of the paper is organized as follows. Section~\ref{sec2} introduces the abstract concept of descent-oriented subdifferentials and establishes its key properties. Section \ref{sec4:algorithm_vanilla} develops a general framework of descent methods based on the proposed descent-oriented subdifferentials and proves its asymptotic convergence. 
In Section~\ref{sec4:subgrad-regularization}, we present a new construction of descent-oriented subdifferentials for a class of marginal functions using subgradient regularization. 
In Section \ref{sec5:algorithm_adaptive}, we propose an adaptive variant of the previous algorithm, which reduces to the prox-linear method when minimizing the composition of a convex function with a smooth map, and establish its local linear convergence under regularity conditions.
Section \ref{sec6:nuericals} evaluates the performance of the proposed algorithms on several challenging nonsmooth optimization problems. The paper ends with a concluding section.

\section{Descent-Oriented Subdifferentials}
\label{sec2}

\subsection{Definition and existence of descent directions}
\label{sec2.1:Def_descent_oriented}

We begin by introducing the concept of descent-oriented subdifferentials, which not only unifies the ideas behind Goldstein-type and bundle-type methods, but also serves as a guiding principle for designing new descent algorithms with convergence guarantees in nonsmooth optimization.

\begin{definition}[Descent-oriented subdifferentials]
\label{def:G}
Let $f: \R^{n}\to \R$ be a locally Lipschitz continuous function. A set-valued map $G: \R^{n}\times (0, +\infty) \rightrightarrows \R^{n}$ is called a \textsl{descent-oriented subdifferential} for $f$ if, for each fixed  $\epsilon > 0$, the mapping $G(\cdot, \epsilon)$ is closed-valued and locally bounded, and the following two properties hold for all $\bx$:
    \begin{equation}\label{eq:Def_G1}
        \displaystyle \limsup_{\epsilon \downarrow 0,\, x \to \bx}\; G(x, \epsilon) \subset \partial f(\bx),
    \tag{G1}
    \end{equation}
    \begin{equation}\label{eq:Def_G2}
        \displaystyle \lim_{\epsilon \downarrow 0}\left(\limsup_{x \to \bx}G(x, \epsilon) \right) = \argmin_{v \in \partial f(\bx)} \|v\|.
    \tag{G2}
    \end{equation}
\end{definition}

\begin{remark}[Scaling]
\label{rmk:G_scaling}
    Suppose $G$ is a descent-oriented subdifferential for $f$. Then, for any scalar $\alpha$, the mapping $\alpha \, G$ is a descent-oriented subdifferential for $\alpha \, f$. This result follows directly from the fact that $\partial (\alpha f)(x) = \alpha \, \partial f(x)$ (see~\cite{clarke1975generalized}). In particular, $(-G)$ is a descent-oriented subdifferential for $(-f)$.
\end{remark}

Below we discuss the two required conditions of descent-oriented subdifferentials. Property \eqref{eq:Def_G1} ensures that under joint perturbation in $(x,\epsilon)$, the set $G(x, \epsilon)$ does not asymptotically introduce information beyond the Clarke subdifferential. This requirement is similar to the one-sided ``gradient consistency" of the smoothing method~\cite{chen2012smoothing}, in which $G(x, \epsilon)$ corresponds to the gradient of a smooth approximation for $f$ and plays an important role in establishing subsequential convergence to stationary points of smoothing methods. 
On the other hand, property \eqref{eq:Def_G2} restricts the sequential limit of $G(x, \epsilon)$ in $x$ and $\epsilon \downarrow 0$ to be the minimal norm subgradient. In particular, it implies that $\lim_{\epsilon \downarrow 0}G(\bx, \epsilon) = \argmin_{v \in \partial f(\bx)} \|v\|$, which guarantees the existence of descent directions for sufficiently small $\epsilon$, as we will show in Proposition \ref{prop:descent_direc}.
Furthermore, property \eqref{eq:Def_G2} is also closely related to the continuity of $G(x, \epsilon)$ with respect to $x$ for a fixed $\epsilon > 0$. A sufficient condition for \eqref{eq:Def_G2} is that, for all $\bx$, we have:
\begin{equation}\label{eq:Def_G2_sufficient}
    \left[\, \limsup_{x \to \bx} G(x, \epsilon) = G(\bx, \epsilon), \quad\forall\,\epsilon > 0 \,\right] \quad\text{and} \quad   \displaystyle\lim_{\epsilon \downarrow 0} G(\bx, \epsilon) = \argmin_{v \in \partial f(\bx)} \|v\|.
\tag{G2$^\prime$}
\end{equation}
The former limit  requires $G(\cdot, \epsilon)$ to be osc for every fixed $\epsilon > 0$. When $G(\cdot, \epsilon)$ is single-valued, outer semicontinuity together with local boundedness of $G(\cdot, \epsilon)$, as specified in Definition \ref{def:G}, is equivalent to continuity of $G(\cdot, \epsilon)$~\cite[Corollary 5.20]{rockafellar2009variational}.

\begin{remark}
\label{rmk:G_regularization_effect_of_epsilon}
    The minimal norm subgradient  may not be a descent-oriented subdifferential, since  property \eqref{eq:Def_G2} is not satisfied. Consider for example  $f(x) = \max\{-100, 2x_{1}+3x_{2}, -2x_{1}+3x_{2}, 5x_{1}+2x_{2}, -5x_{1} +2x_{2}\}$, which was studied in~\cite[Section 2.2]{hiriart1996convex} to demonstrate why the steepest descent method may fail to minimize nonsmooth functions. At $\bx = (0, 0)$, we have $\argmin_{v \in \partial f(\bx)} \|v\| = \{(0, 2)\}$. Along the path
    $x(t) = (t, 3t)$ with $t > 0$, $\argmin_{v \in \partial f(x(t))} \|v\| = \{(2, 3)\}$. Thus, the minimal norm subgradient violates \eqref{eq:Def_G2} because
    \[
        \limsup_{t \to 0} \, \argmin_{v \in \partial f(x(t))} \|v\| = \{(2, 3)\}
        \neq \{(0, 2)\} = \argmin_{v \in \partial f(\bx)} \|v\|.
    \]
\end{remark}

We now characterize stationary points using descent-oriented subdifferentials.
\begin{proposition}[Characterization of stationary points]
    Let $f: \R^n \to \R$ be a locally Lipschitz continuous function and $G$ be a descent-oriented subdifferential for $f$. Then a point $\bx$ is a stationary point of $f$ if and only if there exist sequences $\{\epsilon_k\} \downarrow 0$, $\{v^k\} \to 0$, and $\{x^k\} \to \bx$ such that $v^k \in G(x^k, \epsilon_k)$.
\end{proposition}
\begin{proof}
    The ``if" direction follows immediately from property \eqref{eq:Def_G1}. For the ``only if" direction, since $\bx$ is a stationary point, by property \eqref{eq:Def_G2}, we have
    \[
        \displaystyle \lim_{\epsilon \downarrow 0}\left(\limsup_{x \to \bx}G(x, \epsilon) \right) = \argmin_{v \in \partial f(\bx)} \|v\| = \{0\}.
    \]
    Hence, we can find $\{\epsilon_k\} \downarrow 0$ and $\{x^k\} \to \bx$ such that $0 \in \limsup_{k \to +\infty} G(x^k, \epsilon_k)$, completing the proof.
\end{proof}

Finally, we establish the existence of descent directions at non-stationary points using \eqref{eq:Def_G2}. This proposition will be central to the convergence analysis of our proposed algorithms.

\begin{proposition}[Existence of descent directions]
\label{prop:descent_direc}
    Let $f: \R^n \to \R$ be a locally Lipschitz continuous function and $G$ be a descent-oriented subdifferential for $f$. For any point $\bx$ satisfying $0 \notin \partial f(\bx)$ and any constant $\alpha \in (0,1)$, there exist positive scalars $\epsilon_{\bx}$ and $\eta_{\bx}$ such that for all $\epsilon \in (0, \epsilon_{\bx})$ and $\eta \in (0, \eta_{\bx})$, $0 \notin \limsup_{x \to \bx} G(x, \epsilon)$ and the following descent inequality holds:
    \begin{equation}\label{eq:descent_ineq}
        f(\bx - \eta g) \leq f(\bx) - \alpha \, \eta \|g\|^{2},
        \qquad\forall\, g \in \limsup_{x \to \bx}G(x, \epsilon).
    \end{equation}
\end{proposition}
\begin{remark}
    Noting that $G(\bx, \epsilon) \subset \limsup_{x \to \bx} G(x, \epsilon)$, the descent inequality implies that for sufficiently small $\epsilon > 0$, any element of $-G(\bx, \epsilon)$ is a descent direction.
\end{remark}

\begin{proof}
    The first statement $0 \notin \limsup_{x \to \bx} G(x, \epsilon)$ for all sufficiently small $\epsilon > 0$ is a direct consequence of property \eqref{eq:Def_G2} and our assumption $0 \notin \partial f(\bx)$. We only need to show \eqref{eq:descent_ineq}. Suppose for contradiction that $0 \notin \partial f(\bx)$ but there exist  $\{\epsilon_{k}\} \downarrow 0$, $\{\eta_{k}\} \downarrow 0$ and $g^k\in \limsup_{x \to \bx}G(x, \epsilon_{k})$ such that
    \begin{equation}\label{eq:violate_descent}
       f(\bx - \eta_{k} g^k) > f(\bx) - \alpha \, \eta_{k} \|g^k\|^{2},
       \qquad\forall\, k \in \N.
    \end{equation}
    Let $\bar d \triangleq \argmin_{v \in \partial f(\bx)} \|v\| \neq 0$. By property \eqref{eq:Def_G2}, it holds that $g^k\to \bar d$. Next, consider the Clarke directional derivative of $f$ at $\bx$ in the direction $-\bar d$:
\begin{equation}
\label{eq:lower_bound_Clarke_derivative}
\begin{array}{rl}
    f^\circ \big(\bx; -\bar d \,\big) 
    &= \displaystyle\limsup_{x \to \bx,\; t \downarrow 0} \frac{f\big(x - t \bar d \,\big) - f(x)}{t} 
    \geq \displaystyle\limsup_{d \to \bar d,\; t \downarrow 0} \frac{f\big(\bx - t d\big) - f(\bx)}{t}\\[0.15in]
    &\qquad\qquad\qquad\qquad\quad \geq \displaystyle\limsup_{k \to +\infty}\frac{f(\bx - \eta_{k} g^k) - f(\bx)}{\eta_{k}}
    \;\overset{\eqref{eq:violate_descent}}{\geq} - \alpha \left\|\bar d \,\right\|^{2},
\end{array}
\end{equation}
where the first inequality follows from the local Lipschitz continuity of $f$. To proceed, we use the optimality condition of $\bar d = \argmin_{v \in \partial f(\bx)} \|v\|^2$ to obtain $(\bar d - u)^{\top}\bar d \leq 0$ for all $u \in \partial f(\bx)$. Thus, $f^{\circ}(\bx; -\bar d\,) = \max_{u \in \partial f(\bx)} (-u^\top \bar d\,)$ attains its maximum at $u = \bar d$. Hence, by \eqref{eq:lower_bound_Clarke_derivative}, $-\|\bar d\|^{2} = f^{\circ}(\bx; -\bar d \,) \geq - \alpha \|\bar d \|^{2}$, which implies $\bar d = 0$ since $\alpha \in (0,1)$. This contradicts the assumption $0 \notin \partial f(\bx)$.
\end{proof}

\subsection{Construction of descent-oriented subdifferentials}
\label{sec2.2:Enlarged_subdiff}
In the first part of this subsection, we illustrate how descent-oriented subdifferentials can be constructed from  the minimal norm element of either Goldstein $\epsilon$-subdifferential or the $\epsilon$-subdifferential of convex functions. This perspective enables a unified interpretation of the iteration schemes in \eqref{eq:Goldstein_direc} and \eqref{eq:Bundle_direc} through the lens of descent-oriented subdifferentials.

\begin{proposition}
    Let $f:\mathbb{R}^n \to \mathbb{R}$ be locally Lipschitz continuous. Then the single-valued map $G: \R^{n}\times (0, +\infty) \rightarrow \R^{n}$ defined by $G(x, \epsilon) = \argmin_{v \in \partial^\text{G}_\epsilon f(x)} \|v\|$ is a descent-oriented subdifferential for $f$.
\end{proposition}
\begin{proof}
    For any fixed $\epsilon > 0$, the set $\partial^\text{G}_{\epsilon}f(x)$ is non-empty,  closed and convex for all $x$.  Thus, $G$ is a single-valued map on $\R^{n}\times (0, +\infty)$, and $G(\cdot, \epsilon)$ is closed-valued.
    Since $f$ is locally Lipschitz continuous, the set-valued map $\partial^\text{G}_{\epsilon}f(\cdot)$ is locally bounded for any $\epsilon > 0$. Therefore, its projection map, $G(\cdot, \epsilon)$ is also locally bounded.

    We then verify property \eqref{eq:Def_G1}. Take sequences $\{x^k\} \to \bx$ and $\{\epsilon_k\} \downarrow 0$ such that $G(x^k, \epsilon_k) \to \bar u$. We will show that $\bar u \in \partial f(\bx)$. By noticing $G(x^k, \epsilon_k) \in \partial^\text{G}_{\epsilon_k} f(x^k)$ and using Carath\'eodory's Theorem (see, e.g. \cite[Theorem 17.1]{rockafellar1970convex}), there exist $\lambda_{k,j} \geq 0$ and $y^{k,j} \in \overline{\B}_{\epsilon_k}(x^k)$ for $j=1, \cdots, n+1$ satisfying
    \[
        G(x^k, \epsilon_k) = \sum^{n+1}_{j=1} \lambda_{k,j} \, g^{k,j}
        \quad\text{with}\quad
        g^{k,j} \in \partial f(y^{k,j})
        \;\;\text{and}\;\;
        \sum^{n+1}_{j=1} \lambda_{k,j} = 1.
    \]
    For each $j$, observe that $\{\lambda_{k,j}\}_{k \in \N} \subset [0,1]$, and $\{g^{k,j}\}_{k \in \N}$ is a bounded sequence due to the local Lipschitz continuity of $f$. By passing to a subsequence if necessary, we may assume that $\lim_{k \to +\infty} \lambda_{k,j} = \bar\lambda_j$ and $\lim_{k \to +\infty} g^{k,j} = \bar g^j$ for all $j$. We also have $\bar g^j \in \partial f(\bx)$ by the outer semicontinuity of $\partial f$. Thus, the limit point $\bar u$ of the sequence $\{G(x^k, \epsilon_k)\}$ admits the representation
    \[
        \bar u = \sum^{n+1}_{j=1} \bar\lambda_{j} \, \bar g^{j}
        \,\in\, \conv\{\partial f(\bx)\} 
        = \partial f(\bx),\quad \mbox{where $ \displaystyle\sum^{n+1}_{j=1} \bar\lambda_j = 1$ and $\bar\lambda_j \geq 0$}.
    \]
    This establishes \eqref{eq:Def_G1}.
    
    Next, consider any sequence $\{\epsilon_{k}\} \downarrow 0$ and let $v^k\in \limsup_{x \to \bx}G(x, \epsilon_{k})$ for each $k \in \N$. To prove \eqref{eq:Def_G2}, we need to show that $v^k\to \argmin_{v \in \partial f(\bx)} \|v\|$. By definition, there exist sequences $\{x^{k,i}\}_{i \in \N}\to \bx$ and $\{v^{k,i}\}_{i \in \N} \to v^k$ with $v^{k,i} = \argmin_{v \in \partial^\text{G}_{\epsilon_k} f(x^{k,i})} \|v\|$ for all $i \in \N$. 
    For each $k$, select a sufficiently large index $i_{k}$ such that $\|x^{k,i_k}- \bx\| \leq \epsilon_{k}$ and $\|v^{k,i_k}- v^k\| \leq 1/k$. Thus,
    \begin{equation}\label{eq:Goldstein_inclusion}
        \partial f(\bx) \subset \partial^\text{G}_{\epsilon_k} f(x^{k,i_k}) \subset \partial^\text{G}_{2\epsilon_k}f(\bx),
        \qquad \forall\, k \in \N.
    \end{equation}
    Since the sequence of sets $\big\{\partial^\text{G}_{2\epsilon_k}f(\bx)\big\}$ is monotone, it follows from~\cite[Exercise 4.3]{rockafellar2009variational} that
    \begin{equation}\label{eq:Goldstein_convergence}
        \lim_{k \to +\infty}\partial^\text{G}_{2\epsilon_k} f(\bx)
        = \bigcap_{k \in \N}\operatorname{cl}\big(\partial^\text{G}_{2\epsilon_k} f(\bx)\big)
        = \bigcap_{k \in \N}\partial^\text{G}_{2\epsilon_k}f(\bx) = \partial f(\bx).
    \end{equation}
    Combining \eqref{eq:Goldstein_inclusion} and \eqref{eq:Goldstein_convergence}, we obtain $\partial^\text{G}_{\epsilon_k}f(x^{k,i_k}) \to \partial f (\bx)$. By the continuity of the projection onto closed convex sets~\cite[Proposition 4.9]{rockafellar2009variational}, we conclude that
    \[
        \lim_{k \to +\infty}v^{k,i_k}
        = \lim_{k \to +\infty} \,\argmin_{v \in \partial^\text{G}_{\epsilon_k} f(x^{k,i_k})} \|v\| 
        \;=\, \argmin_{v \in \partial f(\bx)} \|v\|.
    \]
    Since $\|v^{k,i_k}- v^k\| \leq 1/k$, we have $v^k \to \argmin_{v \in \partial f(\bx)} \|v\|$. This proves \eqref{eq:Def_G2}.
\end{proof}

We next turn to the $\epsilon$-subdifferential for convex functions. Asplund and Rockafellar~\cite{asplund1969gradients} proved that the set-valued map $(x, \epsilon) \mapsto \partial_{\epsilon} f(x)$ is jointly continuous on $\R^n \times (0,+\infty)$ in terms of the Hausdorff distance. Thus, fixing $\epsilon > 0$, the mapping $x \mapsto \partial_{\epsilon} f(x)$ is continuous. We will use this fact below.

\begin{proposition}
    Let $f: \R^{n}\to \R$ be a convex function. Then the single-valued map $G: \R^{n}\times (0, +\infty) \rightarrow \R^{n}$ defined by
   $G(x, \epsilon) = \argmin_{v \in \partial_\epsilon f(x)} \|v\|$
    is a descent-oriented subdifferential for $f$.
\end{proposition}
\begin{proof}
    For any fixed $\epsilon > 0$, the set $\partial_{\epsilon}f(x)$ is closed convex for all $x \in \R^{n}$. Hence, $G$ is a single-valued map, and $G(\cdot, \epsilon)$ is closed-valued.
    Also, for any $\epsilon > 0$, $G(\cdot, \epsilon)$ is locally bounded as the set-valued map $\partial_{\epsilon}f(\cdot)$ is locally bounded~\cite[Corollary 1]{asplund1969gradients}. 
    To establish \eqref{eq:Def_G1}, observe that for any $\bx$,
    \[
        \limsup_{\epsilon \downarrow 0, x \to \bx}G(x, \epsilon) 
        \;\subset\; \limsup_{\epsilon \downarrow 0, x \to \bx}\partial_{\epsilon}f(x) 
        \;\subset\; \partial f(\bx),
    \]
    where the second inclusion follows directly from the continuity of $f$ and the definition of $\partial_{\epsilon}f(x)$.
    We show property \eqref{eq:Def_G2} by verifying the sufficient condition \eqref{eq:Def_G2_sufficient}. Fix $\bx \in \R^{n}$ and consider any sequence $\{\epsilon_{k}\} \downarrow 0$. Since $\{\partial_{\epsilon_k}f(\bx)\}$ is a monotonic sequence of sets, we apply~\cite[Exercise 4.3]{rockafellar2009variational} to obtain the set convergence
    \[
        \lim_{k \to +\infty}\partial_{\epsilon_k}f(\bx) = \bigcap_{k \in \N}\operatorname{cl} \big(\partial_{\epsilon_k}f(\bx)\big) = \bigcap_{k \in \N}\partial_{\epsilon_k} f(\bx) = \partial f(\bx).
    \]
    Consequently, it follows from~\cite[Proposition 4.9]{rockafellar2009variational} that $\argmin_{v \in \partial_{\epsilon_k} f(\bx)} \|v\| \to \argmin_{v \in \partial f(\bx)} \|v\|$,
    establishing the second limit in \eqref{eq:Def_G2_sufficient}.
    
    To complete the proof, it remains to show that $G(\cdot, \epsilon)$ is osc for any fixed $\epsilon > 0$. By~\cite[Proposition 5]{asplund1969gradients}, we have $\lim_{x \to \bx}\Hausd(\partial_{\epsilon}f(x), \partial_{\epsilon}f( \bx)) = 0$ for any fixed $\bx$ and $\epsilon > 0$.
    Since $\partial_{\epsilon}f(\cdot)$ is locally bounded at $\bx$, it follows from~\cite[Corollary 5.21]{rockafellar2009variational} that this convergence in Hausdorff distance is equivalent to the continuity of the set-valued map $\partial_{\epsilon}f(\cdot)$ at $\bx$. Thus, for any sequence $x^k \to \bx$, we have $\partial_{\epsilon} f(x^k) \to \partial_{\epsilon} f(\bx)$ and the convergence of their minimal norm elements by using~\cite[Proposition 4.9]{rockafellar2009variational}:
    \[
        \lim_{k \to +\infty} G(x^k, \epsilon) 
        = \lim_{k \to +\infty} \,\argmin_{v \in \partial_{\epsilon} f(x^k)} \|v\| 
        \;=\, \argmin_{v \in \partial_{\epsilon}f(\bx)} \|v\|.
    \]
    Thus, \eqref{eq:Def_G2_sufficient} hold, and the proof is completed.
\end{proof}

We conclude this section by presenting an alternative construction of descent-oriented subdifferentials based on the Moreau envelope. 
A function $f: \R^n \to \R$ is $\rho$-weakly convex if $f + \frac{1}{2} \rho \|\cdot\|^2$ is a convex function. Given $\epsilon > 0$, the Moreau envelope and proximal mapping of $f$ are defined as $e_{\epsilon} f(x) \triangleq \inf_{z} \{f(z) + \frac{1}{2\epsilon} \|z-x\|^2 \}$ and $P_{\epsilon} f(x) \triangleq \argmin_{z} \{f(z) + \frac{1}{2\epsilon} \|z-x\|^2 \}$.

\begin{proposition}
    Let $f: \R^{n}\to \R$ be $\rho$-weakly convex and bounded from below. For all $\epsilon \in (0, \rho^{-1})$, the Moreau envelope $e_{\epsilon} f$ is continuously differentiable. Fix $\bar\epsilon \in (0, \rho^{-1})$. Then the single-valued map $G: \R^{n}\times (0, +\infty) \rightarrow \R^{n}$ defined by
    \begin{equation}\label{Moreau}
        G(x, \epsilon) = \left\{
        \begin{array}{cl}
        \nabla e_{\epsilon} f(x) &\text{ when } \epsilon \in (0, \bar\epsilon\,], \\
        \nabla e_{\bar\epsilon} f(x) &\text{ otherwise}
        \end{array}\right.
    \end{equation}
    is a descent-oriented subdifferential for $f$.
\end{proposition}
\begin{proof}
    Fix any $\bx \in \R^n$. Since $f + \frac{1}{2\rho} \|\cdot\|^2$ is convex, standard results about the Moreau envelope and proximal mapping for convex functions (see, e.g.,~\cite[Theorem 2.26]{rockafellar2009variational}) imply that for all $\epsilon \in (0, \rho^{-1})$, $P_{\epsilon}f$ is single-valued and continuous, and $e_{\epsilon} f$ is continuously differentiable with $\nabla e_{\epsilon} f(x) = \frac{1}{\epsilon} \big(x - P_{\epsilon}f(x)\big)$. 
    Thus, given $\bar\epsilon \in (0, \rho^{-1})$, the map $G$ in \eqref{Moreau} is well-defined. For each fixed $\epsilon \in (0, \bar\epsilon\,]$, $G(\cdot, \epsilon)$ is locally bounded due to the continuity of $P_{\epsilon} f$. By the optimality condition for the proximal mapping, $0 \in \partial f(P_{\epsilon} f(x)) + \frac{1}{\epsilon} \big(P_{\epsilon} f(x) - x\big)$, so $G(x, \epsilon) = \nabla e_{\epsilon} f(x) = \frac{1}{\epsilon} \big(x -  P_{\epsilon} f(x) \big) \in \partial f(P_{\epsilon} f(x))$ for $\epsilon \in (0, \bar\epsilon\,]$. Since $f(P_{\epsilon} f(x)) + \frac{1}{2\epsilon} \|P_{\epsilon} f(x) - x\|^2 \leq f(x)$ and $f$ is bounded from below, we obtain $\|P_{\epsilon} f(x) - x\| \leq \sqrt{2 \epsilon (f(x) - \inf f)}$, implying $\lim_{x \to \bx, \epsilon \downarrow 0} P_{\epsilon} f(x) = \bx$. By the outer semicontinuity of $\partial f$,
    \[
        \limsup_{x \to \bx, \epsilon \downarrow 0} G(x, \epsilon) \subset \limsup_{x \to \bx, \epsilon \downarrow 0} \partial f(P_{\epsilon} f(x)) \subset \partial f(\bx),
    \]
    establishing property \eqref{eq:Def_G1}. To proceed, we aim to verify \eqref{eq:Def_G2_sufficient}. The first condition $\limsup_{x \to \bx} G(x, \epsilon) = G(\bx, \epsilon)$ for $\epsilon > 0$ is a direct consequence of the continuity of $P_{\epsilon} f$. It remains to show that $\lim_{\epsilon \downarrow 0} G(\bx, \epsilon) = \argmin_{v \in \partial f(\bx)} \|v\|$. Let $\{\epsilon_k\} \downarrow 0$ with $\epsilon_k \in (0, \min\{\bar\epsilon, 1/(2 \rho)\})$. 
    We first show that the sequence $\{G(\bx, \epsilon_k)\}$ is bounded. Since $f + \frac{1}{2\epsilon_k} \| \cdot - \bx\|^2$ is $(1/{\epsilon_k} - \rho)$-strongly convex, its minimizer $P_{\epsilon_k} f(\bx)$ satisfies
    \begin{equation}\label{eq:strong_cvx}
        f(\bx) \geq \left(f(P_{\epsilon_k} f(\bx)) + \frac{1}{2\epsilon_k} \|\bx - P_{\epsilon_k} f(\bx)\|^2\right) + \frac{1}{2} \left(\frac{1}{\epsilon_k} - \rho\right) \|\bx - P_{\epsilon_k} f(\bx)\|^2.
    \end{equation}
    Additionally, since $f + \frac{1}{2} \rho \|\bx - \cdot\|^2$ is convex, for any $v \in \partial f(\bx)$,
    \begin{equation}\label{eq:cvx}
        f(P_{\epsilon_k}f(\bx)) + \frac{1}{2} \rho \|\bx - P_{\epsilon_k} f(\bx)\|^2
        \geq f(\bx) + v^\top \big(P_{\epsilon_k}f(\bx) - \bx\big).
    \end{equation}
    Combining \eqref{eq:strong_cvx} and \eqref{eq:cvx}, and using the inequality $\frac{1}{2} \big(\epsilon_k \|v\|^2 + \frac{1}{\epsilon_k} \|P_{\epsilon_k} f(\bx) - \bx\|^2\big) \geq -v^\top \big(P_{\epsilon_k}f(\bx) - \bx\big)$, we obtain that $\|v\|^2 \geq (1 - 2\epsilon_k \rho) \|P_{\epsilon_k}f(\bx) - \bx\|^2 / \epsilon_k^2$ for any $v \in \partial f(\bx)$. As $\epsilon_k < \min\{\bar\epsilon, 1/(2\rho)\}$ and $\epsilon_k \downarrow 0$, it follows that $\|G(\bx, \epsilon_k)\|^2 = \|P_{\epsilon_k} f(\bx) - \bx\|^2 / \epsilon_k^2 \leq \inf_{v \in \partial f(\bx)} \|v\|^2 / (1 - 2\epsilon_k \rho)$. Hence, $\{G(\bx, \epsilon_k)\}$ is bounded and $\limsup_{k \to +\infty} \|G(\bx, \epsilon_k)\| \leq \inf_{v \in \partial f(\bx)} \|v\|$.

    Next, recall that $\lim_{x \to \bx, \epsilon \downarrow 0} P_{\epsilon} f(x) = \bx$ and, thus, $P_{\epsilon_k} f(\bx) \to \bx$. By the outer semicontinuity of $\partial f$ and the fact that $G(\bx, \epsilon_k) \in \partial f(P_{\epsilon_k} f(\bx))$, every accumulation point of $\{G(\bx, \epsilon_k)\}$ must belong to $\partial f(\bx)$. Moreover, the previous bound $\limsup_{k \to +\infty} \|G(\bx, \epsilon_k)\| \leq \inf_{v \in \partial f(\bx)} \|v\|$ ensures that the only possible accumulation point is $\argmin_{v \in \partial f(\bx)} \|v\|$. Therefore, $G(\bx, \epsilon_k) \to \argmin_{v \in \partial f(\bx)} \|v\|$. This completes the proof.
\end{proof}

\section{A Descent-Oriented Subgradient Method}
\label{sec4:algorithm_vanilla}

Using the descent-oriented subdifferential as a first-order oracle, this section presents a general framework of descent algorithms for minimizing a locally Lipschitz continuous function $f$. Throughout this section, we assume that a descent-oriented subdifferential $G: \R^{n}\times (0, +\infty) \rightrightarrows \R^{n}$ for $f$ is given. The construction of such a map for structured nonsmooth functions will be discussed in the next section.

It follows from Proposition \ref{prop:descent_direc} that at any non-stationary point,  descent in the objective can be achieved by selecting a sufficiently small parameter $\epsilon$ along with a proper stepsize $\eta$. Our proposed algorithm progressively reduces  $\epsilon$ based on the norm of the descent-oriented subgradient at the current iterate, and determines  $\eta$ at a given $\epsilon$ via a line search procedure.  In general, a diminishing sequence of $\epsilon$ is needed to ensure subsequential convergence to stationary points, since property  \eqref{eq:Def_G1} only forces $G(x, \epsilon)$ approaches elements of $\partial f(x)$ when $\epsilon \downarrow 0$. For structured nonsmooth problems, we further propose an alternative strategy for adjusting $\epsilon$  in Section \ref{sec5:algorithm_adaptive}, which does not necessarily require $\epsilon$ to vanish asymptotically, yet still  guarantees asymptotic convergence to stationarity and achieves local linear convergence. 
We now formally present the descent-oriented subgradient method in Algorithm \ref{alg1}.

\begin{algorithm}[htp]
\caption{A basic descent-oriented subgradient method}
\begin{algorithmic}[1]
    \STATE {\bf Input:} 
    starting point $x_{0}$,
    regularization parameter $\epsilon_{0,0} > 0$,
    optimality tolerances $\epsilon_{\text{tol}}, \nu_{\text{tol}} \geq 0$, initial stationary target $\nu_{0}$, 
    reduction factors $\theta_{\epsilon}, \theta_{\nu} \in (0, 1)$, 
    and the Armijo parameter $\alpha \in (0,1)$
    \FOR{$k= 0, 1, \cdots $}
        \FOR{$i = 0, 1, \cdots $}
            \STATE Set $\epsilon_{k,i}= \epsilon_{k,0} \, 2^{-i}$ \hfill 
            \STATE Update direction $g^{k, i} \in G(x^k, \epsilon_{k,i})$
            \IF{$\epsilon_{k,i}\leq \epsilon_{\text{tol}}$ and $\|g^{k, i}\| \leq \nu_{\text{tol}}$}
                \STATE Set $i_{k}= i$ and {\bf return} $x^k$ \hfill $\triangleright$\textit{ approximate stationary}
            \ENDIF
            \FOR{$j = 0, 1, \cdots, i$}
                \STATE Set $\eta = \epsilon_{k,0} \, 2^{-j}$ \hfill $\triangleright$\textit{ line search}
                \IF{$f(x^k- \eta \, g^{k, i}) \leq f(x^k) - \alpha \, \eta \|g^{k, i}\|^{2}$}
                    \STATE Set $i_{k}= i$ and go to line 12
                \ENDIF
            \ENDFOR
        \ENDFOR
        \STATE Update $x^{k+1}= x^k- \eta_{k}\, g^{k, i}$ with $\eta_{k}= \eta$ 
        \IF{$\|g^{k, i_k}\| \leq \nu_{k}$}
            \STATE Set $\nu_{k+1} = \theta_{\nu} \cdot \nu_{k}$ and $\epsilon_{k+1, 0} = \theta_{\epsilon} \cdot \epsilon_{k, 0}$
        \ELSE
            \STATE Set $\nu_{k+1}= \nu_{k}$ and $\epsilon_{k+1,0}= \epsilon_{k,0}$ 
        \ENDIF
    \ENDFOR
\end{algorithmic}
\label{alg1}
\end{algorithm}

Below we establish the well-definedness and convergence of Algorithm \ref{alg1}.

\begin{proposition}[Finite termination of the inner loop $i$ in Algorithm \ref{alg1}]
\label{prop:inner_loop} 
   Let $k \in \N$ and fix any optimality tolerances $\epsilon_{\text{tol}}, \nu_{\text{tol}} \geq 0$. If $0 \notin \partial f(x^k)$, then the inner loop indexed by $i$ terminates in finitely many iterations, i.e., $i_k < +\infty$.
\end{proposition}
\begin{proof}
    By Proposition \ref{prop:descent_direc}, if $0 \notin \partial f(x^k)$, then there exist constants $\bar\epsilon, \bar\eta > 0$ such that for all $\epsilon \in (0, \bar\epsilon)$ and $\eta \in (0, \bar\eta)$, the following inequality holds:
    \[
        f(x^k- \eta \, g^k) \leq f(x^k) - \alpha \, \eta \|g^k\|^{2},
        \qquad\forall\, g^k\in G(x^k, \epsilon).
    \]
    Hence, either $x^k$ is certified as an approximate stationary point with $\epsilon_{k,i}\leq \epsilon_{\text{tol}}$ and $\|g^{k, i}\| \leq \nu_{\text{tol}}$, or the descent condition in line 10 is eventually satisfied for sufficiently large $i$. The result follows.
\end{proof}

\begin{remark}
    The above analysis  implies that, for any $k \in \N$, the inner loop would still have the finite termination if we simply used the stepsize $\eta = \epsilon_{k,i}$, rather than performing a backtracking line search over $\{\epsilon_{k,0}, \epsilon_{k,0}/2, \ldots, \epsilon_{k,i}\}$. We adopt the line search strategy in Algorithm~\ref{alg1} to enhance its practical performance.
\end{remark}

The following result was originally established by Kiwiel~\cite[Theorem 3.3]{kiwiel2007convergence} in the convergence analysis of the gradient sampling algorithm. It turns out to be useful in a more general setting, and we restate it here by emphasizing key consequences of a descent inequality, without specifying the update rule for the sequence $\{x^k\}$ and the choice of directions $\{g^k\}$.
\begin{lemma}
\label{lem:consequence_of_descent}
    Suppose that $f$ is bounded from below, and there are sequences $\{x^k\}$ and $\{g^k\}$ satisfying the descent inequality
    \[
        f(x^{k+1}) \leq f(x^k) - \alpha\|x^{k+1}- x^k\| \cdot \|g^k\|
        \quad\text{for some $\alpha \in (0,1)$}.
    \]
    If $\{x^k\}$ has an accumulation point $\bx$, then one of the following holds:\\
    (a) $\{x^k\}$ converges to $\bx$\\
    (b) $\liminf_{k \to +\infty}\max\{\|x^k- \bx\|, \|g^k\| \} = 0$.
\end{lemma}

We prove the subsequential convergence by using Proposition \ref{prop:descent_direc} and Lemma \ref{lem:consequence_of_descent}.

\begin{theorem}[Convergence of Algorithm \ref{alg1}]
\label{thm:convergence_alg1}
    Assume that $f$ is bounded from below. Let $\{x^k\}$ be a sequence generated by Algorithm \ref{alg1} with $\epsilon_{\text{tol}}= \nu_{\text{tol}}= 0$. Then, one of the following occurs:\\
    (a) For some $k \in \N$, the inner loop does not terminate. Consequently,
    $x^k$ is a stationary point.\\
    (b) Any accumulation point of $\{x^k\}$ is a stationary point.
\end{theorem}
\begin{proof}
    First, consider the case that the inner loop does not terminate for some $k \in \N$. By Proposition \ref{prop:inner_loop}, $0 \in \partial f(x^k)$, so (a) holds. Now suppose the inner loop terminates for each $k$. By Proposition \ref{prop:inner_loop}, there exists an index $i_{k}$ for each $k$ such that the descent inequality holds:
    \[
        f(x^{k+1}) = f(x^k- \eta_{k}\, g^{k, i_k}) \leq f(x^k) - \alpha \, \eta_{k}
        \|g^{k, i_k}\|^{2}.
    \]
    Since $f$ is bounded from below, summing the descent inequality over $k$ gives
    \begin{equation}
    \label{eq:finite_sum1}
        \sum^{+\infty}_{k=0}\eta_{k}\|g^{k, i_k}\|^{2} = \sum^{+\infty}_{k=0}\|x^{k+1}- x^k\| \cdot \|g^{k, i_k}\| < +\infty .
    \end{equation}

    \noindent
    {\bf Case 1:} Suppose that there exist $\underline k \in \mathbb{N}$, $\underline\nu > 0$ and $\underline\epsilon > 0$ such that $\nu_{k}= \underline\nu$ and $\epsilon_{k,0}= \underline\epsilon$ for all $k \geq \underline k$. Then, $\|g^{k, i_k}\| > \underline\nu$ for all $k \geq \underline k$. By \eqref{eq:finite_sum1}, it follows that $\eta_{k}\to 0$ and $\sum^{+\infty}_{k=0} \|x^{k+1}- x^k\| < +\infty$. Since $\eta_{k}\geq \epsilon_{k,0} \, 2^{-i_k} \geq \underline\epsilon \, 2^{-i_k}$, the condition $\eta_{k} \to 0$ implies $i_{k}\to +\infty$. Additionally, from $\sum^{+\infty}_{k=0} \|x^{k+1}- x^k\| < +\infty$, we deduce that $\{x^k\}$ converges to some point $\bx$. If $0 \in \partial f(\bx)$, the proof is done; therefore, assume that $0 \notin \partial f(\bx)$.

    For each fixed index $i \in \N$, since $i_{k}\to +\infty$, there exists $k_{i}> \underline k$ such that $i \in \{0, \cdots, i_{k}\}$ for all $k \geq k_{i}$. By definition of $G$, the set-valued map $G(\cdot, \underline\epsilon \, 2^{-i})$ is locally bounded at $\bx$, so $\{g^{k, i}\}_{k \geq k_i}$ is a bounded sequence. Let $\{g^{k, i}\}_{k \in N_i}$ be a subsequence convergent to some point $\bar d^{\,i}$. Thus,
    \begin{equation}\label{eq:d_i}
        \bar d^{\,i}\in \limsup_{k \to +\infty}G(x^k, \epsilon_{k,i}) = \limsup_{k \to +\infty}G(x^k, \underline\epsilon  \, 2^{-i}) 
        \qquad\text{for all }i \in \N.
    \end{equation}
    For each $k$ with $i_k \geq 1$, due to the termination criterion of the inner loop, the descent inequality fails up to index $i = i_{k}$, i.e., for all $i \in \{0, \cdots, i_{k}-1\}$,
    \begin{equation}\label{eq:failure_descent}
        f\big(x^k- \epsilon_{k,0} \, 2^{-j} \cdot g^{k, i}\big) > f(x^k) - \epsilon_{k,0}\, 2^{-j} \cdot \alpha \|g^{k, i}\|^{2}
        \qquad \text{for all }j \in \{0, \cdots, i\}.
    \end{equation}
    By choosing $j=i$, passing to the limit as $k(\in N_i) \to +\infty$ in \eqref{eq:failure_descent}, and using the continuity of $f$, we obtain
    \begin{equation}\label{eq:lim_k_failure_descent}
        f\left(\bx - \underline\epsilon \, 2^{-i} \cdot \bar d^{\,i}\right) \geq f(\bx) - \underline\epsilon \, 2^{-i} \cdot \alpha \big\|\bar d^{\,i}\big \|^{2}
        \qquad \text{for all }i \in \mathbb{N}.
    \end{equation}
    Combining \eqref{eq:lim_k_failure_descent} and \eqref{eq:d_i}, we conclude that the descent inequality \eqref{eq:descent_ineq} fails to hold with $\epsilon = \eta = \underline\epsilon \, 2^{-i}$ for all $i \in \N$. This contradicts Proposition \ref{prop:descent_direc}, which guarantees the existence of descent directions for sufficiently small $\epsilon$ and $\eta$. Consequently, in Case 1, $\{x^k\}$ converges to a point $\bar{x}$ satisfying $0 \in \partial f(\bar{x})$.

    \vspace{0.1in}
    \noindent
    {\bf Case 2:} Otherwise, the reduction in line 13 of Algorithm \ref{alg1} occurs infinitely many times, and we have $(\nu_{k}, \epsilon_{k,0}) \downarrow 0$. Let $\bx$ be an accumulation point of $\{x^k\}$. Since $\{x^k\}$ and $\{g^{k, i_k}\}$ satisfy the descent inequality
    \[
        f(x^{k+1}) \leq f(x^k) - \alpha \, \eta_{k} \|g^{k, i_k}\|^{2}= f(x^k
        ) - \alpha \|x^{k+1}- x^k\| \cdot \|g^{k, i_k}\|,
    \]
    we can apply Lemma \ref{lem:consequence_of_descent} to obtain that either $x^k \to \bx$, or $\liminf_{k \to +\infty}\, \max\{\|x^k- \bx\|, \|g^{k, i_k}\|\} = 0$. 
    Since $\nu_k \downarrow 0$, it always holds that $\liminf_{k \to +\infty} \|g^{k,i_k}\| = 0$. Therefore, in either scenario, we can find an index set $N \in \N^{\sharp}_{\infty}$ such that $\max\big\{\|x^k- \bx\|, \|g^{k, i_k}\|\big\} \to_{N} 0$. 
    As $\epsilon_{k, i_k}\leq \epsilon_{k,0}\downarrow 0$, by property \eqref{eq:Def_G1}, we have
    \[
        0 = \lim_{k(\in N) \to +\infty}g^{k, i_k}\,\in \left[\,\limsup_{k
        \to +\infty}G(x^k, \epsilon_{k,i_k})\right] \subset \partial f(\bx).
    \]
    Hence, any accumulation point $\bx$ of $\{x^k\}$ is a stationary point.
\end{proof}

\section{Descent Directions via Subgradient Regularization}
\label{sec4:subgrad-regularization}

This section introduces the technique of subgradient regularization for constructing descent-oriented subdifferentials. In contrast to the Goldstein-type and bundle-type constructions discussed in Section \ref{sec2.2:Enlarged_subdiff}, which rely on enlarged subdifferentials without exploring specific structures, we now focus on a subclass of nonsmooth functions known as marginal functions. For simplicity, we restrict attention to maximal marginal functions. Analogous arguments apply when the maximum is replaced by a minimum, i.e., when the maximal value function $f$ is replaced by $(-f)$ (see Remark \ref{rmk:G_scaling} following the definition of descent-oriented subdifferentials).

\subsection{Subgradient regularization for marginal functions with fixed feasible sets}
\label{sec4.1:subgrad-regularized_marginal}

We start by analyzing the case where the feasible set of the marginal function is fixed, i.e., independent of $x$. Specifically, we consider the marginal function
\begin{equation}\label{eq:P}
    f(x) \triangleq \max_{y \in Y \subset \R^{m}}\; \varphi(x,y),
    \tag{P}
\end{equation}
where  $Y$ is convex and compact; $\varphi (\cdot, y)$ is differentiable for every fixed $y \in Y$; both $\varphi (\cdot, \cdot)$ and $\nabla_{x}\varphi (\cdot, \cdot)$ are jointly continuous; and $\varphi(x, \cdot)$ is concave for every fixed $x$. Under these assumptions, $f$ is locally Lipschitz continuous~\cite[Theorem 10.31]{rockafellar2009variational}.

Let $Y^{\ast}(x)$ denote the solution set of the inner maximization problem for a given $x$, which must be convex. By Danskin's theorem~\cite[Theorem 2.1]{clarke1975generalized}, we have
\begin{equation}\label{eq:Clarke_subdiff}
    \partial f(x) = \conv\big(S(x)\big) 
    \quad\text{with}\quad S(x) \triangleq \bigcup \big\{\nabla_{x}\varphi(x, y) \mid y \in Y^{\ast}(x)\big\}.
\end{equation}
The convex hull is redundant when $\nabla_{x}\varphi(x, y)$ is affine in $y$.
For any fixed $x$ and $\epsilon > 0$, we define the subgradient-regularized problem associated with \eqref{eq:P} as:
\begin{equation}\label{eq:P_regularized}
    Y^{\epsilon}(x) \triangleq \argmax_{y \in Y} \left\{
        \varphi (x,y) - \frac{\epsilon}{2}\|\nabla_{x}\varphi (x,y)\|^{2}\right
    \}.
    \tag{P$^\epsilon$}
\end{equation}
Throughout this subsection, we consider a set-valued map $G: \R^n \times (0, +\infty) \rightrightarrows \R^n$ defined by
\begin{equation}\label{eq:G_max_marginal}
    G(x, \epsilon) \triangleq \bigcup \big\{\nabla_{x}\varphi (x, y) \mid y \in Y^{\epsilon}(x)\big\}.
\end{equation}

We note that  when $\varphi(x, y) = y^\top A x - h(y)$ for a convex function $h$ and a matrix $A \in \R^{m \times n}$, the above construction of $G(x,\epsilon)$ is different from the well-known smoothing technique by Nesterov~\cite{nesterov2005smooth}, where  the smoothed approximation of $f$ in \eqref{eq:P} is defined as $f_{\epsilon}(x) \triangleq \max_{y \in Y}\{\varphi(x, y) - \epsilon d(y)\}$ for some strongly convex function $d$. At a given $x$, the optimal solution of $y$ is always unique and $f_\epsilon$ is continuously differentiable. In contrast, $Y^\epsilon(x)$ in \eqref{eq:P_regularized} may not be a singleton, but $G$ is always a descent-oriented subdifferential for $f$ in this case, as shown in the following lemma. The proof can be found in Appendix A.1.

\begin{lemma}[Subgradient regularization defines a descent-oriented subdifferential]
\label{lem:G_max_marginal}
    Consider problem \eqref{eq:P}, and let $Y^\epsilon(x)$ and $G(x, \epsilon)$ be defined as in \eqref{eq:P_regularized} and \eqref{eq:G_max_marginal}, respectively. Let $y^{\epsilon}: \R^n \to \R^m$ be a single-valued map satisfying $y^{\epsilon}(x) \in Y^{\epsilon}(x)$ for all $(x, \epsilon) \in \R^{n}\times (0, +\infty)$. Then for any fixed $\bx \in \R^{n}$, the following properties hold:\\[0.03in]
    (a) $\displaystyle\lim_{\epsilon \downarrow 0, x \to \bx}\varphi (x, y^{\epsilon}(x)) = f(\bx)$.\\[0.03in]
    (b) $\displaystyle\limsup_{\epsilon \downarrow 0, x \to \bx}\nabla_{x}\varphi (x, y^{\epsilon} (x)) \subset \partial f(\bx)$. Consequently, the map $G$ satisfies \eqref{eq:Def_G1}.\\[0.03in]
    (c) For any  $\epsilon > 0$, $\limsup_{x \to \bx}\nabla_{x}\varphi(x, y^{\epsilon}(x)) \subset \bigcup\big\{\nabla_{x}\varphi (\bx, y) \mid y \in Y^{\epsilon}(\bx)\big\}$.
    Consequently, $G$ satisfies the first condition in \eqref{eq:Def_G2_sufficient}.\\[0.03in]
    (d) If $S(\bx)$ is a convex set, then $\lim_{\epsilon \downarrow 0}\nabla_{x}\varphi (\bx, y^{\epsilon}(\bx)) = \argmin_{v \in \partial f(\bx)} \|v\|$, and $G$ also satisfies the second condition in \eqref{eq:Def_G2_sufficient}. Thus, $G$ is a descent-oriented subdifferential for $f$.
\end{lemma}

It can be seen from the above lemma that when $\nabla_{x}\varphi (x, y)$ is affine in $y$ for any fixed $x$, the set $S(x)$ is convex for any $x$, and the mapping $G$ defined in \eqref{eq:G_max_marginal} is a descent-oriented subdifferential. In this case, the subgradient-regularized problem \eqref{eq:P_regularized} reduces to a concave maximization over a convex compact set and can  be solved to optimal in principle.
We also highlight two useful properties in this setting. Part (a) shows that, although $Y^{\epsilon}(x)$ may not be a singleton, the gradient $\nabla_{x}\varphi (x, y)$ remains invariant over $y \in Y^{\epsilon}(x)$.

\begin{lemma}[A special case of affine dependence in $y$]
\label{lem:subgrad-regularized_Affine_in_y} 
    Consider problem \eqref{eq:P} with $\varphi (x, y) = g(x) + y^{\top}\psi(x) - h(y)$, where $g: \R^{n}\to \R$ and $\psi: \R^{n}\to \R^{m}$ are continuously differentiable maps, and $h: \R^{m} \to \R$ is convex. Let $G$ be defined as in \eqref{eq:G_max_marginal}. Then the following properties hold:\\
    (a) For any $(x, \epsilon) \in \R^n \times (0, +\infty)$, the set $G(x, \epsilon)$ is a singleton.\\
    (b) Let $\{x^k\} \subset \R^n$ and $\{\epsilon_{k}\} \subset (0, +\infty)$ be sequences such that $\max\{\|x^k - \bx\|, \|G(x^k, \epsilon_{k})\| \} \to_N 0$ for some point $\bx$ and some index set $N \in \N^{\sharp}_{\infty}$. Then $\bx$ is a stationary point.
\end{lemma}

\begin{remark}
    Part (b) implies that, under the specific structure of $\varphi$, if an algorithm produces a sequence $\{x^k\}$ with an accumulation point $\bx$ satisfying the condition in (b), then $\bx$ is guaranteed to be a stationary point without the need for $\epsilon_k \to 0$. In contrast, the general convergence analysis of Algorithm~\ref{alg1} in Section~\ref{sec4:algorithm_vanilla} requires property (G1), which connects $G(x, \epsilon)$ to $\partial f(x)$ only asymptotically as $\epsilon \downarrow 0$.
\end{remark}

\begin{proof}
    (a) This result can be easily derived using the strong convexity of the squared function $\|\cdot\|^2$ and the linearity of $\nabla_x \varphi(x,y)$ in $y$. A detailed proof can be found in (e.g.,~\cite{tseng1991descent,luo1992linear}) and is omitted here for brevity. 
    (b) We let $\nabla \psi(x) \in \R^{n \times m}$ be the transpose of the Jacobian matrix of $\psi$ at $x$. For any $y^{k} \in Y^{\epsilon_k}(x^k)$, by the optimality condition for problem \eqref{eq:P_regularized}, we have
    \begin{equation}\label{eq:optimality_y}
        0 \in -\psi(x^k) + \partial h(y^k) + \epsilon_{k}\nabla \psi(x^k)^\top \nabla_{x}\varphi (x^k, y^{k}) + \mathcal{N}_{Y} (y^{k}).
    \end{equation}
    where $\mathcal{N}_Y(\by) \triangleq \{v \in \R^m \mid v^\top (y - \by) \leq 0 \text{ for all }y \in Y\}$ is the normal cone of $Y$ at a point $\by \in Y$. Since $Y$ is compact, by further passing to a subsequence if necessary, we assume that $y^{k} \to_{N} \bar y \in Y$. It follows from our assumption that $\|G(x^k, \epsilon_k)\| = \|\nabla_{x}\varphi (x^k, y^{k})\| \to_{N} 0$. By the continuity of $\nabla_x \varphi(\cdot, \cdot)$, we also have $\nabla_x \varphi(\bx, \by) = 0$. Taking the outer limit in \eqref{eq:optimality_y} as $k \to +\infty$, and using the continuity of $\psi(\cdot), \nabla \psi(\cdot)$, along with the outer semicontinuity of $\partial h(\cdot)$ and $\mathcal{N}_{Y}(\cdot)$ (see~\cite[Propositions 6.6 and 8.7]{rockafellar2009variational}), we obtain $0 \in [-\psi(\bx) + \partial h(\by)+ \mathcal{N}_{Y}(\by)] = \partial_{y}(-\varphi) (\bx, \bar y) + \mathcal{N}_{Y}(\by)$.
    This implies that $\by \in Y^{\ast}(\bx)$. Since $\nabla_x \varphi(\bx, \by) = 0$, we conclude that
    \[
        0 = \nabla_{x}\varphi (\bx, \by) 
        \in \Big[\bigcup\big\{\nabla_{x}\varphi (\bx, y) \mid y \in Y^\ast(\bx)\big\}\Big] 
        \subset \partial f(\bx),
    \]
    and hence $\bx$ is a stationary point.
\end{proof}

To illustrate the broad applicability of the subgradient-regularization technique, we next present four examples that fall within the formulation of \eqref{eq:P} and satisfy the structural assumption in Lemma \ref{lem:subgrad-regularized_Affine_in_y}.

\begin{example}[Finite max of smooth functions]
\label{ex:max-of-smooth}
    Let $\{f_{i}\}_{i=1}^m$ be continuously differentiable. Then
    \begin{equation}\label{eq:Max_of_smooth}
        f(x) = \max_{1 \leq i \leq m}f_{i}(x) = \max_{y \in \Delta^{m}}\sum^{m}_{i=1}y_{i}f_{i}(x),
    \end{equation}
    satisfies the assumption in Lemma \ref{lem:subgrad-regularized_Affine_in_y}.
\end{example}

\begin{example}[Maximal eigenvalue functions]
\label{ex:eigen-value}
    Consider the  function
    \[
        f(x) = \lambda_{\max}(\mathcal{A}(x)) ,
    \]
    where $\mathcal{A}: \R^n \to \mathcal{S}_m$ is an affine map of the form $\mathcal{A}(x) = A_{0}+ \mathcal{F}x$ with $\mathcal{F}x = \sum^n_{i=1} x_i A_i$ defined by matrices $A_{i} \in \mathcal{S}_{m}$ for $i=0,1,\cdots,n$. Here, $\lambda_{\max}(X)$ denotes the largest eigenvalue of the matrix $X \in \mathcal{S}_{m}$. The function $\lambda_{\max}(\cdot)$ is convex and can be expressed as the support function of the \textsl{spectrahedron} set $\mathcal{C}_{m}\triangleq \{Q \in \mathcal{S}_{m}\mid Q \succeq 0, \Tr(Q ) = 1\}$, which is compact and convex. Using Rayleigh's variational formulation, we can rewrite the objective as a marginal function:
    \[
        \lambda_{\max}(\psi(x)) = \max_{u \in \R^m, \|u\| = 1}\, u^{\top}\mathcal{A}(x) u \;=\; \max_{y \in \mathcal{C}_m}\, \langle y, \mathcal{A}(x) \rangle .
    \]
\end{example}

\begin{example}[Distributionally robust optimization]
\label{ex:DRO}
    Let $\{s_i\}^{N}_{i=1}$ be observed data and $\ell(\cdot\,; s_i)$ be the corresponding loss. 
    The distributionally robust formulation of the training problem   based on $\phi$-divergence~\cite{levy2020large} takes the form of
    \[
        f(x) = \max_{\left\{p \in \Delta^{N} \,\middle|\,  \sum^{N}_{i=1} \phi(N p_i) / N \,\leq\, \rho\right\}} \sum^{N}_{i=1} \left(p_{i} \, \ell(x; s_i) - \frac{1}{N} \psi(N p_i) \right),
        \qquad\rho > 0,
    \]
    where $\phi, \psi: [0, +\infty) \to \R \cup \{+\infty\}$ are convex functions satisfying $\phi(1) = \psi(1) = 0$. If $\ell(\cdot\,; s)$ is continuously differentiable, the assumptions required in Lemma \ref{lem:subgrad-regularized_Affine_in_y} hold.
\end{example}

\begin{example}[Composition of a convex function with a smooth map]
\label{ex:cvx_composite}
    Consider the composite function $f(x) = h(c(x))$, where $c: \R^{n}\to \R^{m}$ is $C^{1}$-smooth and $h: \R^{m}\to \R$ is a convex function. By Fenchel duality, we write $f$ as $f(x) = \max_{y \in \R^m}\left\{y^{\top}c(x) - h^{\ast}(y)\right\}$.
    For any fixed $x$, we denote by $\nabla c(x) \in \R^{n \times m}$ the transpose of the Jacobian matrix of $c$ at $x$. Then $\partial f(x) = \bigcup\left\{\nabla c(x) \, y \,\middle|\, y \in Y^{\ast}(x)\right\}$ where $Y^{\ast}(x)$ denotes the solution set of the above maximization problem. We construct a descent-oriented subdifferential via subgradient regularization:
    \begin{equation}\label{eq:G_cvx-composite}
        G(x, \epsilon) = \bigcup\big\{\nabla c(x) \, y \mid y \in Y^{\epsilon}(x) \big\},
    \end{equation}
    where
    \[
        Y^{\epsilon}(x) = \argmax_{y \in \dom(h^\ast)}\left\{
            y^{\top} c(x) - h^{\ast}(x) -\frac{\epsilon}{2}\|\nabla c(x) \, y\|^{2}
        \right\}.
    \]
\end{example}

\subsection{Subgradient regularization for marginal functions with varying feasible sets}
\label{sec4.2:marginal_vary}

We turn to the case in which the feasible set of the marginal function also depends on $x$. Consider the following marginal function:
\begin{align*}
    f(x) 
    &\displaystyle\triangleq \max_{y \in \R^m}\Big\{ \varphi_0 (x,y), 
    \quad\mbox{subject to } \varphi_j(x,y) \leq 0,\; j=1,\cdots,r \Big\}
    \tag{P$_x$} \label{eq:P_x} \\
    &\displaystyle= \max_{y \in \R^m} \min_{\lambda \geq 0} \Bigg[\mathcal{L}(x, y; \lambda) \triangleq \varphi_0(x, y) - \sum^{r}_{j=1}\lambda_j \, \varphi_j(x,y) \Bigg].
\end{align*}
Here, all functions $\{\varphi_j(x, y)\}^{r}_{j=0}$ are continuously differentiable.
Given $(x,y)$, we define the index set of active constraints as $J(x, y) \triangleq \{j \in \{1,\cdots, r\} \mid \varphi_j(x,y) = 0\}$. For the inner maximization in \eqref{eq:P_x}, denote the solution set by $Y^{\ast}(x)$.
We summarize basic assumptions below:
\begin{itemize}
    \item[(i)] For every $x$, the function $\varphi_0(x, \cdot)$ is concave, and $\{\varphi_j(x, \cdot)\}^{r}_{j=1}$ are convex.
    \item[(ii)] For every $x$, the feasible set $Y(x) \triangleq \{y \in \R^m \mid \varphi_j(x,y) \leq 0, j=1, \cdots, r\}$ is nonempty, and the mapping $x \mapsto Y(x)$ is locally bounded.
    \item[(iii)] The Mangasarian-Fromovitz constraint qualification (MFCQ) holds at every $x$: 
    For any $y \in Y^\ast(x)$, there exists $w \in \mathbb{R}^{m}$ such that $\nabla_y\, \varphi_{j} (x, y)^\top w < 0$ for all $j \in J(x, y)$.
\end{itemize}
Under assumption (i), the set of Lagrange multipliers associated with the inner maximization in~\eqref{eq:P_x} is the same for all $y \in Y^\ast(x)$. Thus, we can denote this set as $\Lambda(x)$. Assumption (iii) ensures that $\Lambda(x)$ is nonempty and bounded for any fixed $x$~\cite[Proposition 5.47]{bonnans2013perturbation}. Furthermore, by~\cite[Theorem 5.1]{gauvin1982differential}, assumptions (ii) and (iii) imply the local Lipschitz continuity of $f$.

We now present two scenarios, each giving rise to a distinct form of subgradient regularization based on the characterization of $\partial f(x)$.

\vspace{0.1in}
\noindent{\bf Case 1:} 
We strengthen assumption (iii) by requiring uniqueness of multipliers. Specifically, we replace the MFCQ with:
\begin{itemize}
    \item[(iii$^\prime$)] The linear independence constraint qualification (LICQ) holds at every $x$: 
    For any $y \in Y^\ast(x)$, the gradients $\{\nabla_{y}\,\varphi_j(x,y)\}_{j \in J(x, y)}$ are linearly independent.
\end{itemize}
Under (iii$^\prime$),  it follows from~\cite[Corollary 5.4]{gauvin1982differential} that 
\[
    \partial f(x) = \conv \left(\bigcup \big\{ \nabla_{x}\mathcal{L}(x, y; \lambda(x)) \mid y \in Y^{\ast}(x) \big\} \right)
    \quad\text{for }\lambda(x) = \Lambda(x).
\]
If the convex hull operation is superfluous (for instance, when $\{\nabla_{x}\varphi_j(x, \cdot)\}_{j=0}^r$ are affine), and if an oracle returns the unique multiplier $\lambda(x)$ for each $x$, we consider the following subgradient-regularized problem:
\[
\begin{array}{rl}
    Y^{\epsilon}(x) \triangleq 
    &\;\,\displaystyle\argmax_{y \in \R^m} \;\;\left\{ \varphi_0(x,y) - \frac{\epsilon}{2} \big\|\nabla_{x}\mathcal{L} (x,y;\lambda(x)) \big\|^{2}\right\} \\ 
    &\mbox{subject to }\; \varphi_j(x,y) \leq 0,\quad j=1,\cdots,r. 
\end{array}
\]
Correspondingly, we define the associated descent-oriented subdifferential as
\begin{equation}\label{eq:G_LICQ}
    G(x, \epsilon) \triangleq \bigcup\big\{\nabla_{x} \mathcal{L} (x,y; \lambda(x)) \mid y \in Y^{\epsilon}(x)\big\}.
\end{equation}

\vspace{0.1in}
\noindent{\bf Case 2:} 
Instead of requiring the uniqueness of multipliers as in Case 1, we ensure that $Y^\ast(x)$ is a singleton and the map $x \mapsto Y^\ast(x)$ is locally Lipschitz continuous. For simplicity, suppose additionally that $\varphi(x,\cdot)$ is strongly concave for every fixed $x$. This assumption can be relaxed to strong second-order sufficient optimality conditions (see~\cite[(5.120)]{bonnans2013perturbation}). Under this condition, by~\cite[Corollary 4]{mordukhovich2009subgradients}, we have
\[
    \partial f(x) = \bigcup \big\{\nabla_{x}\mathcal{L} (x, y(x); \lambda) \mid \lambda \in \Lambda(x)\big\}
    \quad\text{for } y(x) = Y^{\ast}(x).
\]
If an oracle returns the unique optimal solution $y(x) = Y^{\ast}(x)$ for each $x$, we define the following subgradient-regularized problem:
\[
\begin{array}{rl}
    \Lambda^{\epsilon}(x) \triangleq
    &\;\,\displaystyle\argmin_{\lambda \geq 0} \;\;\,\left\{ \mathcal{L}(x,y(x); \lambda) + \frac{\epsilon}{2} \big\|\nabla_{x}\mathcal{L}(x,y(x);\lambda) \big\|^{2}\right\} \\
    &\mbox{subject to } \;\nabla_y \mathcal{L}(x, y(x); \lambda) = 0  .
\end{array}
\]
We then construct the associated descent-oriented subdifferential as
\begin{equation}\label{eq:G_MFCQ}
    G(x, \epsilon) \triangleq \bigcup \big\{\nabla_{x}\mathcal{L} (x,y(x);\lambda) \mid \lambda \in \Lambda^\epsilon(x)\big\}.
\end{equation}

We conclude this section by presenting the main results for the two scenarios described above. The proof follows similar arguments as those employed in Lemma \ref{lem:G_max_marginal}; for completeness, we provide a concise outline of the proof in Appendix A.2.

\begin{lemma}[Subgradient regularization defines a descent-oriented subdifferential]
\label{lem:G_opt_val}
    Consider problem \eqref{eq:P_x} under assumptions (i)-(iii).\\
    (a) If assumption (iii$^\prime$) holds, and $\{\nabla_x \,\varphi_j(x, y)\}^{r}_{j=0}$ are affine in $y$ for any fixed $x$, then the mapping $G$ defined by \eqref{eq:G_LICQ} is a descent-oriented subdifferential for $f$.\\
    (b) If $\varphi_0(x, \cdot)$ is strongly concave for any fixed $x$, then the mapping $G$ defined by \eqref{eq:G_MFCQ} is a descent-oriented subdifferential for $f$.
\end{lemma}

\section{Adaptive Stepsizes and Reduction to the Prox-Linear Method}
\label{sec5:algorithm_adaptive}

The focus of last section was on the choice of descent directions.
In this section, we shift attention to stepsize control and refine Algorithm \ref{alg1} by further exploiting problem structure. We recall the nonsmooth composite problem in Example \ref{ex:cvx_composite}:
\begin{equation}\label{eq:cvx_composite}
    \Min_{x \in \R^n}\; f(x) = h (c(x)),
\end{equation}
where $c: \R^{n}\to \R^{m}$ is $C^{1}$-smooth and $h: \R^{m}\to \R$ is a convex function. Throughout this section, we also assume that $h$ is $L$-Lipschitz continuous with $\inf_y h(y) > -\infty$, and $\nabla c$ is $\beta$-Lipschitz continuous. For this problem, we will show in Proposition \ref{prop:equivalent_prox-linear} that the direction \eqref{eq:G_cvx-composite} obtained from subgradient regularization coincides with that of the widely studied prox-linear method. This connection motivates us to study a variant of Algorithm \ref{alg1}  that adaptively adjusts  the regularization parameter $\epsilon$ and stepsize $\eta$.

\subsection{A dual perspective of the prox-linear method}
\label{sec5.1:dual_prox-linear}

This section presents a dual perspective of the prox-linear method through the lens of subgradient regularization. To build intuition, we begin with a concrete example before generalizing in Proposition~\ref{prop:equivalent_prox-linear}. Consider the finite max of smooth functions $f(x) = \max_{1 \leq i \leq d}f_{i}(x)$ in Example~\ref{ex:max-of-smooth}, and recall that a descent-oriented subdifferential can be constructed via subgradient regularization. We now turn to the prox-linear method, which updates $x^{k+1}$ by solving
\[
    x^{k+1} = \argmin_{x \in \R^n} \left\{f(x; x^k) + \frac{1}{2 \epsilon}\|x - x^k\|^{2}\right\},
\]
where $f(x; x^k)$ is a piecewise linear approximation of $f$ at $x^k$:
\[
    f(x; x^k) 
    \triangleq \displaystyle\max_{1 \leq i \leq d} \left\{ f_i(x^k) + \nabla f_i(x^k)^\top (x - x^k) \right\}
    = \,\displaystyle\max_{y \in \Delta^{d}} \left\{\sum^d_{i=1} y_i \big(f_i(x^k) + \nabla f_i(x^k)^\top (x - x^k) \big) \right\}.
\]
This leads to the minimax formulation:
\[
    \min_{x \in \R^n} \; \max_{y \in \Delta^{d}}\left\{ \Phi(x,y) \triangleq
    \sum^{d}_{i=1}y_{i}\big(f_{i}(x^k) + \nabla f_{i}(x^k)^{\top}(x - x^k
    ) \big) + \frac{1}{2 \epsilon}\|x - x^k\|^{2}\right\}.
\]
By the minimax theorem~\cite[Propositions 1.2 and 2.4]{ekeland1999convex}, we may interchange the order of minimization and maximization. Completing the square reveals that the minimization problem attains its minimum at $x = x^k- \epsilon \sum^{d}_{i=1}y_{i}\nabla f_{i}(x^k)$. Substituting into the maximization problem gives:
\[
    \Max_{y \in \Delta^{d}}\left\{ \sum^{d}_{i=1} y_{i} \, f_{i}(x^k) - \frac{\epsilon}{2} \bigg\|\sum^{d}_{i=1}y_{i} \nabla f_{i}(x^k) \bigg\|^{2}\right\},
\]
which coincides with the subgradient-regularized problem \eqref{eq:P_regularized} for Example~\ref{ex:max-of-smooth}. Any solution $\bar y \in \Delta^{d}$ of this problem yields a saddle point $(x^{k+1}, \bar y)$ of $\Phi(x, y)$, and we recover the prox-linear update as
\[
    x^{k+1}= \argmin_{x \in \R^n}\; \Phi(x, \bar y) = x^k- \epsilon \, G(x^k, \epsilon) \quad\text{with}\quad G(x^k, \epsilon) \triangleq \sum^{d}_{i=1}\bar y_{i}\nabla f_{i}(x^k).
\]
Here, $G(x^k, \epsilon)$ is a descent-oriented subgradient for $f$, representing a convex combination of gradients $\{\nabla f_{i}(x^k)\}^{d}_{i=1}$. This example illustrates how the prox-linear method implicitly constructs descent directions through subgradient regularization, revealing its dual interpretation. 

\vspace{0.1in}
Next, we come back to the general composite problem \eqref{eq:cvx_composite}. A descent-oriented subdifferential has been defined in Example~\ref{ex:cvx_composite}. 
To define the prox-linear update, we consider an approximation of $f$ at $x$ by linearizing the inner smooth map $c$:
\[
    f(z; x) \triangleq h \big(c(x) + \nabla c(x)^\top (z - x) \big),
\]
which is convex  in $z$. It is known that $f(\cdot; x)$ is a \textsl{two-sided model}~\cite{drusvyatskiy2018error} satisfying
\begin{equation}\label{eq:two-sided-model}
    -\frac{L\beta}{2}\|z - x\|^{2} \leq f(z) - f(z;x) \leq \frac{L\beta}{2} \|z - x\|^{2}.
\end{equation}
This motivates the prox-linear method, which generates iterates via
\begin{equation}\label{eq:prox-linear}
    x^{k+1} = \argmin_{z \in \R^n}\left\{ f(z; x^k) + \frac{1}{2 \epsilon}\|z - x^k\|^{2}\right\}.
\tag{Prox-linear}
\end{equation}
Complementing the model-based interpretation in \eqref{eq:two-sided-model}, the following proposition offers a novel dual perspective by connecting it to subgradient regularization.

\begin{proposition}[Equivalence of prox-linear and subgradient regularization]
\label{prop:equivalent_prox-linear} 
    Consider problem \eqref{eq:cvx_composite}, and let $G$ be the descent-oriented subdifferential defined in \eqref{eq:G_cvx-composite} via subgradient regularization. Then \eqref{eq:prox-linear} with stepsize $\epsilon$ is equivalent to the update $x^{k+1} = x^k - \epsilon \, G(x^k, \epsilon)$.
\end{proposition}
\begin{proof}
    By the convexity of $h$, we can rewrite \eqref{eq:prox-linear} as a minimax problem using the conjugate function of $h$:
    \begin{equation}\label{eq:prox-sub-conjugate}
        x^{k+1}= \argmin_{x \in \R^n}\left\{\max_{y \in \R^m}\left[y^{\top}\big(c(x^k) + \nabla c(x^k)^\top (x - x^k) \big) - h^{\ast}(y) + \frac{1}{2 \epsilon}\|x - x^k\|^{2}\right] \right\}.
    \end{equation}
    Observe that for any $x$, the inner maximum is attained at some $y_{x}\in \partial h \big(c(x^k) + \nabla c(x^k)^\top (x-x^k) \big)$.
    Since $h$ is $L$-Lipschitz continuous, it follows that $\|y_{x}\| \leq L$ for any $x$. Consequently, we can write \eqref{eq:prox-sub-conjugate} equivalently by restricting $y$ in the compact set $Y = \{y \in \R^{m}\mid \|y\| \leq L\}$. Similarly, due to the strong convexity of $\|x - x^k\|^{2} / (2\epsilon)$, we can also restrict $x$ in the compact set $\overline\B_{\rho}(x^k)$, where the radius is given by $\rho \triangleq \max\big\{\epsilon L \|\nabla c(x^k)\|, \sqrt{2 \epsilon \,[h(c(x^k)) - \inf h ]} \big\}$. 
    Thus, the minimax problem \eqref{eq:prox-sub-conjugate} is equivalent to
    \begin{equation}\label{eq:minimax_bounded}
        \min_{x \in \overline\B_\rho(x^k)} \; \max_{y \in Y}\left[y^{\top}\big(c(x^k) + \nabla c(x^k)^\top (x - x^k) \big) - h^{\ast}(y) + \frac{1}{2 \epsilon}\|x - x^k\|^{2}\right] .
    \end{equation}
    Since the objective function is convex-concave, $\overline\B_{\rho}(x^k)$ and $Y$ are compact, it follows from the minimax theorem~\cite[Propositions 1.2 and 2.4]{ekeland1999convex} that there exists a saddle point on $\overline\B_{\rho}(x^k) \times Y$.
    Thus, we can interchange the minimize and maximum in \eqref{eq:minimax_bounded}. As the inner minimal value is attained at $x = \big(x^k- \epsilon \, \nabla c(x^k) \, y\big) \in \overline\B_{\rho}(x^k)$ for any $y \in Y$, we get a dual solution by direct computation:
    \begin{equation}\label{eq:dual-sol}
        \by \in \argmax_{y \in Y} \left\{ 
        y^{\top} c(x^k) - h^{\ast}(y) - \frac{\epsilon}{2}\big\|\nabla c(x^k) \, y\big\|^{2} 
        \right\}.
    \end{equation}
    Because $(x^{k+1}, \by)$ is a saddle point of the minimax problem \eqref{eq:minimax_bounded}, by the optimality of $x^{k+1}$ for the minimization problem at $\by $, we immediately obtain $x^{k+1}= x^k- \epsilon \, \nabla c(x^k) \, \by$. 
    Noting that the feasible region $Y$ in \eqref{eq:dual-sol} can be omitted without affecting the optimal solution, we have $\nabla c(x^k) \, \bar y = G(x^k, \epsilon)$ and, thus, $x^{k+1} = x^k - \epsilon \, G(x^k, \epsilon)$, completing the proof.
\end{proof}

As a consequence of Proposition \ref{prop:equivalent_prox-linear}, we obtain a stronger version of Proposition \ref{prop:descent_direc} on uniform non-diminishing ranges of  $\epsilon$ and $\eta$ to guarantee the descent inequality \eqref{eq:descent_ineq}. The result below follows directly from~\cite[Lemma 5.1]{drusvyatskiy2018error}.

\begin{corollary}[Descent directions for composite functions]
\label{cor:uniform_descent_direc}
    Consider the nonsmooth composite problem \eqref{eq:cvx_composite}, and let $G$ be the descent-oriented subdifferential defined in \eqref{eq:G_cvx-composite}. For any $\epsilon \in (0, 1/(2L\beta)]$, the following descent inequality holds for every $\bx$:
    \[
        f(\bx - \epsilon g) \leq f(\bx) - \epsilon \|g\|^{2},
        \qquad\forall\, g \in G(\bx, \epsilon).
    \]
\end{corollary}

\subsection{An adaptive descent-oriented subgradient method}
\label{sec5.2:algorithm_adaptive}

We have seen in previous section that specializing to the composite structure \eqref{eq:cvx_composite}, the direction obtained from subgradient regularization is equivalent to that of the prox-linear method. It is known that prox-linear method with a constant stepsize enjoys local linear convergence under an error bound condition~\cite{drusvyatskiy2018error}. This naturally leads to the question of whether our proposed algorithm can also attain a local linear rate under the same condition.

However, Algorithm \ref{alg1} with $G$ defined in \eqref{eq:G_cvx-composite} does not guarantee such a rate. The algorithm typically decreases the parameter $\epsilon$ over iterations and employs a diminishing stepsize  to ensure subsequential convergence to stationary points. 
To address this limitation, we propose an adaptive variant of Algorithm \ref{alg1}, presented as Algorithm \ref{alg2}.

\begin{algorithm}[htp]
\caption{An adaptive descent-oriented subgradient method}
\begin{algorithmic}[1]
    \STATE {\bf Input:} 
    starting point $x_{0}$, 
    regularization parameter $\epsilon_{0,0}> 0$, 
    a non-summable sequence $\{a_t\} \downarrow 0$,
    initial stationary target $\nu_{0}$, 
    optimality tolerances $\epsilon_{\text{tol}}, \nu_{\text{tol}} \geq 0$, 
    counter of ratio test $t = 0$, 
    reduction factors $\theta_{\epsilon}, \theta_{\nu} \in (0, 1)$, 
    and the Armijo parameter $\alpha \in (0, 1)$
    \FOR{$k= 0, 1, \cdots $}
        \FOR{$i = 0, 1, \cdots $}
            \STATE Set $\epsilon_{k,i}= \epsilon_{k,0} \, 2^{-i}$ \hfill 
            \STATE Update direction $g^{k, i} \in G(x^k, \epsilon_{k,i})$
    
            \IF{$\epsilon_{k,i}\leq \epsilon_{\text{tol}}$ and $\|g^{k, i}\| \leq \nu_{\text{tol}}$}
                \STATE Set $i_{k}= i$ and {\bf return} $x^k$ \hfill $\triangleright$\textit{ approximate stationary}
            \ENDIF
    
            \FOR{$j = 0, 1, \cdots, i$}
                \STATE Set $\eta = \epsilon_{k,0}  \, 2^{-j}$ \hfill $\triangleright$\textit{ line search}
                \IF{$f(x^k- \eta \, g^{k, i}) \leq f(x^k) - \alpha \, \eta \|g^{k, i}\|^{2}$}
                    \STATE Set $i_{k}= i$ {and $\eta_k = \argmin\big\{f(x^k - \eta\,g^{k,i}) \mid \eta \in \{\epsilon_{k,0}\,2^{-j}, \cdots, \epsilon_{k,0}\,2^{-i}\} \big\}$}
                    \STATE Go to line 13
                \ENDIF
            \ENDFOR
        \ENDFOR
        \STATE Update $x^{k+1} = x^k- \eta_{k}\, g^{k, i}$ 

        \IF{$\|g^{k, i_k}\| \leq \nu_{k}$}
            \STATE $t \leftarrow (t+1)$ 
            \STATE Update $\tilde g^k\in G(x^k, \tilde\epsilon_{k})$ with $\tilde\epsilon_{k} = (a_{t})^{1/4}$
    
        \IF{$\tilde\epsilon_{k}\leq \epsilon_{\text{tol}}$ and $\|\tilde g^k\| \leq \nu_{\text{tol}}$}
            \STATE {\bf Return} $x^k$ \hfill $\triangleright$\textit{ approximate stationary}
        \ENDIF
            \STATE Set $\nu_{k+1}= \theta_{\nu} \cdot \nu_{k}$ 
            
            \IF{$\frac{\tilde\epsilon_{k} \|\tilde g^k\|}{\left(\epsilon_{k,i_k} \|g^{k, i_k}\|\right)^{1/2}} \leq \frac{1}{\epsilon_{k,0}}$}
                \STATE Set $\epsilon_{k+1, 0} = \epsilon_{k, 0}$ \hfill $\triangleright$\textit{ ratio test} 
            \ELSE 
                \STATE Set $\epsilon_{k+1,0} = \theta_{\epsilon} \cdot \epsilon_{k,0}$
            \ENDIF 
        \ELSE 
            \STATE Set $\epsilon_{k+1,0} = \epsilon_{k,0}$ and $\nu_{k+1} = \nu_{k}$ 
        \ENDIF
    \ENDFOR
\end{algorithmic}
\label{alg2}
\end{algorithm}

The key innovation of Algorithm \ref{alg2} lies in the presence of two sequences $\{\epsilon_{k,0}\}$ and $\{\tilde\epsilon_k\}$. The primary sequence $\{\epsilon_{k,0}\}$ is responsible for computing descent directions, as in Algorithm \ref{alg1}, while the auxiliary diminishing sequence $\{\tilde\epsilon_k\}$ is used to monitor approximate stationarity through the norm of $\tilde g^k \in G(x^k, \tilde\epsilon_k)$. We compare $\|\tilde g^k\|$ and $\|g^{k,i_k}\|$, where $g^{k,i_k}$ denotes the descent-oriented subgradient satisfying the descent inequality, to determine whether to reduce $\epsilon_{k,0}$ in the next iteration. Specifically, when the ratio test
\[
    \frac{\tilde\epsilon_{k} \|\tilde g^k\|}{(\epsilon_{k,i_k} \|g^{k, i_k}\|)^{1/2}} \leq \frac{1}{\epsilon_{k,0}}
\]
is satisfied, $\epsilon_{k,0}$ remains a constant; otherwise, it is reduced by a factor of $\theta_\epsilon$. This mechanism allows Algorithm \ref{alg2} to maintain a non-diminishing effective stepsize whenever possible, thereby recovering local linear convergence for structured problems while still ensuring subsequential convergence to stationary points in general settings.

\subsubsection{Convergence analysis}

The inner loop indexed by $i$ in Algorithm \ref{alg2} is essentially the same as that of Algorithm \ref{alg1}. So the inner loop terminates in finitely many iterations by Proposition \ref{prop:inner_loop} and $i_k < +\infty$. We next prove the subsequential convergence.
\begin{theorem}[Subsequential convergence of Algorithm \ref{alg2}]
\label{thm:convergence_alg2}
    Suppose that $f$ is bounded from below. Let $\{x^k\}$ be a sequence generated by Algorithm \ref{alg2} with $\epsilon_{\text{tol}} = \nu_{\text{tol}}= 0$. One of the following occurs:\\
    (a) For some $k$, the inner loop does not terminate. Then $x^k$ is a stationary point.\\
    (b) There is an index set $K \in \N^{\sharp}_{\infty}$ such that $\tilde g^k \to_{K} 0$. Consequently, for any such subsequence $\tilde g^k \to_{K} 0$, any accumulation point of $\{x^k\}_{k \in K}$ is a stationary point.\\
    (c) Any accumulation point of $\{x^k\}$ is a stationary point.
\end{theorem}
\begin{proof}
    If the inner loop does not terminate for some $k$, then by Proposition \ref{prop:inner_loop}, we have $0 \in \partial f(x^k)$. Now suppose the inner loop terminates for every $k$. We will prove that either (b) or (c) holds. According to Proposition \ref{prop:inner_loop}, there exists an index $i_{k}$ for each $k$ such that the following descent inequality holds:
    \[
        f(x^{k+1}) = f(x^k- \eta_{k}\, g^{k, i_k}) \leq f(x^k) - \alpha \, \eta_{k} \|g^{k, i_k}\|^{2}.
    \]
    Since $f$ is bounded from below, summing the descent inequality over $k$ gives
    \begin{equation}\label{eq:finite_sum}
        \sum^{+\infty}_{k=0}\eta_{k}\|g^{k, i_k}\|^{2}= \sum^{+\infty}_{k=0}\|x^{k+1}- x^k\| \cdot \|g^{k, i_k}\| < +\infty.
    \end{equation}

    \noindent
    {\bf Case 1:} Suppose there exist $\underline k \in \N$, $\underline\epsilon > 0$, and $\underline\nu > 0$ such that $\epsilon_{k,0}= \underline\epsilon$ and $\nu_{k}= \underline\nu$ for all $k \geq \underline k$. Then, for all such $k$, we have $\|g^{k, i_k}\| > \underline\nu$, and consequently, $\tilde g^k$ is not constructed. Therefore statement (b) does not hold, and it remains to prove (c).
    Using \eqref{eq:finite_sum} and the fact that $\|g^{k, i_k}\| > \underline\nu$ for all $k \geq \underline k$, we can apply the same arguments as in Case 1 of Theorem \ref{thm:convergence_alg1} to conclude that $\{x^k\}$ converges to some point $\bx$ with $0 \in \partial f(\bx)$. Therefore, statement (c) holds.

    \vspace{0.1in}
    \noindent
    {\bf Case 2:} Otherwise, there exists $N \in \N^{\sharp}_{\infty}$ such that $g^{k, i_k}\to_{N}0$. Let $\bx$ be an accumulation point of $\{x^k\}$. By the descent inequality, we apply Lemma \ref{lem:consequence_of_descent} to have that either $x^k \to \bx$, or $\liminf_{k \to +\infty}\, \max\big\{\|x^k- \bx\|, \|g^{k, i_k}\|\big\} = 0$.
    In either scenarios, by passing to a subsequence if necessary, we can assume that $\max\{\|x^k- \bx\|, \|g^{k, i_k}\|\} \to_{N} 0$.

    \noindent
    {\bf Case 2.1:} If $\epsilon_{k,0}\downarrow 0$, since $\epsilon_{k, i_k} \leq \epsilon_{k,0}\downarrow 0$, property \eqref{eq:Def_G1} implies
    \[
        0 = \lim_{k(\in N) \to +\infty}g^{k, i_k}\,\in \left[\,\limsup_{k(\in N) \to +\infty}G(x^k, \epsilon_{k,i_k})\right] \subset \partial f(\bx).
    \]
    Thus, any accumulation point $\bx$ of $\{x^k\}$ satisfies $0 \in \partial f(\bx)$, establishing (c).

    \vspace{0.05in}
    \noindent
    {\bf Case 2.2:} Otherwise, the sequence $\{\epsilon_{k,0}\}$ converges monotonically to a positive scalar $\underline\epsilon$, and there exists $\underline k \in \N$ such that $\epsilon_{k,0} = \underline\epsilon$ for all $k \geq \underline k$. Define the index set $N^{\prime} \triangleq \{ k \in \N \mid \|g^{k, i_k}\| \leq \nu_{k} \}$. 
    Since $\|g^{k, i_k}\| \to_{N}0$, it follows that $N^{\prime} \in \N^{\sharp}_{\infty}$. Let $k_{t}$ denote the $t $-th index in $N^{\prime}$. To prove statement (b), we aim to find a subsequence $\tilde g^k \to_{K} 0$ for some index set $K \in N^{\sharp}_{\infty}$ with $K \subset N^{\prime}$.
    
    Suppose for contradiction that there exists $\delta > 0$ such that $\|\tilde g^k\| \geq \delta$ for all $k \in N^{\prime}$, or equivalently, $\|\tilde g^{k_t}\| \geq \delta$ for all $t \in \N$. For any $t \in \N$ with $k_{t} \geq \underline k$, since $\epsilon_{k_t, 0} = \underline\epsilon$, the ratio test must be satisfied:
    \[
        \frac{\tilde\epsilon_{k_t}\|\tilde g^{k_t}\|}{\big(\epsilon_{k_t, i_{k_t}}\|g^{k_t, i_{k_t}}\|\big)^{1/2}} 
        \leq \frac{1}{\epsilon_{k_t,0}} 
        = \frac{1}{\underline\epsilon}.
    \]
    Observe that $\tilde\epsilon_{k_t} = (a_{t})^{1/4}$ by the update rule of $\tilde\epsilon_{k}$. For all $t \in \N$ with $k_t \geq \underline k$, Rearranging the inequality and using $\epsilon_{k_t, i_{k_t}} = \epsilon_{k_t, 0} \, 2^{-i_{k_t}}\leq \eta_{k_t} \leq \epsilon_{k_t, 0} = \underline\epsilon$ yield
    \[
        (a_{t})^{\frac{1}{4}}\, \|\tilde g^{k_t}\| = \tilde\epsilon_{k_t}\|\tilde g^{k_t}\|
        \leq \underline\epsilon^{-1} \big(\epsilon_{k_t, i_{k_t}}\|g^{k_t, i_{k_t}}\|\big)^{1/2}
        \leq \underline\epsilon^{-3/4} \big(\eta_{k_t} \| g^{k_t, i_{k_t}}\|^{2}\big)^{1/4}.
    \]
    Summing both sides over $t \in \N$ with $k_t \geq \underline k$ and using the lower bound $\|\tilde g^{k_t}\| \geq \delta$, we get
    \begin{align*}
        \delta^{4}\sum_{\{t \in \N \,\mid\, k_t \geq \underline k\}}a_{t}
        \leq \sum_{\{t \in \N \,\mid\, k_t \geq \underline k\}} a_{t} \, \|\tilde g^{k_t}\|^{4} 
        & \leq\; \underline\epsilon^{-3}\cdot \sum_{\{t \in \N \,\mid\, k_t \geq \underline k\}}\eta_{k_t}\|g^{k_t, i_{k_t}}\|^{2} \\
        & \leq\; \underline\epsilon^{-3}\cdot \sum_{k \in \N}\eta_{k}\|g^{k, i_k}\|^{2} 
        \;\overset{\eqref{eq:finite_sum}}{<}\; \infty.
    \end{align*}
    This implies $\sum_{\{t \in \N \,\mid\, k_t \geq \underline k\}}a_{t} < +\infty$, contradicting the fact that $\{a_{t}\}$ is a non-summable sequence. Hence, there exists a subsequence $\tilde g^k \to_{K} 0$ for some index set $K \in \N^{\sharp}_{\infty}$ with $K \subset N^{\prime}$. Finally, let $\bx$ be any accumulation point of $\{x^k\}_{k \in K}$. Since $\tilde\epsilon_{k_t}= (a_{t})^{1/4}\downarrow 0$, it follows from property \eqref{eq:Def_G1} that $0 \in \partial f(\bx)$. This establishes statement (b).
\end{proof}

\subsubsection{Local linear convergence for nonsmooth composite functions}
\label{sec:local_linear_converge}

In this section, we analyze convergence rate of Algorithm \ref{alg2} to solve composite problem \eqref{eq:cvx_composite} under regularity conditions, where $G(x,\epsilon)$ is obtained from the subgradient regularization as in \eqref{eq:G_cvx-composite}. Let us denote the set of stationary points of \eqref{eq:cvx_composite} by $(\partial f)^{-1}(0) \triangleq \{x \in \R^{n}\mid 0 \in \partial f(x)\}$. We begin by recalling the notion of metric subregularity.

\begin{definition}
    A set-valued mapping $S: \R^{n}\rightrightarrows \R^{m}$ is \textsl{metrically subregular at $(\bx, \bar y) \in \operatorname{gph} S$} if there exist positive scalars $\kappa$ and $\tau$ such that
    \[
        \dist\big(x, S^{-1}(\bar y)\big) \leq \kappa \cdot \dist(\bar y, S(x)),
        \qquad\forall\, x \in \B_{\tau}(\bx).
    \]
\end{definition}

The following lemma provides an error bound based on metric subregularity. It is a uniformed version of \cite[Theorem 5.11]{drusvyatskiy2018error} in terms of $\epsilon$. A self-contained proof is given in Appendix B.1.
\begin{lemma}
\label{lem:from_metric_subregular_to_error_bound} 
    Consider problem \eqref{eq:cvx_composite} with $G$ defined in \eqref{eq:G_cvx-composite}. Let $\bx$ be a stationary point. If $\partial f$ is metrically subregular at $(\bx, 0)$ with modulus $\kappa$, then for any $\bar\epsilon > 0$, there exists a constant $\tau > 0$ such that the following error bound holds:
    \[
        \dist\big(x, (\partial f)^{-1}(0) \big) \leq \big((3L\beta \epsilon + 2) \kappa + 2\epsilon \big) \cdot \|G(x, \epsilon)\|,
        \qquad\forall\, x \in \B_{\tau}(\bx),\, \epsilon \in (0,\bar\epsilon).
    \]
\end{lemma}

With all the ingredients, we begin by establishing in Proposition \ref{prop:constant_stepsize} that the stepsize sequence $\{\eta_{k}\}$ in Algorithm \ref{alg2} remains uniformly bounded away from zero. Recall that for each iteration $k$, the stepsize $\eta_{k}$ is selected from the set $\{\epsilon_{k,0} \,, 2^{-1}\epsilon_{k,0} \,,\, \cdots \,, 2^{-i_k}\epsilon_{k,0}\}$ satisfying the descent condition $f(x^k- \eta g^{k, i}) \leq f(x^k) - \alpha \, \eta \|g^{k, i}\|^{2}$. If the ratio test fails, the upper bound $\epsilon_{k, 0}$ is reduced by a factor of $\theta_{\epsilon}$. According to Corollary \ref{cor:uniform_descent_direc} and given $\alpha \in (0,1)$, any $\epsilon \in (0, 1/(2L \beta) ]$ paired with a stepsize $\eta = \epsilon$ ensures the descent condition. Thus, a uniform lower bound on $\{\eta_{k}\}$ can be guaranteed by showing that $\epsilon_{k,0}$ is only reduced by a factor of $\theta_{\epsilon}$ finitely many times; i.e., the ratio test fails only for finitely many iterations under an error bound condition. We defer the proof to Appendix B.2.

\begin{proposition}[Eventual constant stepsize]
\label{prop:constant_stepsize} 
    Consider problem \eqref{eq:cvx_composite}, and let $\{x^k\}$ be a sequence generated by Algorithm \ref{alg2} using $G$ defined in \eqref{eq:G_cvx-composite}. Suppose that the level set $X \triangleq \{x \in \R^{n}\mid f(x) \leq f(x^0)\}$ is bounded, and $\partial f$ is metrically subregular at $(\bx, 0)$ for every $\bx \in (\partial f)^{-1}(0) \cap X$. Then there exists a constant $\underline\epsilon > 0$ such that $\epsilon_{k,0}= \underline\epsilon$ for all sufficiently large $k$. Consequently, $\eta_{k} \geq \epsilon_{k,i_k} \geq \min\{\underline\epsilon, 1/(4L\beta)\}$ for all sufficiently large $k$.
\end{proposition}

To establish local linear convergence, we next state a key inequality from~\cite[Lemma 5.1]{drusvyatskiy2018error} that quantifies the decrease in the objective value after taking a step along the direction $-G(x, \epsilon)$.
\begin{lemma}
\label{lem:Gradient_ineq}
    Consider nonsmooth composite problem \eqref{eq:cvx_composite} with $G$ defined in \eqref{eq:G_cvx-composite}. For any $x, y \in \R^n$ and $\epsilon > 0$, we have
    \begin{equation}\label{eq:Gradient_ineq}
        f(y) \geq f\big(x - \epsilon \, G(x, \epsilon)\big) + \langle G(x, \epsilon), y - x \rangle + \frac{\epsilon}{2}(2 - L \beta \epsilon) \|G(x, \epsilon)\|^{2}- \frac{L\beta}{2} \|x - y\|^{2}.
    \end{equation}
\end{lemma}

We now impose an additional assumption ensuring a separation of the isocost surfaces of $f$ on the set of stationary points $(\partial f)^{-1}(0)$. This assumption is standard in the literature, and sufficient conditions for it to hold can be found in \cite{luo1993error}.

\begin{assumption}
\label{assump3:separate}
    There is a constant $\delta_0 > 0$ such that
    \[
        x \in (\partial f)^{-1}(0),\; y \in (\partial f)^{-1}(0),\; f(x) \neq f(y) \quad\Rightarrow\quad \|x - y\| \geq \delta_{0}.
    \]
\end{assumption}

Finally, we arrive at the main convergence result.
\begin{corollary}[Local linear convergence]
\label{cor:local_linear_converge}
    Consider problem \eqref{eq:cvx_composite} under Assumption \ref{assump3:separate}, and let $\{x^k\}$ be a sequence generated by Algorithm \ref{alg2} with $G$ defined in \eqref{eq:G_cvx-composite}. Suppose that the level set $X \triangleq \{x \in \R^{n}\mid f(x) \leq f(x^0)\}$ is compact, and $\partial f$ is metrically subregular at $(\bx, 0)$ for every $\bx \in (\partial f)^{-1}(0) \cap X$. Then for all sufficiently large $k$, the function values $f(x^k)$ and the iterates $x^k$ converge linearly.
\end{corollary}

\begin{proof}
    By Proposition \ref{prop:constant_stepsize}, there exist $\underline k \in \N$ and $\underline\epsilon > 0$ such that $\epsilon_{k,0} = \underline\epsilon$ and $\eta_k \geq \epsilon_{k,i_k} \geq \underline\eta \triangleq \min\{\underline\epsilon, 1/(4L\beta)\}$ for all $k \geq \underline k$.
    Let $\bx$ be an accumulation point of $\{x^k\}$. By the continuity of $f$, $f(\bx)$ is an accumulation point of $\{f(x^k)\}$. Since $\{f(x^k)\}$ is non-increasing and bounded from below, it converges and $f(x^k) \to f(\bx)$.
    From the descent inequality
    \begin{equation}\label{eq:descent_ineq2}
        f(x^{k+1}) \leq f(x^k) - \alpha \, \eta_{k} \|g^{k, i_k}\|^{2} = f(x^k) - \frac{\alpha}{{\eta_k}} \|x^{k+1} - x^k\|^2
    \end{equation}
    and the lower bound $\eta_{k} \geq \underline\eta$, we obtain
    \[
        \|g^{k, i_k}\|^2 
        \leq {\frac{f(x^k) - f(x^{k+1})}{\alpha \, \eta_{k}}}
        \leq {\frac{f(x^k) - f(x^{k+1})}{\alpha \, \underline\eta} }\to 0.
    \]
    Thus, $G(x^k, \epsilon_{k,i_k}) = g^{k, i_k} \to 0$. By Lemma \ref{lem:subgrad-regularized_Affine_in_y}(b), $\bx$ is a stationary point. Since $\partial f$ is metrically subregular at $(\bx, 0)$, Lemma \ref{lem:from_metric_subregular_to_error_bound} guarantees the existence of positive scalars $\kappa_{\bx}$ and $\tau$ such that
    \begin{equation}\label{eq:Error_Bound}
        \dist\big(x, (\partial f)^{-1}(0)\big)
        \leq \big( (3L \beta \epsilon + 2) \kappa_{\bx}+ 2\epsilon \big) \cdot \|G(x, \epsilon)\|,
        \qquad\forall\, x \in \B_{\tau}(\bx),\, \epsilon > 0.
    \end{equation}
    Now apply the inequality \eqref{eq:Gradient_ineq} from Lemma \ref{lem:Gradient_ineq} with $x = x^k$, $\epsilon = \epsilon_{k,i_k}$, and $y = \hx^k$, where $\hx^k \in (\partial f)^{-1} (0)$ is a point such that $\|x^k- \hx^k\| = \dist\big(x^k, (\partial f)^{-1}(0)\big)$. For all $k \in \N$, we have
    \[
    \begin{array}{rl}
        & \quad f\big(x^k - \epsilon_{k,i_k} g^{k,i_k}\big) - f(\hx^k) \\
        & \displaystyle\leq -\big\langle g^{k, i_k}, \hx^k - x^k \big\rangle - \frac{\epsilon_{k,i_k}}{2} (2 - L \beta \epsilon_{k,i_k}) \|g^{k, i_k}\|^2 + \frac{L \beta}{2} \|x^k - \hx^k\|^2  \\[0.1in]
        & \displaystyle\leq \|g^{k, i_k}\|^2 \left( \frac{\|x^k - \hx^k\|}{\|g^{k, i_k}\|} - \frac{\epsilon_{k,i_k}}{2} (2 - L \beta \epsilon_{k,i_k}) + \frac{L \beta}{2} \frac{\|x^k - \hx^k\|^{2}}{\|g^{k, i_k}\|^{2}} \right).
    \end{array}
    \]
    Define $\gamma \triangleq \max\{1, (3L\beta \underline\epsilon + 2) \kappa_{\bx}+ 2 \underline\epsilon \}$. Consider an iterate $x^k$ such that $k \geq \underline k$ and $\|x^k - \bx\| < \min\{\tau, \delta_0/2\}$. The error bound \eqref{eq:Error_Bound}, combined with the fact $\epsilon_{k,i_k} \leq \underline\epsilon$, implies that $\|x^k - \hx^k\| \leq \gamma \|g^{k,i_k}\|$. Furthermore, noting that $f(x^{k+1}) \leq f(x^k - \epsilon_{k,i_k} g^{k,i_k})$ due to line 11 of Algorithm \ref{alg2}, we obtain
    \[
    \begin{array}{rl}
         f(x^{k+1}) - f(\hx^k)
         &\;\displaystyle\leq\, \|g^{k, i_k}\|^2 \left( \left(1 + {L\beta}/{2}\right) \gamma^2 - \frac{\epsilon_{k,i_k}}{2} (2 - L \beta \epsilon_{k,i_k}) \right)  \\[0.15in]
         &\displaystyle\overset{\eqref{eq:descent_ineq2}}{\leq} \frac{f(x^k) - f(x^{k+1})}{\alpha \, \eta_{k}} \left( \left(1 + {L\beta}/{2}\right) \gamma^2 - \frac{\epsilon_{k,i_k}}{2} (2 - L \beta \epsilon_{k,i_k}) \right)  \\[0.15in]
         &\;\displaystyle\leq \big(f(x^k) - f(x^{k+1})\big) \left( \left(1 + {L\beta}/{2}\right) \frac{\gamma^{2}}{\alpha \underline\eta} - \frac{2 - L \beta \underline\epsilon}{2 \alpha} \right).
    \end{array}
    \]
    The last inequality uses the bounds $\eta_{k} \geq \epsilon_{k,i_k} \geq \underline\eta$ and $\epsilon_{k,i_k} \leq \epsilon_{k,0} = \underline\epsilon$ for $k \geq \underline k$.
    Rearranging terms, we arrive at the recursion:
    \[
        f(x^{k+1}) - f(\hx^k) \leq q \,\big(f(x^k) - f(\hx^k)\big),
    \] 
    where
    \[
        q \triangleq 1 - \left({(1+{L\beta}/{2}) \frac{\gamma^2}{\alpha \, \underline\eta} - \frac{2-L\beta \underline\epsilon}{2 \alpha} + 1}\right)^{-1} \in (0,1).
    \]
    Observe that $\|\hx^k - \bx\| \leq \|\hx^k - x^k\| + \|x^k - \bx\| \leq 2\|x^k - \bx\| < \delta_0$. Since both $\hx^k$ and $\bx$ are stationary points, Assumption \ref{assump3:separate} implies that $f(\hx^k) = f(\bx)$. Consequently, the recursion simplifies to, for all $k \geq \underline k$ such that $x^k \in \B_{\min\{\tau, \delta_0/2\}}(\bx)$,
    \begin{equation}\label{eq:f_Q-linear}
        f(x^{k+1}) - f(\bx) \leq q \,\big(f(x^k) - f(\bx)\big).
    \end{equation} 
    We next show that once $x^k$ is sufficiently close to $\bx$, all subsequent iterates remain within the neighborhood. Let $k_{1} \geq \underline k$ be such that $x^{k_1}\in \B_{\min\{\tau, \delta_0/2\}}(\bx)$. Suppose, for the sake of contradiction, that a later iterate exits this neighborhood. Then there exists $k_{2}\geq k_{1} + 1$ as the smallest index such that $x^{k_2}\notin \B_{\min\{\tau, \delta_0/2\}}(\bx)$. We can estimate the distance between $x^{k_1}$ and $x^{k_2}$ as follows:
    \[
    \begin{array}{rl}
        \|x^{k_2} - x^{k_1}\|
        & \displaystyle\leq \sum^{k_2 - 1}_{k=k_1} \|x^{k+1} - x^k\|
        \;\overset{\eqref{eq:descent_ineq2}}{\leq}\; \sum^{k_2 - 1}_{k=k_1} \sqrt{\eta_{k} \big(f(x^{k+1}) - f(x^k)\big) / \alpha} \\
        & \displaystyle\overset{\eqref{eq:f_Q-linear}}{\leq} \sqrt{\underline\epsilon \,\big(f(x^{k_1+1}) - f(\bx)\big) / \alpha} \,\sum^{k_2 - 1}_{k=k_1} q^{(k - k_1)/2} \,,
    \end{array}
    \]
    which can be made arbitrarily small by choosing $k_1$ sufficiently large, implying that $x^{k_2} \in \B_{\min\{\tau, \delta_0/2\}}(\bx)$, a contradiction. Thus, there exists an index $r$ such that $x^k \in \B_{\min\{\tau, \delta_0/2\}}(\bx)$ for all $k \geq r$. Consequently, the geometric decay in \eqref{eq:f_Q-linear} holds for all $k \geq r$, and the function values $f(x^k)$ asymptotically converge linearly. 
    Finally, the inequalities \eqref{eq:descent_ineq2} and \eqref{eq:f_Q-linear}, combined with $\|x^k - \bx\| \leq \sum^{+\infty}_{i=k} \|x^i - x^{i+1}\|$, allow us to conclude that the iterates $x^k$ converge linearly to $\bx$ for $k \geq r$.
\end{proof}

\section{Numerical Experiments}
\label{sec6:nuericals}

In this section, we evaluate the performance of our algorithms on several classes of nonsmooth functions. All experiments were conducted in \textsc{Matlab} 2023b on a desktop equipped with an Intel Core i7-13700 CPU and 32 GB of RAM. The code is available at \url{https://github.com/lhyoung99/subgradient-regularization}.
For each test example, we consider two types of oracles:

{\bf Oracle 1:} Given $x$, return the objective $f(x)$ and a subgradient $v \in \partial f(x)$.

{\bf Oracle 2:} Given $x$ and $\epsilon > 0$, return the objective $f(x)$ and a descent-oriented subgradient $g \in G(x, \epsilon)$.

\vspace{0.05in}
\noindent The descent-oriented subdifferential $G$ in Oracle 2 is constructed by solving a subgradient-regularized problem, which reduces to a convex quadratic program (see Section \ref{sec4:subgrad-regularization} for details). Among several solvers tested, \textsc{Matlab}'s built-in function \textit{quadprog}  was found to be the most efficient for these subproblems.

\vspace{0.05in}
\noindent 
When Algorithms \ref{alg1} and \ref{alg2} are equipped with the subgradient regularization to generate $G(x,\epsilon)$, we call them SRDescent and SRDescent-adapt, respectively. 
We summarize the parameter setting used in these two algorithms.

\vspace{0.05in}
{\em Regularization parameters}. We set the initial regularization parameter to $\epsilon_{0,0} = 5$ and apply a reduction factor $\theta_{\epsilon} = 0.9$. In Algorithm \ref{alg2}, we use the non-summable sequence $\{a_t = 1/t\}$ for the regularization parameter $\tilde\epsilon_k$.

{\em Stationary tolerance parameters}. The initial target for the stationarity measure is set to $\nu_0 = 10^{-2}$, with a reduction factor $\theta_{\nu} = 0.5$.

{\em Line search parameter}. We use $\alpha = 10^{-4}$ as the Armijo parameter.

\vspace{0.1in}
\noindent For comparison, we consider the following algorithms, all relying on Oracle 1:

\vspace{0.05in}
\underline{Polyak}: The subgradient method with Polyak stepsizes that iterates as
\[
    x^{k+1} = x^k - \frac{f(x^k) - f^\ast}{\|g^k\|^2} g^k \quad\text{for some $g^k \in \partial f(x^k)$},
\]
where $f^*$ is the optimal objective value.

\vspace{0.05in}
\underline{\textsc{PBMDC}}: A proximal bundle method for difference-of-convex functions~\cite{de2019proximal},  available at \url{https://www.oliveira.mat.br/solvers}.

\vspace{0.05in}
\underline{NTDescent}: Originally implemented in PyTorch~\url{https://github.com/COR-OPT/ntd.py}, we re-implemented this method in \textsc{Matlab} to ensure fair runtime comparison.   Across several test cases under the same termination conditions, the \textsc{Matlab} version matches the iteration count of the original implementation while being faster in CPU time.

\vspace{0.05in}
\underline{\textsc{HANSO}}: A hybrid nonsmooth optimization algorithm combining BFGS~\cite{lewis2013nonsmooth} with gradient sampling (\textsc{GS})~\cite{burke2005robust}, available at \url{https://cs.nyu.edu/~overton/software/hanso/}. We separately tested \textsc{GS} as an implementable Goldstein-type method with convergence guarantees, and BFGS as a heuristic known for good empirical performance. In many examples, GS suffers from the line search failures, which, as noted in \cite{burke2005robust}, occur when either the maximum number of backtracking steps (set to 50) is reached or the computed direction is not a descent direction.

\vspace{0.05in}
\noindent Unless otherwise specified, all parameters were set to their default values in these algorithms. Performance is evaluated primarily in terms of CPU time and the number of calls to Oracle 1 and Oracle 2, respectively.

\subsection{Finite max of convex quadratic functions}
\label{sec6.1:Max_of_cvx_quad}
We first test our algorithms
on the finite max of convex quadratic functions:
\[
    \Min_{x \in \R^n} \left\{ f(x) = \max_{1 \leq i \leq m} \left(g_i^\top x + \frac{1}{2} x^\top H_i \, x\right) \right\},
\]
where each $H_i \in \R^{n \times n}$ is a randomly generated positive semi-definite matrix, and $\{g_i\}^{m}_{i=1}$ are random vectors satisfying $\sum^{\lfloor m/2 \rfloor}_{i=1} {\lambda_i} \, g_i = 0$ and $\sum^{m}_{j=1} \mu_j \, g_j = 0$ for some vector $\lambda \in \R^{\lfloor m/2 \rfloor}$ with $\lambda > 0$ and $\sum^{\lfloor m/2 \rfloor}_{i=1} \lambda_i = 1$, and $\mu \in \R^m$ with $\mu \neq 0$ and $\sum^m_{j=1} \mu_j = 0$. Under this setup, $x^\ast = 0$ is the unique minimizer, and all assumptions in Corollary \ref{cor:local_linear_converge} hold. Thus, by Corollary \ref{cor:local_linear_converge}, SRDescent-adapt is expected to converge linearly. Meanwhile, the affine independence and non-degeneracy requirements in~\cite{helou2017local,davis2024local,kong2024lipschitz} do not hold here, so NTDescent and GS lack theoretical linear convergence guarantees for these instances.

\begin{figure}[!ht]
    \centering
    \begin{minipage}[b]{0.24\textwidth}
        \centering
        \scriptsize{(a) $(n, m) = (200, 10)$} \\[0.5em]
        \includegraphics[width=1\textwidth]{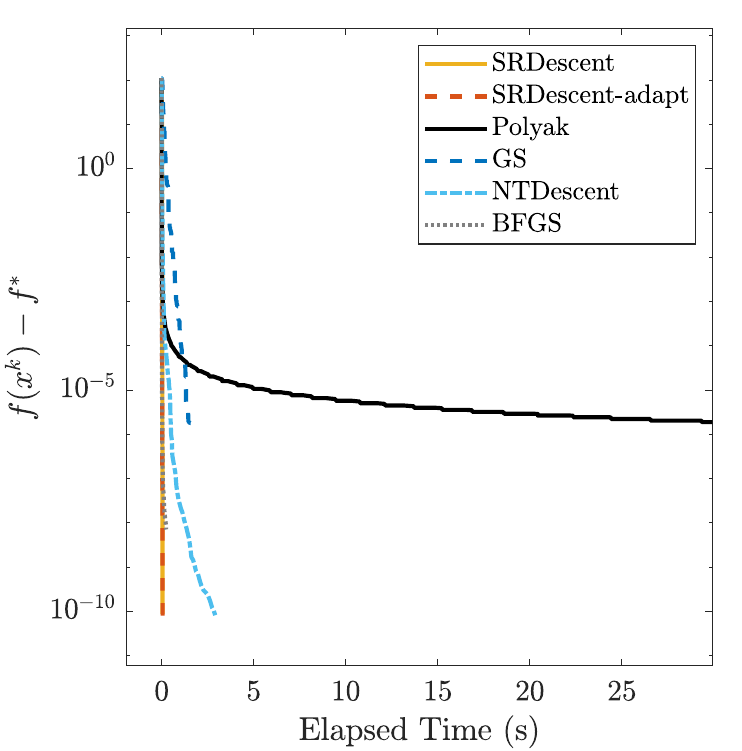}\\
        \includegraphics[width=1\textwidth]{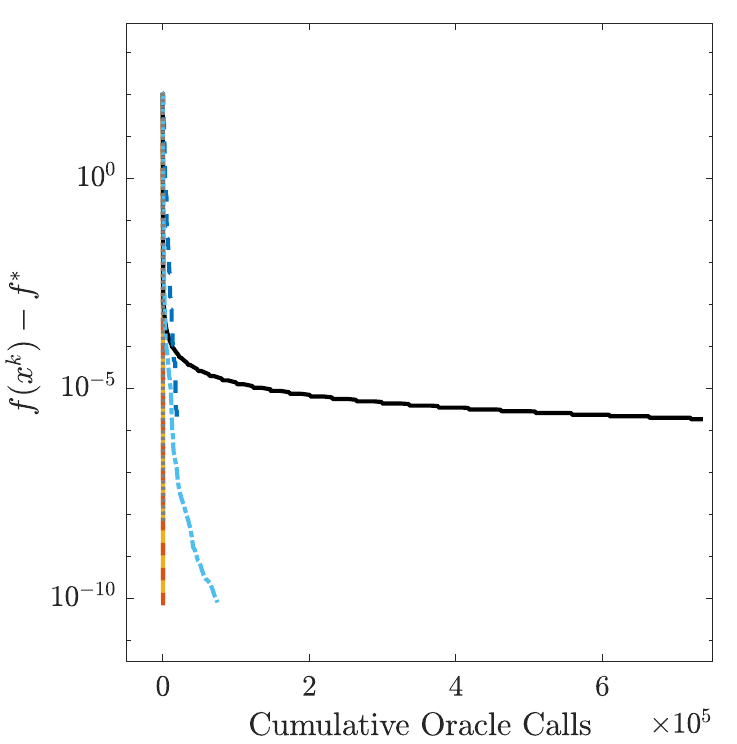}
    \end{minipage}
    \begin{minipage}[b]{0.24\textwidth}
        \centering
        \scriptsize{(b) $(n, m) = (200, 50)$} \\[0.5em]
        \includegraphics[width=1\textwidth]{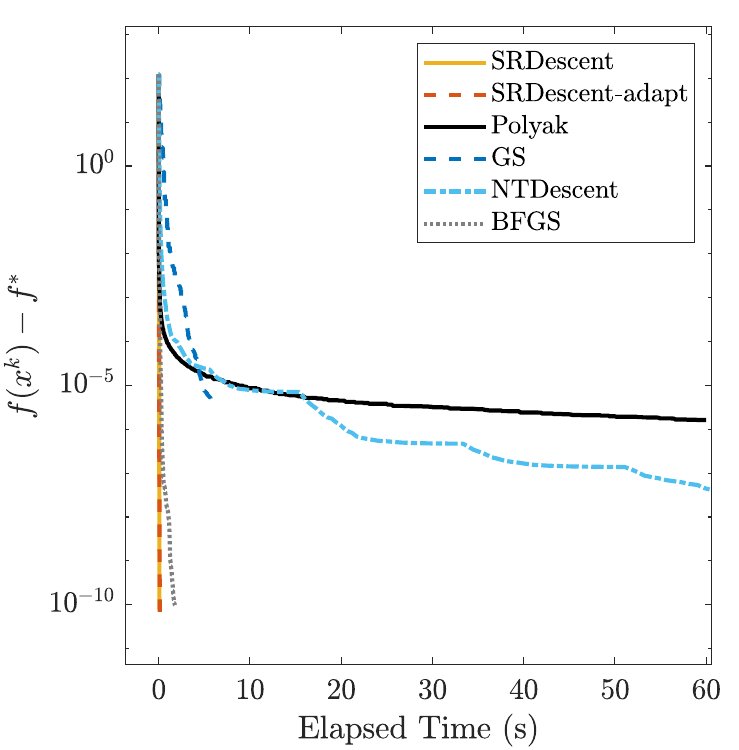}\\
        \includegraphics[width=1\textwidth]{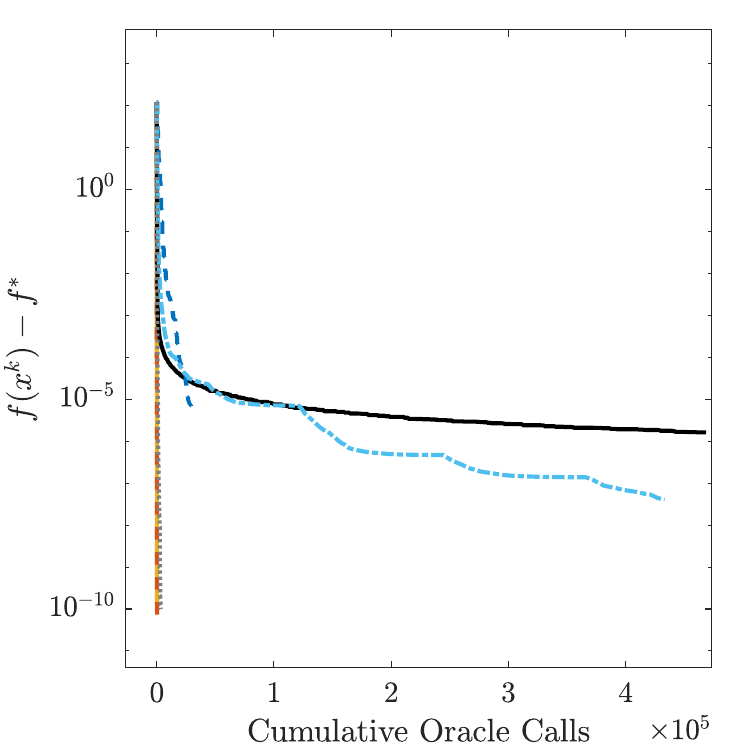}
    \end{minipage}
    \begin{minipage}[b]{0.24\textwidth}
        \centering
        \scriptsize{(c) $(n, m) = (200, 100)$} \\[0.5em]
        \includegraphics[width=1\textwidth]{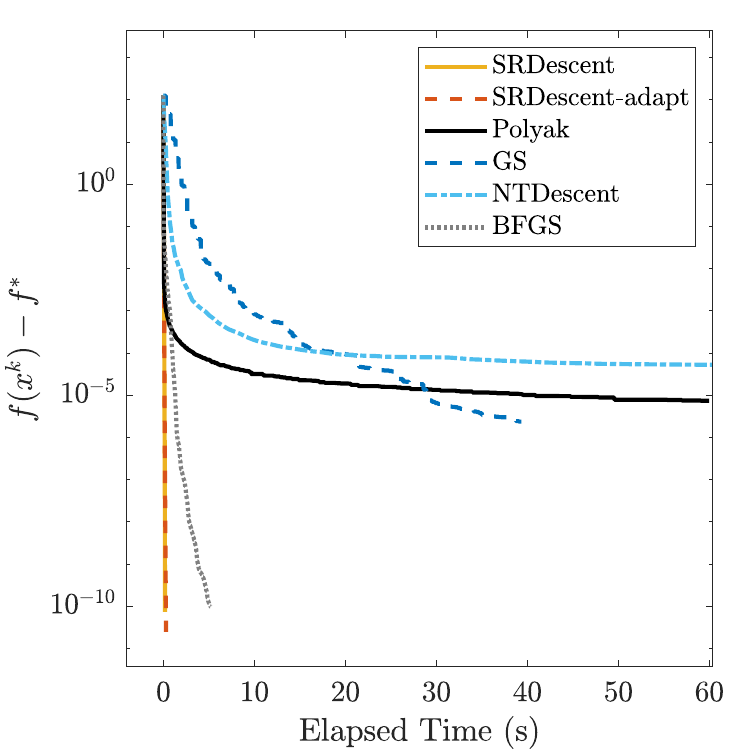}\\
        \includegraphics[width=1\textwidth]{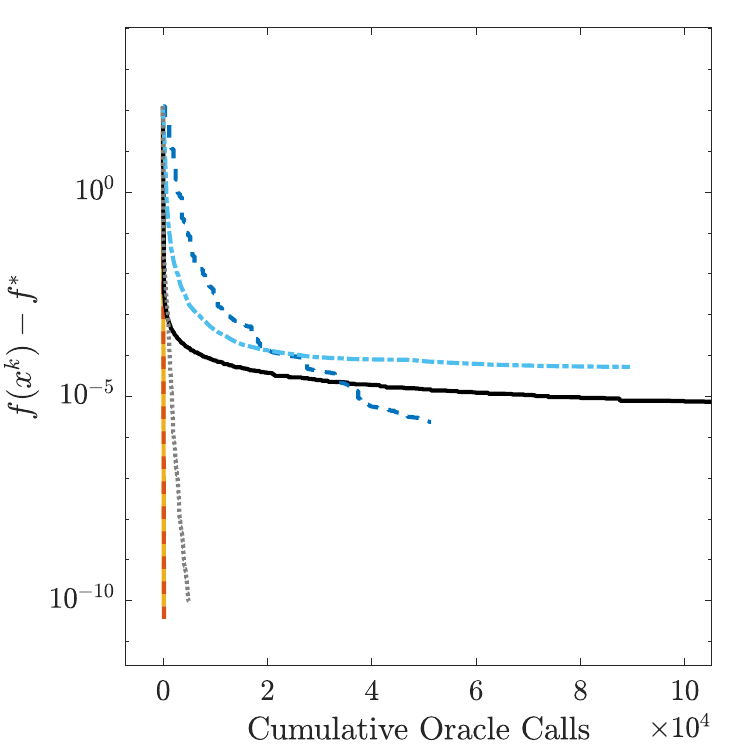}
    \end{minipage}
    \begin{minipage}[b]{0.24\textwidth}
        \centering
        \scriptsize{(d) $(n, m) = (200, 200)$} \\[0.5em]
        \includegraphics[width=1\textwidth]{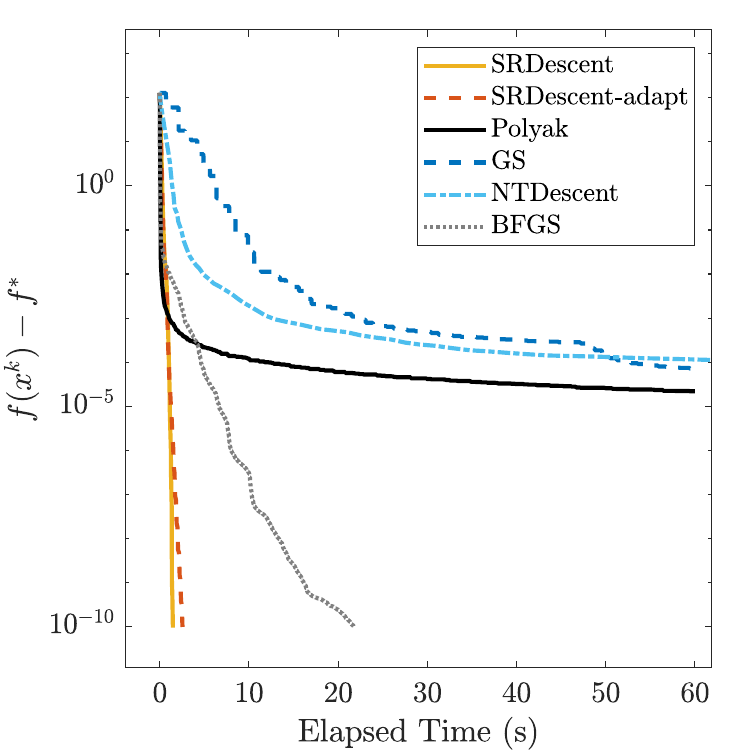}\\
        \includegraphics[width=1\textwidth]{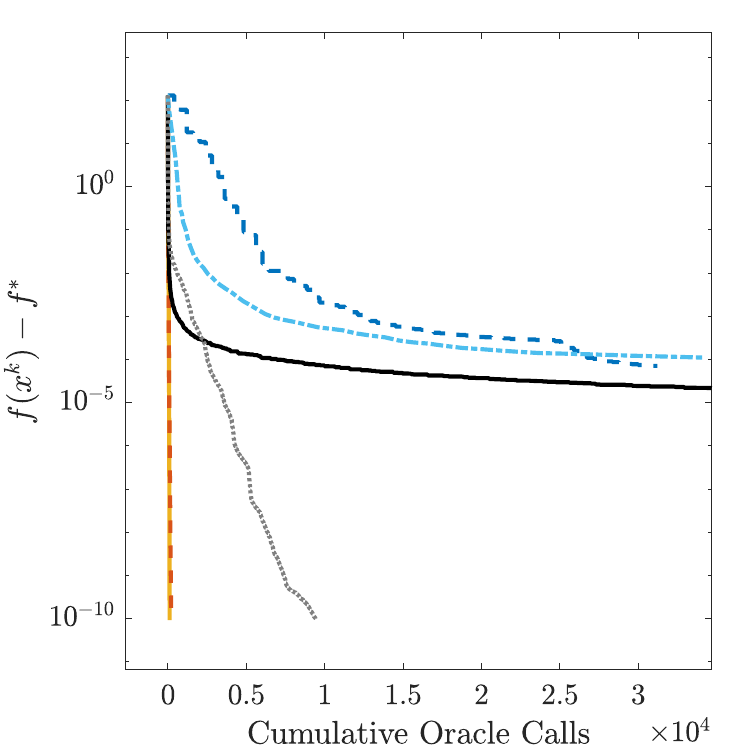}
    \end{minipage}
    \caption{\small Performance on the max of convex quadratics functions with $n=200$ and $m \in \{10, 50, 100, 200\}$, all initialized from the same random points. For the Polyak method, we plot the best objective value after $k$ oracle calls, i.e., $\min_{1 \leq i \leq k} f(x^k)$ rather than $f(x^k)$. In cases $m \in \{10, 50, 100\}$, GS terminates early due to the line search failures.}
    \label{fig:max_of_smooth_vary_nb_functions}
\end{figure}

We generate test problems by fixing $n = 200$ and varying the number of convex quadratic functions $m \in \{10, 50, 100, 200\}$. Figure \ref{fig:max_of_smooth_vary_nb_functions} reports the results. We excluded PBMDC because of its unstable behavior, converging quickly to moderate accuracy but significantly slow afterward; for $(n, m) = (200, 50)$, it already becomes slow near $10^{-7}$ and behaves worse as $m$ increases. Polyak method converges sublinearly, serving as a baseline. Both SRDescent and SRDescent-adapt exhibit linear convergence and outperform other methods, including BFGS, for larger $m$. Goldstein-type methods struggle when $m$ grows, possibly due to the bottleneck of sampling more pieces of functions to approximate the Goldstein direction. In addition, GS encounters numerical issues in several cases due to the line search failures, making it less reliable for producing high quality solutions.

\subsection{Finite max of nonconvex quadratic functions}

We consider Nesterov's nonsmooth Chebyshev-Rosenbrock function:
\begin{equation}\label{eq:Nesterov_nonsmooth}
    \Min_{x \in \R^n} \left\{f(x) = \frac{1}{4}(x_{1}- 1)^{2} + \sum^{n-1}_{i=1} \left|x_{i+1}- 2 (x_i)^{2} + 1\right| \right\},
\end{equation}
which is known for its challenging landscape~\cite{gurbuzbalaban2012nesterov}. This nonconvex function has a unique global minimizer $x^\ast = (1, 1, \cdots, 1 )$, which is also the only stationary point~\cite[Theorem 1]{gurbuzbalaban2012nesterov}. Although $f$ is partly smooth with respect to the manifold $\mathcal{M}\triangleq \big\{x \mid x_{i+1} - 2 (x_{i})^2 + 1 = 0, i = 1, \cdots, n-1 \big\}$, it is not an active manifold around $x^\ast$ in the sense of~\cite{davis2024local} and therefore, it is not clear whether NTDescent locally converges nearly linearly for this problem.

As numerically shown in~\cite{gurbuzbalaban2012nesterov}, this problem becomes increasingly difficult as $n$ increases. For large $n$, methods like BFGS result in low accuracy because line search breaks down near $\mathcal{M}$ but not close enough to $\bx$.
In our experiments, each algorithm is run until it either achieves an objective gap $f(x^k) - f^\ast$ below an accuracy of $10^{-2}$ or $10^{-5}$, or exceeds a time limit of $1,000$ seconds. For each dimension $n$, we test $10$ random initial points and consider any run that does not meet the target accuracy as a failure. Table \ref{tab:nesterov_nonsmooth_comparison1} summarizes the detailed results. When $n=8$ and $10$, all methods except SRDescent and SRDescent-adapt frequently fail to reach the high accuracy within the time limit, while our methods are more robust against the challenges posed by this nonsmooth landscape. Additionally, we can see from Table \ref{tab:nesterov_nonsmooth_comparison1} that SRDescent-adapt consistently converges faster than SRDescent, benefiting from the adaptive adjustment of the regularization parameter $\epsilon_{k,0}$.

\begin{table}[!ht]
    \centering
    \resizebox{0.8\textwidth}{!}
    {
    \begin{tabular}{|c|c||cccc||cccc||}
        \hline
        \multirow{2}{*}{$n$}   & \multirow{2}{*}{method} & \multicolumn{4}{c||}{low accuracy 1e-2} & \multicolumn{4}{c||}{high accuracy 1e-5} \\
        \cline{3-6}\cline{7-10} & 
        & \#fails  & time (s)  & $f(x^k)$  &  oracle calls
        & \#fails  & time (s)  & $f(x^k)$  &  oracle calls \\
        \hline\hline
        \multirow{6}{*}{3}
        
        & PBMDC      & 0   & 5.6e-2           & 7.9e-3           & 6.4e+1
                     & 0   & 6.0e-1           & 9.9e-6           & 4.6e+2 \\
        & GS         & 0   & 1.1e-2           & 8.8e-3           & 5.5e+2
                     & 0   & 1.5e-1           & 9.9e-6           & 1.6e+4 \\
        & NTDescent       & 0   & 5.7e-3           & 8.4e-3           & 2.5e+3
                     & 0   & 2.4e-2           & 8.2e-6           & 3.1e+4 \\
        & BFGS      & 0   & 7.7e-3           & 9.6e-3           & 4.0e+2
                     & 0   & 5.3e-2           & 9.9e-6           & 2.9e+3 \\
        & SRDescent      & 0   & 6.1e-3           & 7.3e-3           & 4.9e+1
                     & 0   & 5.4e-2           & 1.0e-5           & 7.2e+2 \\
        & SRDescent-adapt  & 0 & 8.2e-3 & 7.8e-3 & 1.0e+2
                     & 0 & 5.1e-2 & 9.9e-6 & 7.1e+2 \\
        \hline\hline
        \multirow{11}{*}{5}
   
        & \multirow{2}{*}{PBMDC}      & \multirow{2}{*}{0}   & \multirow{2}{*}{1.6e+0}           & \multirow{2}{*}{8.0e-3}           & \multirow{2}{*}{1.4e+3}
                     & \multirow{2}{*}{10}  & --      & --    & --      \\
        & & & & & & &  (1.0e+3) &  (1.9e-4) & (5.1e+4) \\ 
        \cdashline{2-10}
        & \multirow{2}{*}{GS}         & \multirow{2}{*}{0}   & \multirow{2}{*}{1.1e-1}           & \multirow{2}{*}{9.8e-3}           & \multirow{2}{*}{1.5e+4}
        & \multirow{2}{*}{6}   & 6.6e+0   & 1.0e-5   & 9.8e+5  \\
                     
        & & & & & & &  (4.9e+0) &  (1.6e-5) & (7.1e+5) \\
        \cdashline{2-10}

        & \multirow{2}{*}{NTDescent}      & \multirow{2}{*}{0}   & \multirow{2}{*}{1.0e-1}           & \multirow{2}{*}{9.4e-3}           & \multirow{2}{*}{1.7e+5}
                     & \multirow{2}{*}{0}   & \multirow{2}{*}{1.1e+2}           & \multirow{2}{*}{9.5e-6}           & \multirow{2}{*}{2.0e+8} \\[0.12in]
        \cdashline{2-10}
        & \multirow{2}{*}{BFGS}      & \multirow{2}{*}{1}   & 3.5e-1   & 8.9e-3   & 1.6e+4 
                     & \multirow{2}{*}{10}  & --      & --     & --     \\

& & & (5.3e-1) & (1.4e-1) & (2.4e+4) & &  (2.3e+0) &   (1.7e-2) & (8.2e+4)  \\ 
        \cdashline{2-10}
        & \multirow{2}{*}{SRDescent}      & \multirow{2}{*}{0}   & \multirow{2}{*}{3.5e-2}           & \multirow{2}{*}{7.9e-3}           & \multirow{2}{*}{5.5e+2}
                     & \multirow{2}{*}{0}   & \multirow{2}{*}{1.5e+0}           & \multirow{2}{*}{9.1e-6}           & \multirow{2}{*}{2.4e+4} \\[0.12in]
        \cdashline{2-10}
        & \multirow{2}{*}{SRDescent-adapt}  & \multirow{2}{*}{0}   & \multirow{2}{*}{5.1e-2}           & \multirow{2}{*}{7.9e-3}           & \multirow{2}{*}{8.5e+2}
                     & \multirow{2}{*}{0}   & \multirow{2}{*}{2.4e+0}           & \multirow{2}{*}{9.1e-6}           & \multirow{2}{*}{4.7e+4} \\[0.12in]
        \hline\hline
        \multirow{12}{*}{8}
        & \multirow{2}{*}{PBMDC}      & \multirow{2}{*}{8}   & 2.4e-2  & 7.5e-3   & 3.6e+1 
                     & \multirow{2}{*}{10}  & --       & --      & --   \\
& & & (1.0e+3)  & (1.6e-2) & (2.9e+5) & &  (1.0e+3) &    (1.4e-2) & (2.4e+5) \\ 
\cdashline{2-10}

        & \multirow{2}{*}{GS}         & \multirow{2}{*}{0}   & \multirow{2}{*}{9.9e+0}           & \multirow{2}{*}{9.7e-3}           & \multirow{2}{*}{1.7e+6}
                     & \multirow{2}{*}{10}  & --       & --     & --      \\

& & & & & & &  (5.5e+1) &    (9.2e-4) & (8.1e+6) \\ 
\cdashline{2-10}
                     
        & \multirow{2}{*}{NTDescent}       & \multirow{2}{*}{0}   & \multirow{2}{*}{2.6e+2}           & \multirow{2}{*}{9.2e-3}           & \multirow{2}{*}{4.6e+8} 
                     & \multirow{2}{*}{9}   & 1.5e-2   & 9.9e-6   & 1.6e+4  \\

& & & & & & &  (1.0e+3) &  (2.6e-3) & (1.7e+9) \\ 
\cdashline{2-10}

& \multirow{2}{*}{BFGS}      & \multirow{2}{*}{6}   & 1.2e-2   & 7.9e-3   & 7.5e+1 
                     & \multirow{2}{*}{10}  & --       & --      & --      \\

& & & (8.9e+0) & (2.8e-1) & (2.6e+5) & &  (6.8e+0) &   (1.7e-1) & (1.8e+5) \\ 
        \cdashline{2-10}
        & \multirow{2}{*}{SRDescent}      & \multirow{2}{*}{0}   & \multirow{2}{*}{1.7e+0}           & \multirow{2}{*}{7.4e-3}           & \multirow{2}{*}{2.6e+4}
                     & \multirow{2}{*}{0}   & \multirow{2}{*}{1.1e+2}           & \multirow{2}{*}{9.2e-6}           & \multirow{2}{*}{1.7e+6} \\[0.12in]

                     \cdashline{2-10}
        & \multirow{2}{*}{SRDescent-adapt}  & \multirow{2}{*}{0}   & \multirow{2}{*}{2.2e+0}           & \multirow{2}{*}{7.4e-3}           & \multirow{2}{*}{4.0e+4}
                     & \multirow{2}{*}{0}   & \multirow{2}{*}{4.0e+1}          & \multirow{2}{*}{9.2e-6}       & \multirow{2}{*}{6.2e+5} \\[0.12in]
        \hline\hline
        \multirow{12}{*}{10}
        & \multirow{2}{*}{PBMDC}      & \multirow{2}{*}{10}  & --       & --      & --    
                     & \multirow{2}{*}{10}  & --       & --       & --     \\

& & & (1.0e+3) &  (1.3e-1) &  (6.5e+4) & &  (1.0e+3) &  (1.3e-1) & (6.4e+4)  \\ 
\cdashline{2-10}

        & \multirow{2}{*}{GS}         & \multirow{2}{*}{7}   & 6.6e+1   & 9.1e-3   & 1.2e+7 
                     & \multirow{2}{*}{10}  & --       & --       & --     \\

& & & (1.1e+2) & (1.9e-2) &  (1.9e+7) & & (1.0e+2) &  (1.5e-2) &  (1.8e+7)  \\ 
\cdashline{2-10}

& \multirow{2}{*}{NTDescent}       & \multirow{2}{*}{7}   & 1.2e-2   & 8.6e-3  & 3.3e+3 
                     & \multirow{2}{*}{10}  & --       & --     & --      \\

& & & (1.0e+3) &  (4.2e-1)  &  (1.7e+9) & &  (1.0e+3) &    (3.0e-1) & (1.7e+9)  \\ 
\cdashline{2-10}

        & \multirow{2}{*}{BFGS}      & \multirow{2}{*}{6}   & 5.5e-3   & 9.1e-3   & 1.0e+2 
                     & \multirow{2}{*}{10}  & --       & --      & --    \\

& & & (1.0e+1) &  (3.0e-1)  &  (2.5e+5) & &  (7.6e+0)  &     (1.8e-1) & (1.8e+5)  \\ 
\cdashline{2-10}
                     
& \multirow{2}{*}{SRDescent}      & \multirow{2}{*}{0}   & \multirow{2}{*}{3.0e+1}           & \multirow{2}{*}{7.8e-3}           & \multirow{2}{*}{4.7e+5}
                     & \multirow{2}{*}{8}   & 4.0e+2   & 6.6e-6   & 6.1e+6  \\

& & & & & & & (1.0e+3) &  (5.5e-5) & (1.5e+7) \\ 
\cdashline{2-10}

& \multirow{2}{*}{SRDescent-adapt}  & \multirow{2}{*}{0}   & \multirow{2}{*}{2.4e+1}           & \multirow{2}{*}{7.8e-3}           & \multirow{2}{*}{4.1e+5}
                     & \multirow{2}{*}{0}   & \multirow{2}{*}{6.4e+2}  & \multirow{2}{*}{9.3e-6}  & \multirow{2}{*}{9.5e+6} \\[0.12in]
        \hline
    \end{tabular}%
    }
    \caption{\small Comparison on Nesterov's nonsmooth Chebyshev-Rosenbrock function \eqref{eq:Nesterov_nonsmooth} over $10$ random initial points. The column ``\#fails" counts the number of runs (out of $10$) that do not reach the target accuracy within $1,000$ seconds, or terminate early due to the line search failures. For successful runs, the average runtime, final objective values, and number of oracle calls are reported without parentheses; parenthetical values represent averages over failed runs.}
    \label{tab:nesterov_nonsmooth_comparison1}
\end{table}

Beyond algorithms based on Oracle 1 in a black-box setting, we also compare with LiPsMin~\cite{fiege2019algorithm}, which explicitly uses the kink structure to build piecewise linear approximations. The results are summarized in Table~\ref{tab:nesterov_nonsmooth_comparison2}, where the results of LiPsMin are taken from~\cite{fiege2019algorithm}. SRDescent and SRDescent-adapt achieve lower objective values with a comparable number of gradient evaluations as LiPsMin. For $n=20$, our methods using additional objective and gradient evaluations successfully escape the ``pseudo stationary value" of $0.81814$ observed in~\cite{fiege2019algorithm} for both LiPsMin and \textsc{HANSO}.

\begin{table}[H]
    \centering
    \resizebox{0.8\textwidth}{!}
    {
    \begin{tabular}{|c||ccc|ccc|ccc|}
    \hline
    \multirow{2}{*}{method} & \multicolumn{3}{c|}{$n=5$} & \multicolumn{3}{c|}{$n=10$} & \multicolumn{3}{c|}{$n=20$} \\
    \cline{2-10}
    & $f(x^k)$ & \#$f$ & \#$\nabla f$ & $f(x^k)$ & \#$f$ & \#$\nabla f$ & $f(x^k)$ & \#$f$ & \#$\nabla f$ \\
    \hline \hline
    LiPsMin 
    & 6.4e-2 & 1001 & 15095 & 0.81744 & 1001 & 30371 & 0.81814 & 1001 & 42150 \\
    \hline
    \multirow{2}{*}{SRDescent}
    & 5.4e-5  & 3949 & 15092 & 0.80087 & 7768 & 30368 & 0.81814 & 2218 & 42142 \\
    & 2.4e-7 & 1.3e+5 & 5.0e+5 & 2.7e-5 & 1.1e+7 & 1.0e+8 & 0.81804 & 5.3e+7 & 1.0e+9 \\
    \hline
    \multirow{2}{*}{SRDescent-adapt} 
    & 2.1e-5  & 3975 & 15092 & 0.7914 & 7792 & 30368 & 0.81814 & 2204 & 42142 \\
    & 4.2e-10 & 1.3e+5 & 5.0e+5 & 2.6e-6 & 1.1e+7 & 1.0e+8 & 0.81773 & 5.3e+7 & 1.0e+9 \\
    \hline
    \end{tabular}%
    }
    \caption{\small Comparison on Nesterov's nonsmooth Chebyshev-Rosenbrock function \eqref{eq:Nesterov_nonsmooth} with a fixed initial point $x^0$, where $x^0_i = 0.5$ for odd $i$ and $x^0_i = -0.5$ otherwise. ``$\# f$'' and ``$\# \nabla f$'' denote the number of objective and gradient evaluations. For SRDescent(-adapt), the first row reports the final objective using a comparable number of gradient evaluations as LiPsMin, while the second row shows results obtained with additional objective and gradient evaluations.}
    \label{tab:nesterov_nonsmooth_comparison2}
\end{table}

\subsection{Finite min of convex quadratic functions}
In this part, we test our algorithms on the finite min of convex quadratic functions:
\[
    f(x) = \min_{1 \leq i \leq m} \frac{1}{2} \|A_i \, x - b_i\|^2,
\]
where each $A_i \in \R^{d \times n}$ with $d \geq n$ has linearly independent columns, and $b_i = A_i \, x^\ast$ for a solution $x^\ast \sim \mathcal{N}\big(0, n^{-1/2} I_n\big)$. By construction, $x^\ast$ is a global minimizer with $f^\ast = 0$, which is also the unique stationary point. This function is not subdifferentially regular~\cite[Definition 7.25]{rockafellar2009variational} at points where multiple quadratic components are active, creating the so-called ``downward cusps".

We fix $n = d = 300$ and vary the number of pieces $m$. Each algorithm is tested from the same random initial point. Figure \ref{fig:min_of_smooth_vary_nb_functions1} summarizes run time and oracle calls averaged over $100$ random instances, and Figure \ref{fig:min_of_smooth_vary_nb_functions2} plots results on a random instance for $m \in \{10, 50, 100\}$. 
SRDescent(-adapt) and BFGS show the best overall performance. Their oracle calls required for achieving a target accuracy of $10^{-8}$ remain stable across different $m$. However, BFGS tends to be more efficient than SRDescent(-adapt) as $m$ increases to $350$. While Polyak and \textsc{GS} perform well for small $m$, their results start to have high deviations for large $m$. In Figure \ref{fig:min_of_smooth_vary_dimensions}, we also include an additional comparison for $m=50$ and varying dimensions $n$.

We remark that although the finite min of quadratic functions is not subdifferentially regular near the minimizer $x^\ast$, gradients of all component quadratic functions vanish at this point. This suggests that, unlike in the case of the finite max of quadratics, combining gradients of all active components may be unnecessary. Nonetheless, the observed slow convergence of NTDescent and GS indicates additional challenges inherent to this class of nonsmooth functions. Further investigation into methods for handling non-regular functions is left for future work.

\begin{figure}[!ht]
    \centering
    \begin{minipage}[b]{0.3\textwidth}
        \centering
        \includegraphics[width=1\textwidth]{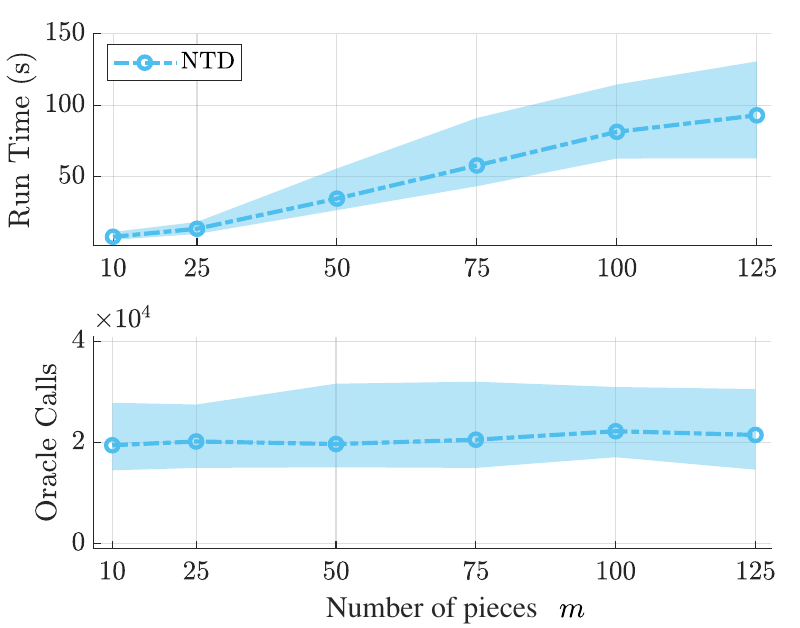}
    \end{minipage}
    \hspace{1em}
    \begin{minipage}[b]{0.3\textwidth}
        \centering
        \includegraphics[width=1\textwidth]{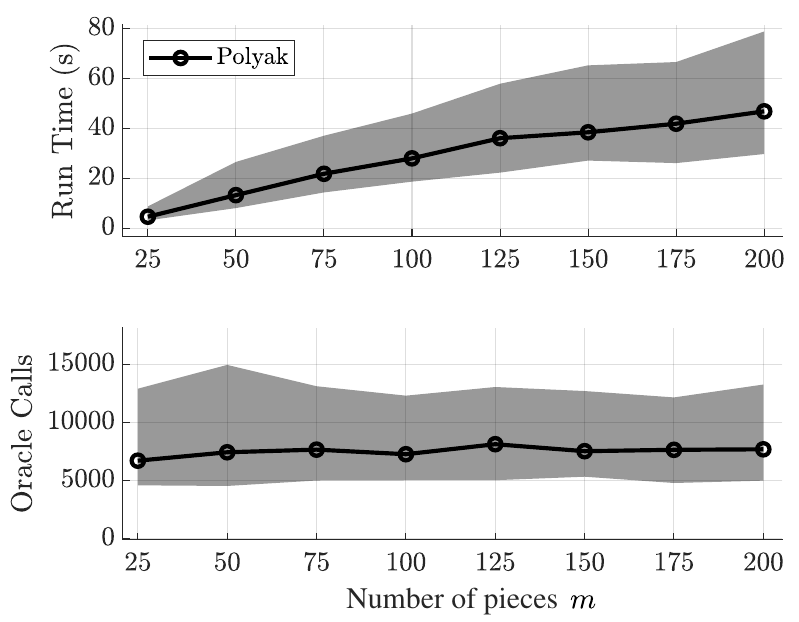}
    \end{minipage}
    \hspace{1em}
    \begin{minipage}[b]{0.3\textwidth}
        \centering
        \includegraphics[width=1\textwidth]{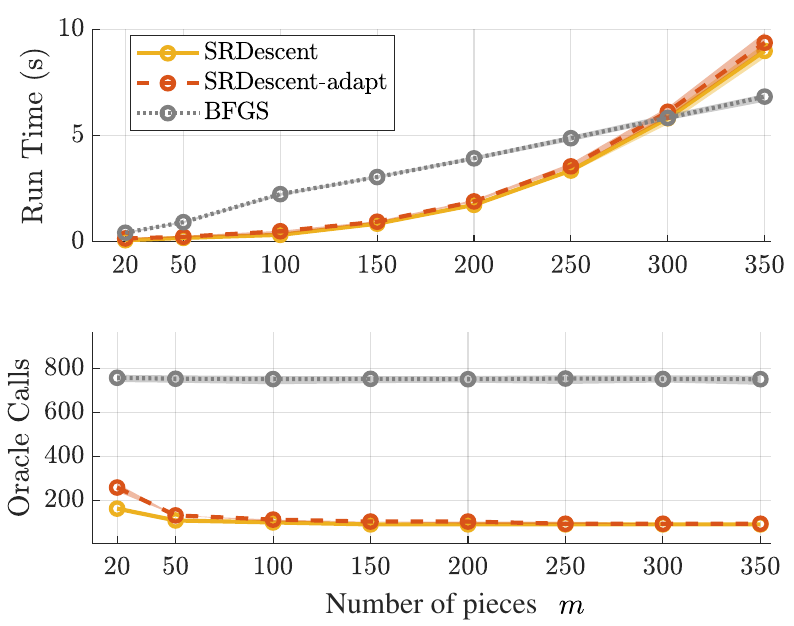}
    \end{minipage}
    \caption{\small Run time and oracle calls on the finite min of convex quadratic functions over $100$ randomly generated instances with $n=d=300$ and varying numbers of pieces $m$. The lines represent the median and shaded areas indicate inter-quartiles. NTDescent terminates when the objective gap $f(x^k) - f^\ast$ is below $10^{-6}$ and other methods terminate when the objective gap is below $10^{-8}$.}
\label{fig:min_of_smooth_vary_nb_functions1}
\end{figure}

\begin{figure}[!ht]
    \centering
    \begin{minipage}[b]{0.24\textwidth}
        \centering
        \scriptsize{(a) $(n, m) = (300, 10)$} \\[0.5em]
        \includegraphics[width=1\textwidth]{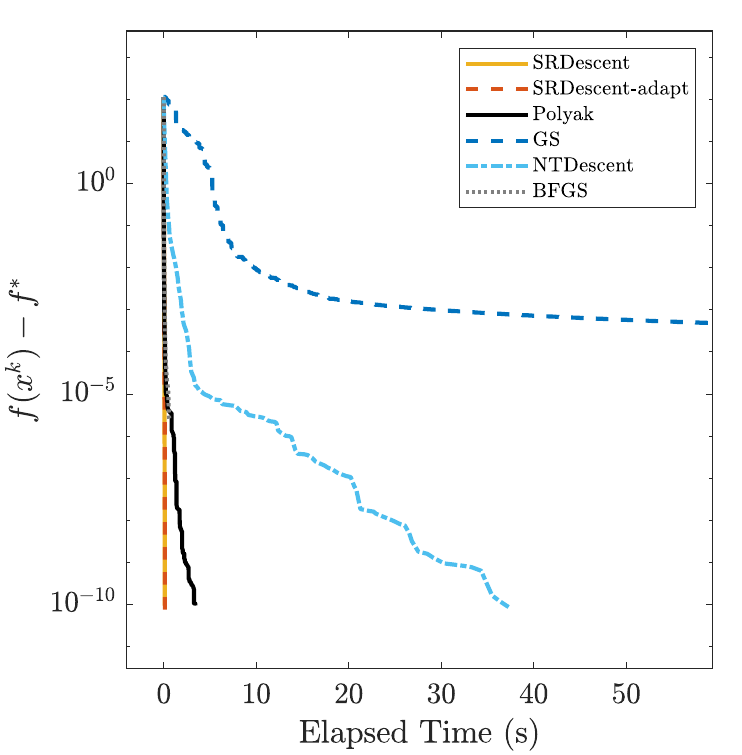}\\
        \includegraphics[width=1\textwidth]{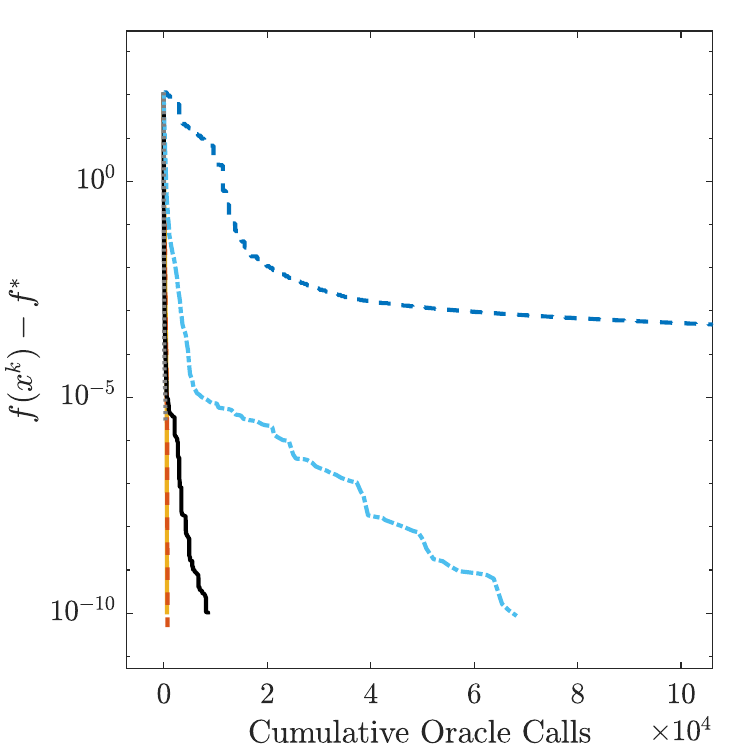}
    \end{minipage}
    \begin{minipage}[b]{0.24\textwidth}
        \centering
        \scriptsize{(b) $(n, m) = (300, 50)$} \\[0.5em]
        \includegraphics[width=1\textwidth]{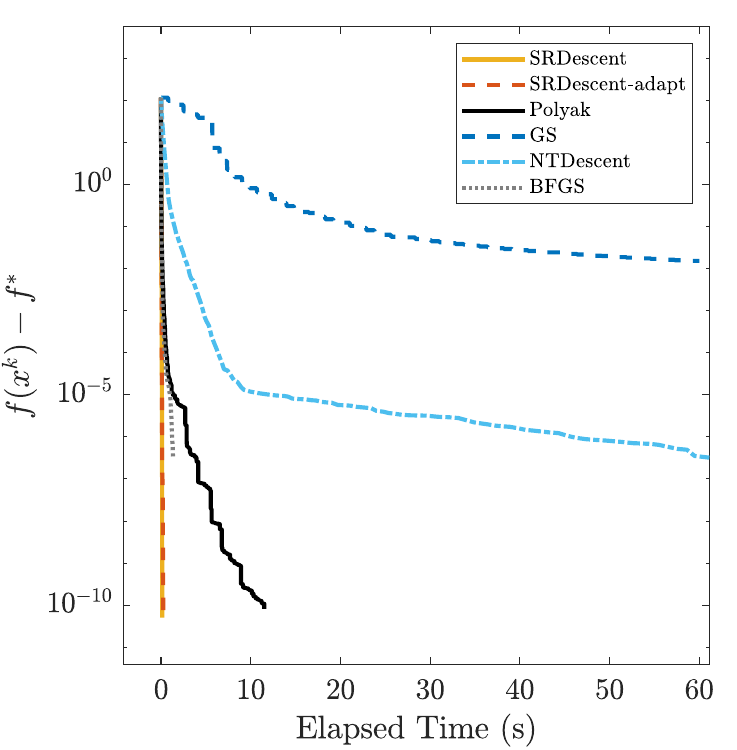}\\
        \includegraphics[width=1\textwidth]{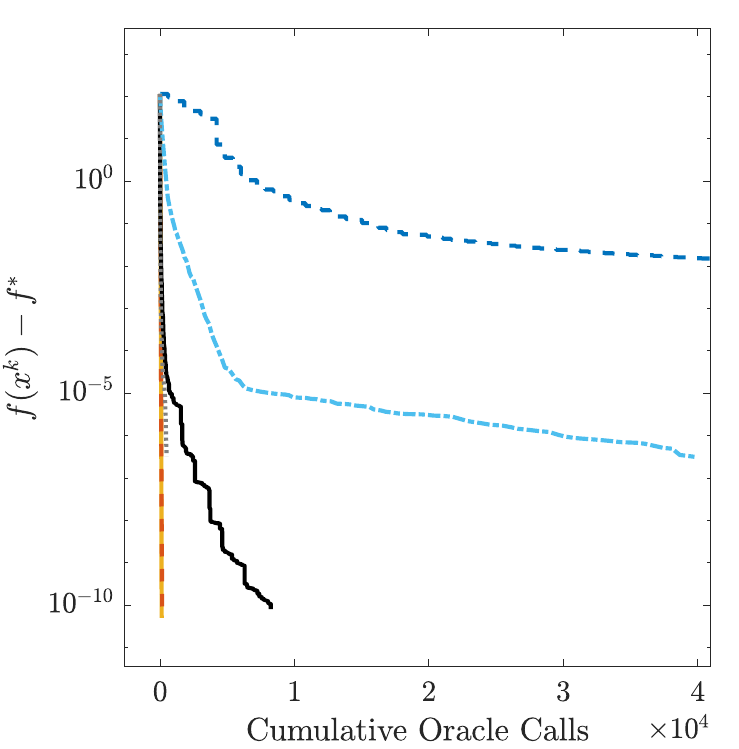}
    \end{minipage}
    \begin{minipage}[b]{0.24\textwidth}
        \centering
        \scriptsize{(c) $(n, m) = (300, 100)$} \\[0.5em]
        \includegraphics[width=1\textwidth]{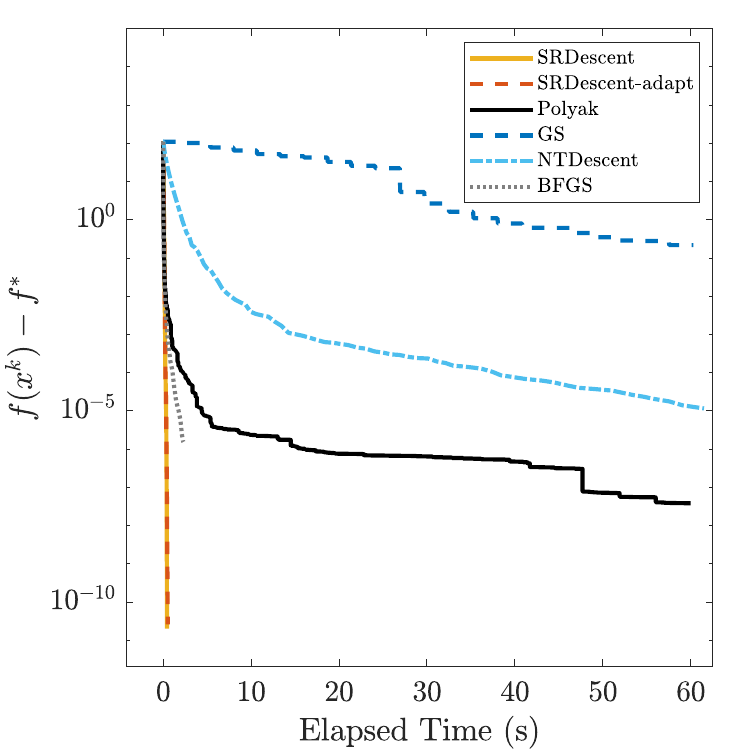}\\
        \includegraphics[width=1\textwidth]{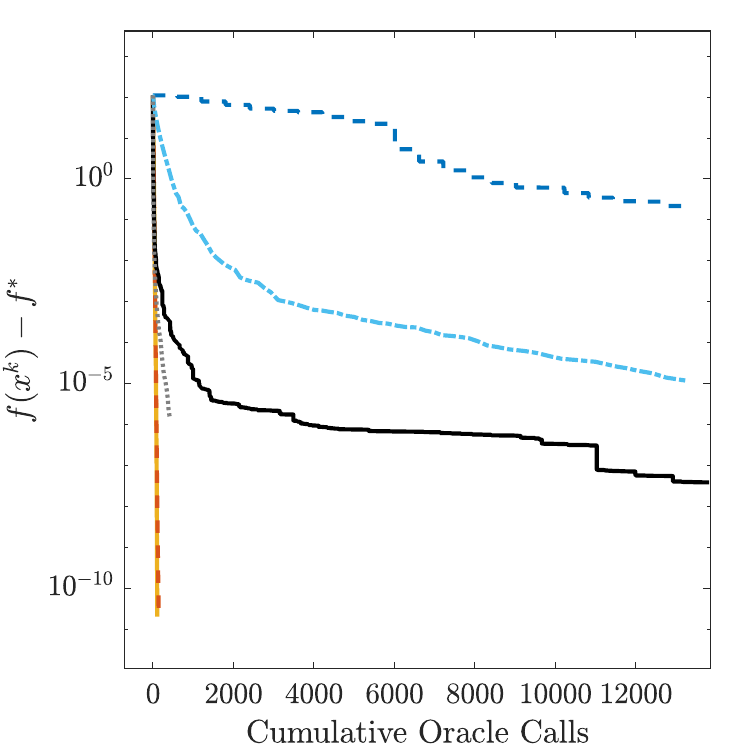}
    \end{minipage}
    \caption{\small Performance on the finite min of convex quadratic functions for $n=d=300$ and varying numbers of pieces $m \in \{10, 50, 100\}$, all initialized from the same randomly generated  points. For Polyak, we report the best objective value after $k$ oracle calls, i.e., $\min_{1 \leq i \leq k} f(x^k)$, rather than $f(x^k)$. For all cases, BFGS terminates early due to the line search failures.}
    \label{fig:min_of_smooth_vary_nb_functions2}
\end{figure}

\begin{figure}[H]
    \centering
    \begin{minipage}{0.24\textwidth}
        \centering
        \scriptsize{(a) $(n, m) = (100, 50)$} \\[0.5em]
        \includegraphics[width=1\textwidth]{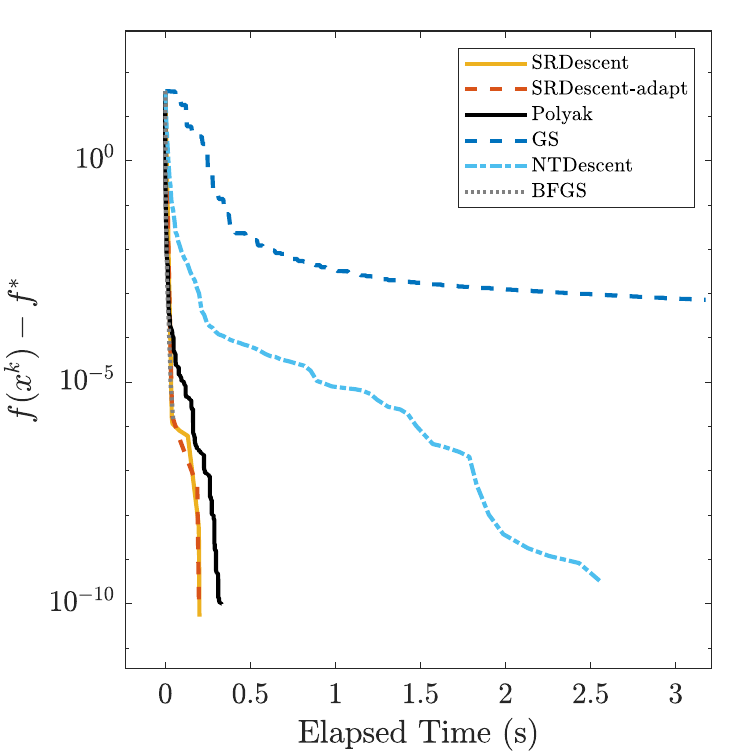}\\
        \includegraphics[width=1\textwidth]{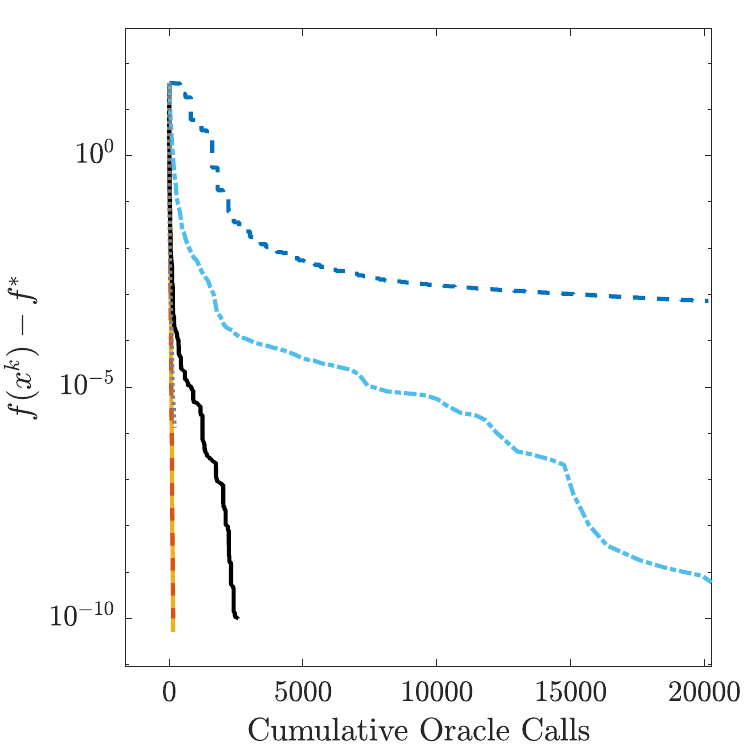}
    \end{minipage}
    \begin{minipage}{0.24\textwidth}
        \centering
        \scriptsize{(b) $(n, m) = (250, 50)$} \\[0.5em]
        \includegraphics[width=1\textwidth]{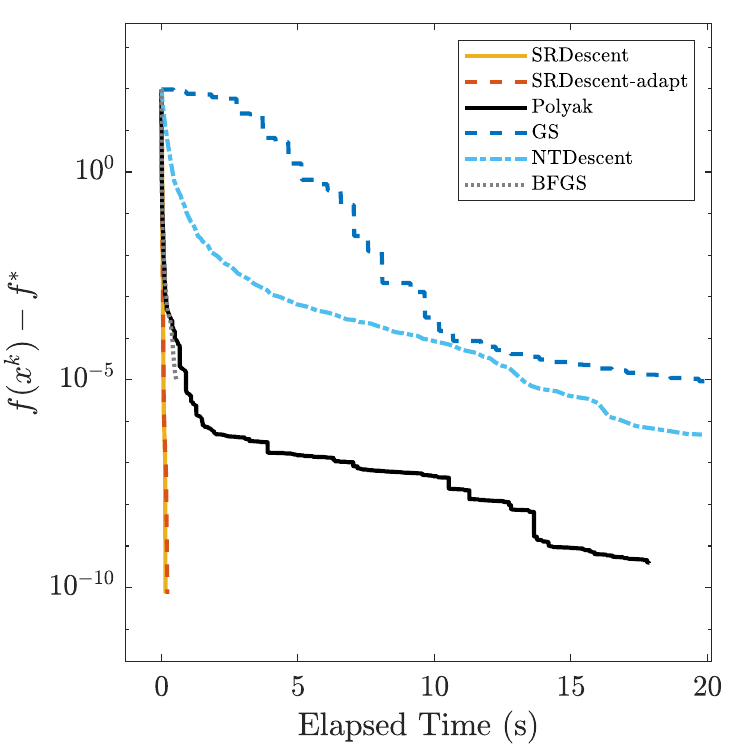}\\
        \includegraphics[width=1\textwidth]{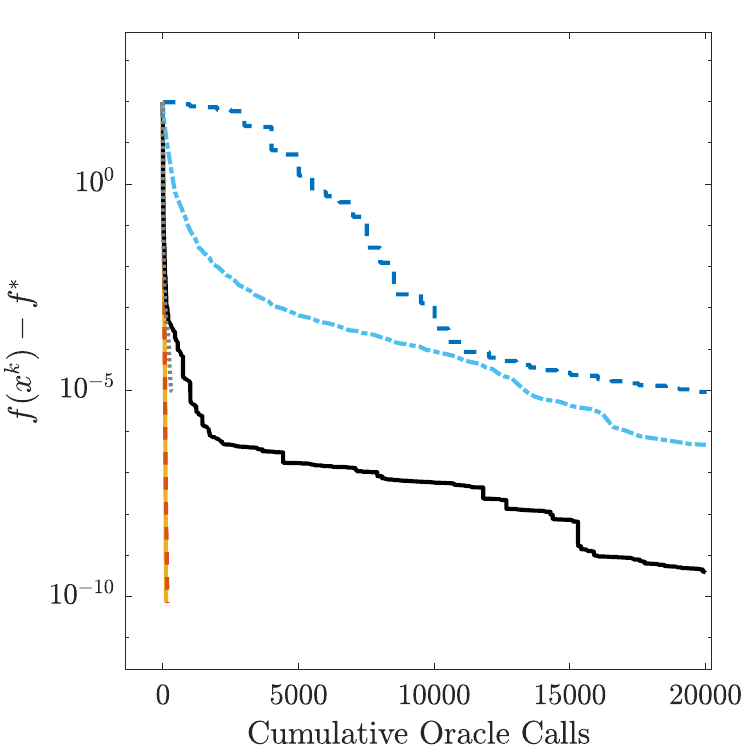}
    \end{minipage}
    \begin{minipage}{0.24\textwidth}
        \centering
        \scriptsize{(c) $(n, m) = (500, 50)$} \\[0.5em]
        \includegraphics[width=1\textwidth]{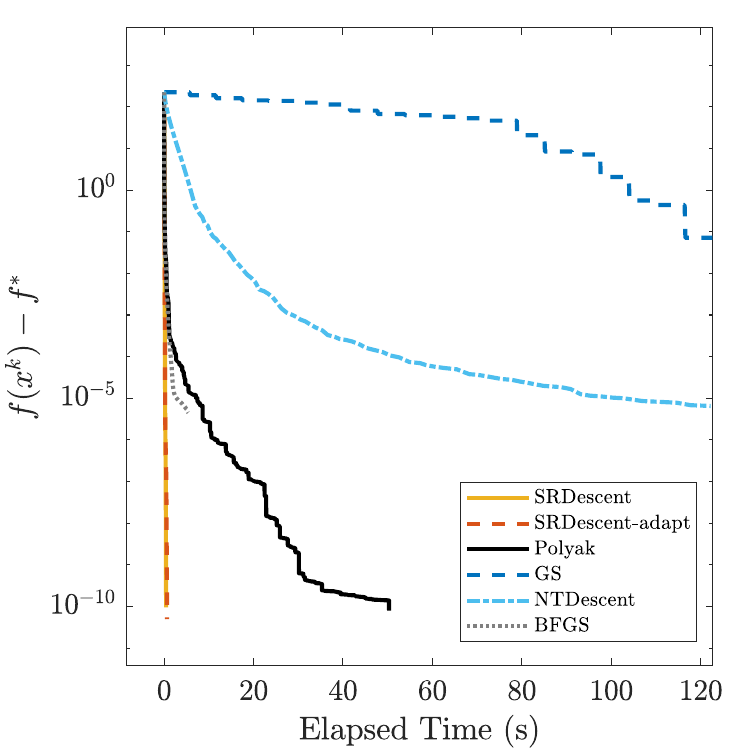}\\
        \includegraphics[width=1\textwidth]{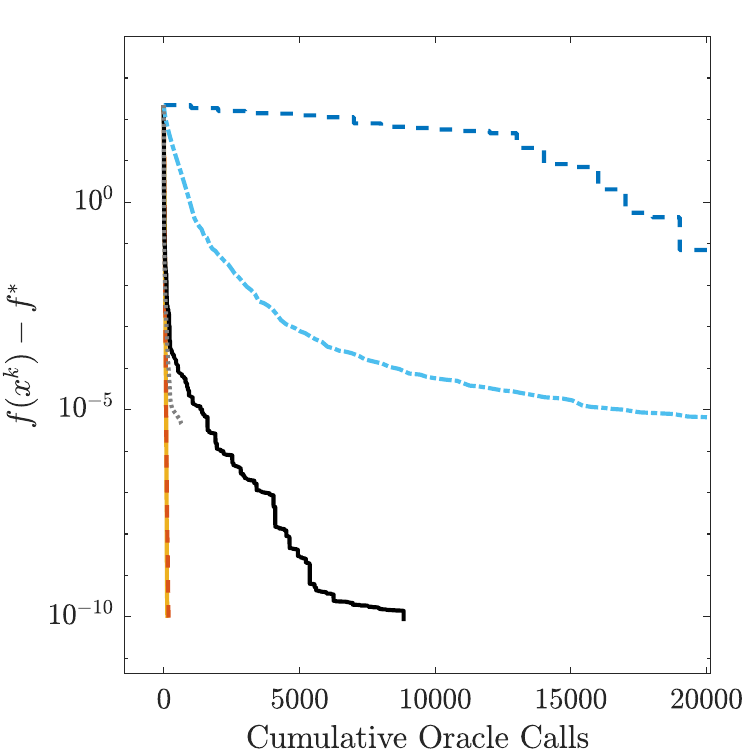}
    \end{minipage}
    \begin{minipage}{0.24\textwidth}
        \centering
        \scriptsize{(d) $(n, m) = (1000, 50)$} \\[0.5em]
        \includegraphics[width=1\textwidth]{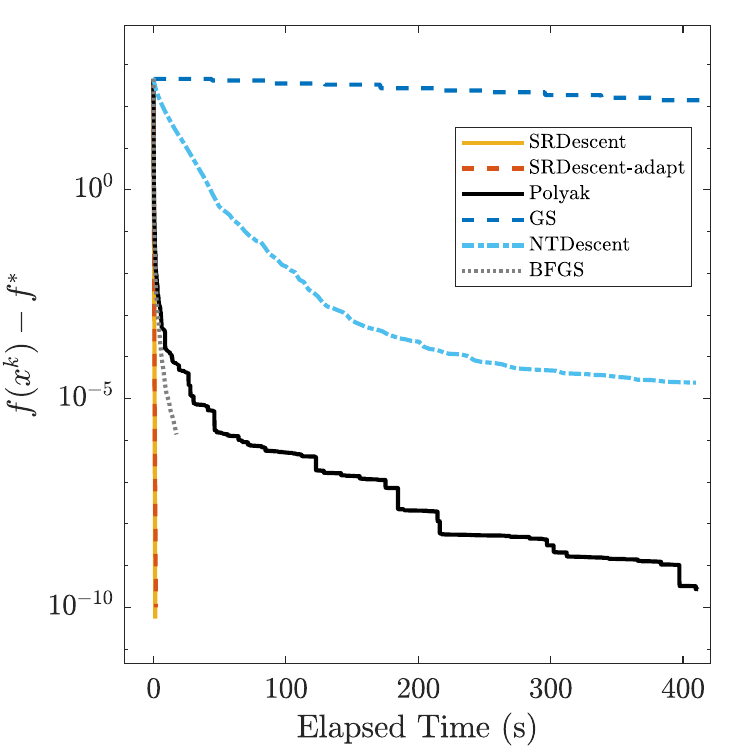}\\
        \includegraphics[width=1\textwidth]{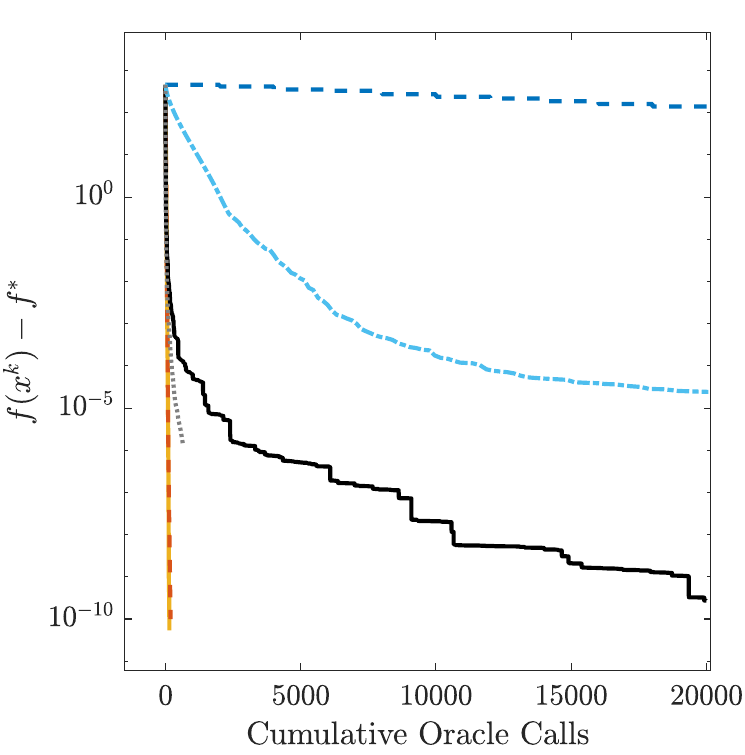}
    \end{minipage}
    \caption{\small Performance on the finite min of convex quadratic functions for a fixed numbers of pieces $m=50$ and varying dimensions $n \in \{100, 250, 500, 1000\}$, all initialized from the same randomly generated  points. For Polyak, we report the best objective value after $k$ oracle calls, i.e., $\min_{1 \leq i \leq k} f(x^k)$, rather than $f(x^k)$. For all cases, BFGS terminates early due to the line search failures.}
    \label{fig:min_of_smooth_vary_dimensions}
\end{figure}

\subsection{Marginal functions with varying feasible sets}

In the final example, we study the minimization of a marginal function given by:
\[
\begin{array}{rl}
    f(x) =& \;\;\;\;\displaystyle\min_{y \in \R^m} \;\;\;\left\{(c + Dx)^\top y + \frac{1}{2} \, y^\top Q y + \|x\|^4\right\} \\[0.1in]
    &\mbox{subject to }\; b - Ax - \mathbf{1} \leq W y \leq b - Ax \,.
\end{array}
\]
Here, the vectors and matrices are generated randomly: the entries of $b, c, A,$ and $D$ are independently drawn from $\mathcal{N}\big(0, m^{-1/2}\big)$, the matrix $W \in \R^{m \times m}$ has linearly independent rows, and $Q$ is positive semi-definite. The vector $\mathbf{1}$ denotes the $m$-dimensional all-ones vector. Under this setup, $(-f)$ satisfies assumptions (i)-(ii), and (iii$^\prime$) LICQ in Section \ref{sec4.2:marginal_vary}. Without the term $\|x\|^4$, the marginal function is known to be piecewise linear-quadratic on its domain of finiteness~\cite[Proposition 5.2.2]{cui2021modern}. The additional term $\|x\|^4$ ensures that $f$ is bounded from below.

We fix $n = 300$ and vary the dimension $m$ of the inner minimization problem. We run SRDescent with optimality tolerances  $\epsilon_\text{tol} = \nu_\text{tol} = 10^{-3}$. Figure \ref{fig:optval_vary_dim} presents the results for run time and oracle calls. We focus solely on SRDescent because different algorithms converge to different stationary values, complicating direct comparisons. Additionally, we observe that SRDescent-adapt exhibits performance nearly identical to that of SRDescent. As shown in Figure \ref{fig:optval_vary_dim}, increasing $m$ leads to a notable increase in run time and oracle calls. This indicates that the problem becomes more challenging as the dimension of the inner minimization grows.

\begin{figure}[!ht]
\centering
\begin{minipage}[b]{0.3\textwidth}
    \centering
    \includegraphics[width=\textwidth]{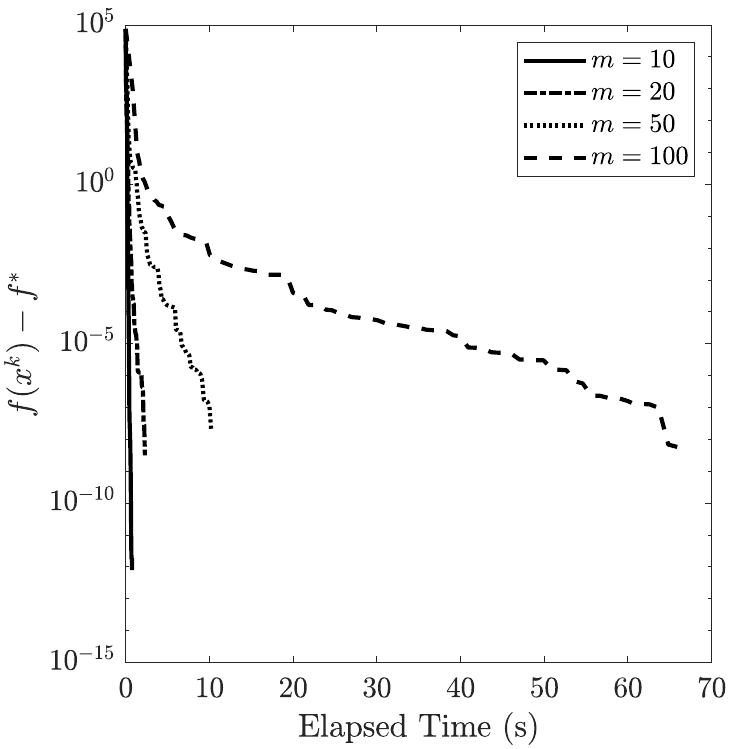}
\end{minipage}
\hspace{3em}
\begin{minipage}[b]{0.31\textwidth}
    \centering
    \includegraphics[width=\textwidth]{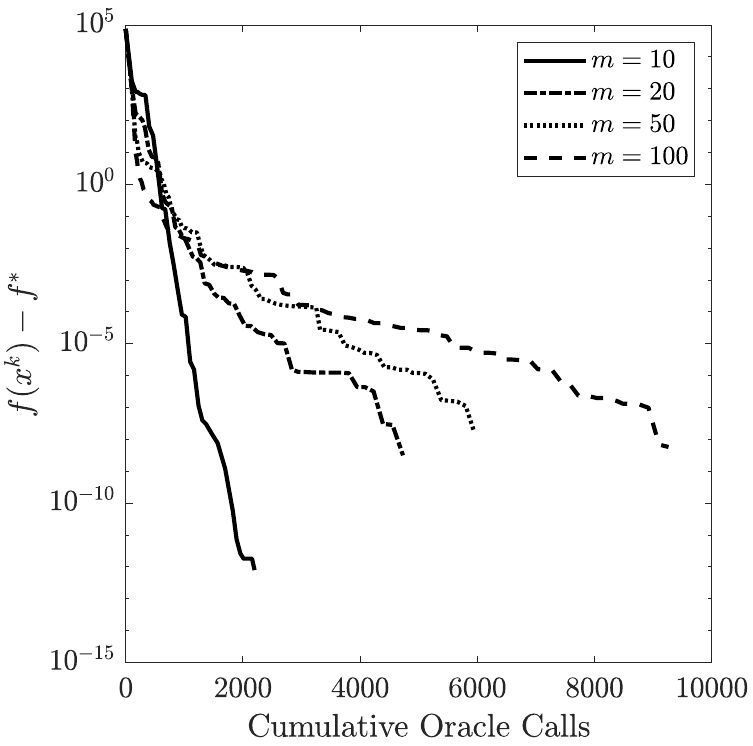}
\end{minipage}
\caption{\small Run time and oracle calls of SRDescent on marginal  functions with varying feasible sets for a fixed dimension $n = 300$ and varying $m \in \{10, 20, 50, 100\}$. For each instance, we use the final objective value as an estimation of $f^\ast$ to compute the objective gap.}
\label{fig:optval_vary_dim}
\end{figure}

\section{Conclusions}
\label{sec:conclusions}
In this work, we proposed a unifying principle for designing convergent descent methods in nonsmooth optimization through the concept of descent-oriented subgradients. This framework reveals the common structure of Goldstein-type methods and bundle methods, and motivates a novel technique for constructing descent directions, which we term subgradient regularization.
Our work opens several promising directions for future research. A key question is to develop computationally efficient descent-oriented subdifferentials for broader classes of nonsmooth functions. Additionally, extending linear convergence guarantees under regularity conditions to more general settings, as well as establishing iteration complexity bounds, are important topics for further exploration.

\bibliography{reference}

\appendix
\subsection*{A Omitted Proofs in Section \ref{sec4:subgrad-regularization}}
\subsubsection*{A.1 Proof of Lemma \ref{lem:G_max_marginal}}
Before proving Lemma \ref{lem:G_max_marginal}, we first establish a technical result that quantifies the objective value and gradient norm associated with a solution to the subgradient-regularized problem \eqref{eq:P_regularized}.
\begin{lemma}
\label{lem:bound_opt_val}
    Consider problem \eqref{eq:P}, and let $Y^\epsilon(x)$ be defined as in \eqref{eq:P_regularized}. Let $y^{\epsilon}: \R^n \to \R^m$ be a single-valued map with $y^{\epsilon}(x) \in Y^{\epsilon}(x)$ for all $(x, \epsilon) \in \R^{n}\times (0, +\infty)$. Then for all such $(x, \epsilon)$,\\[0.05in]
    (a) (Error estimate for the objective value)
    \[
        f(x) - \frac{\epsilon}{2}\inf\limits_{y \in Y^{\ast}(x)}\|\nabla_{x}\varphi (x, y)\|^{2}\leq \varphi (x, y^{\epsilon}(x)) \leq f(x).
    \]
    (b) (Bound on the gradient norm)
    \[
        \|\nabla_{x}\varphi(x, y^{\epsilon}(x))\| \leq \inf\limits_{y \in Y^{\ast}(x)} \|\nabla_{x}\varphi(x, y)\|.
    \]
\end{lemma}
\begin{proof}
    (a) The upper bound $\varphi (x, y^{\epsilon}(x)) \leq f(x)$ follows immediately from the definition of $f(x)$. For the lower bound, we use the optimality of $y^{\epsilon}(x) \in Y^{\epsilon}(x)$. For any $y \in Y^{\ast}(x)$, we have
    \begin{equation}\label{eq:opt_regularized-soln}
        \varphi (x, y) - \frac{\epsilon}{2}\|\nabla_{x}\varphi (x, y)\|^{2}
        \leq \varphi (x, y^{\epsilon}(x)) - \frac{\epsilon}{2} \|\nabla_{x}\varphi (x, y^{\epsilon}(x))\|^{2}
        \leq \varphi (x, y^{\epsilon}(x)).
    \end{equation}
    Taking the supremum over $y \in Y^{\ast}(x)$ on the left-hand side yields
    \[
        f(x) - \frac{\epsilon}{2}\inf_{y \in Y^{\ast}(x)}\|\nabla_{x}\varphi (
        x, y)\|^{2}
        \leq \varphi (x, y^{\epsilon}
        (x)).
    \]
    Here, the infimum is attained and finite since $Y^{\ast}(x)$ is compact and $\nabla_{x}\varphi (x, \cdot)$ is continuous.

    \vspace{0.05in}
    \noindent (b) For any $y \in Y^{\ast}(x)$, combining the inequality $\varphi (x, y^{\epsilon}(x)) \leq f(x)$ with \eqref{eq:opt_regularized-soln}, we have
    \[
        f(x) - \frac{\epsilon}{2}\|\nabla_{x}\varphi (x, y)\|^{2}
        = \varphi (x, y) - \frac{\epsilon}{2}\|\nabla_{x}\varphi (x, y)\|^{2}
        \leq f(x ) -\frac{\epsilon}{2}\|\nabla_{x}\varphi (x, y^{\epsilon}(x))\|^{2}.
    \]
    Thus, $\|\nabla_{x}\varphi(x, y^{\epsilon}(x))\| \leq \|\nabla_{x}\varphi(x, y)\|$. Taking the infimum over $y \in Y^{\ast}(x)$ completes the proof.
\end{proof}

We now proceed to prove Lemma \ref{lem:G_max_marginal}, using the results of Lemma \ref{lem:bound_opt_val}.
\begin{proof}[Proof of Lemma \ref{lem:G_max_marginal}]
    (a) Based on Lemma \ref{lem:bound_opt_val}(a), if we can further show that
    \begin{equation}\label{eq:limsup_bounded}
        \limsup_{x \to \bx}\left(\inf_{y \in Y^{\ast}(x)} \|\nabla_{x}\varphi (x, y)\|^{2}\right) < +\infty,
    \end{equation}
    the proof is then completed by taking the limits as $\epsilon \downarrow 0$ and $x \to \bx$ and using the continuity of $f$
    Observe that $\inf_{y \in Y^{\ast}(x)}\|\nabla_{x}\varphi (x, y)\| \leq \sup_{v \in \partial f(x)}\|v\|$ because $\partial f(x) = \conv \big(\bigcup\{\nabla_{x} \varphi(x, y) \mid y \in Y^{\ast}(x)\}\big)$.
    From~\cite[Proposition 2.1.2]{clarke1990optimization}, the term $\sup_{v \in \partial f(x)} \|v\|$ is upper bounded by the local Lipschitz constant of $f$ at $x$. Thus, \eqref{eq:limsup_bounded} is a direct consequence of the locally Lipschitz continuity of $f$ at $\bx$, and (a) follows.

    \vspace{0.05in}
    \noindent (b) Let $\{\epsilon_{k}\} \downarrow 0$ and $\{x^k\} \to \bx$.
    We aim to show that the sequence $\big\{\nabla_{x} \varphi(x^k, y^{\epsilon_k}(x^k))\big\}$ is bounded and that any of its accumulation points belongs to $\partial f(\bx)$. 
    Boundedness follows immediately from Lemma \ref{lem:bound_opt_val}(b) and our assumption that $\nabla_x \varphi(\cdot, \cdot)$ is continuous. Now consider a convergent subsequence $\big\{\nabla_{x} \varphi(x^k, y^{\epsilon_k}(x^k))\big\}_{k \in N}$ for some index set $N \in \N^{\sharp}_{\infty}$. Since $Y$ is compact, the corresponding subsequence $\{y^{\epsilon_k}(x^k)\}_{k \in N}\subset Y$ has at least one accumulation point $\bar y \in Y$. Let $y^{\epsilon_k}(x^k) \to_{N^\prime} \bar y$ along a further subsequence indexed by $N' \subset N$. By the continuity of $\varphi (\cdot, \cdot)$ and part (a), we have $\varphi(x^k, y^{\epsilon_k}(x^k)) \to_{N^\prime} \varphi(\bx, \bar y) = f(\bx)$, so $\bar y \in Y^{\ast}(\bx)$. Using the continuity of $\nabla_x \varphi(\cdot, \cdot)$, we obtain
    \[
        \lim\limits_{k(\in N) \to +\infty} \nabla_x \varphi(x^k, y^{\epsilon_k} (x^k)) 
        = \lim\limits_{k(\in N^\prime) \to +\infty} \nabla_x \varphi(x^k, y^{\epsilon_k} (x^k))
        = \;\nabla_x \varphi(\bx, \bar y) 
        \,\in\, \partial f(\bx).
    \]
    We can conclude that any accumulation point of $\big\{\nabla_{x} \varphi(x^k, y^{\epsilon_k}(x^k))\big\}$ is an element of $\partial f(x)$. This completes the proof of part (b).

    \vspace{0.05in}
    \noindent(c) Fix $\epsilon > 0$ and take any sequence $\{x^k\} \to \bx$. By Lemma \ref{lem:bound_opt_val}(b) and the continuity of $\nabla_x \varphi(\cdot, \cdot)$, the sequence $\big\{\nabla_{x} \varphi(x^k, y^{\epsilon}(x^k))\big\}$ is bounded. To establish the inclusion
    \begin{equation}\label{eq:result_part(c)}
        \limsup_{k \to +\infty} \nabla_x \varphi(x^k, y^{\epsilon}(x^k)) \subset \bigcup \big\{ \nabla_x \varphi(\bx, y) \mid y \in Y^{\epsilon}(\bx)\big\},
    \end{equation}
    it suffices to show that every accumulation point of the sequence $\{y^{\epsilon}(x^k)\}$ belongs to $Y^{\epsilon}(\bx)$. Recall that
    \[
        y^{\epsilon}(x) \,\in\, Y^{\epsilon}(x) = \argmax_{y \in Y}\left\{\varphi(x,y) - \frac{\epsilon}{2} \|\nabla_x \varphi(x,y)\|^2 \right\}.
    \]
    By assumption, both $\varphi(\cdot, \cdot)$ and $\nabla_x \varphi(\cdot, \cdot)$ are continuous. Hence, the objective function of the parametric maximization above is continuous in $x$ for every fixed $y \in Y$. It then follows from \cite[Theorem 1.17]{rockafellar2009variational} that any accumulation point of $\{y^{\epsilon}(x^k)\}$ must lie in $Y^{\epsilon}(\bx)$. 
    
    Observe that \eqref{eq:result_part(c)} holds for any single-valued selection $y^{\epsilon}: \R^n \to \R^m$ satisfying $y^{\epsilon}(x) \in Y^{\epsilon}(x)$ for all $(x, \epsilon) \in \R^{n}\times (0, +\infty)$. It follows that $\limsup_{x \to \bx}G(x, \epsilon) \subset G(\bx, \epsilon)$; equivalently, $\limsup_{x \to \bx}G (x, \epsilon) = G(\bx, \epsilon)$. Thus, $G$ satisfies the first condition in \eqref{eq:Def_G2_sufficient}.

    \vspace{0.05in}
    \noindent(d) Suppose that $S(\bx)$ is a convex set. It is immediate from \eqref{eq:Clarke_subdiff} that $\partial f(\bx) = \bigcup\{\nabla_{x}\varphi(\bx, y) \mid y \in Y^{\ast}(x)\}$.
    Take any sequence $\{\epsilon_{k}\} \downarrow 0$. By Lemma
    \ref{lem:bound_opt_val}, we know that
    \[
        \|\nabla_{x}\varphi(\bx, y^{\epsilon_k}(\bx))\| 
        \leq \inf_{y \in Y^\ast(\bx)}\|\nabla_{x}\varphi(\bx, y)\|
        = \inf_{v \in \partial f(\bx)}\|v\|,
    \]
    so the sequence $\{\nabla_{x}\varphi (\bx, y^{\epsilon_k}(\bx))\}$ is bounded and has at least one accumulation point. Additionally, by part (b), every accumulation point of this sequence belongs to $\partial f(\bx)$. Moreover, the previous bound $\|\nabla_{x}\varphi(\bx, y^{\epsilon_k}(\bx))\| \leq \inf_{v \in \partial f(\bx)}\|v\|$ ensures that the only possible accumulation point is $\argmin_{v \in \partial f(\bx)} \|v\|$. Therefore, $\{\nabla_{x}\varphi (\bx, y^{\epsilon_k}(\bx))\}$ is a bounded sequence with a unique accumulation point. Finally, if the sequence did not converge, it would have at least two distinct accumulation points, contradicting uniqueness of the minimal norm subgradient. Thus, $\nabla_{x}\varphi (\bx, y^{\epsilon_k}(\bx)) \to \argmin_{v \in \partial f(\bx)} \|v\|$, and the result follows because
    \[
        \lim_{\epsilon \downarrow 0}G(\bx, \epsilon) 
        = \lim_{\epsilon \downarrow 0}\bigcup \big\{\nabla_{x} \varphi(\bx, y^{\epsilon}(\bx)) \mid y^{\epsilon}(\bx) \in Y^{\epsilon}(\bx)\big\} 
        \,=\, \argmin_{v \in \partial f(\bx)} \|v\|.
    \]
\end{proof}

\subsubsection*{A.2 Proof of Lemma \ref{lem:G_opt_val}}
\begin{proof}
    (a) First, it is easy to see that for each fixed $\epsilon > 0$, $G(\cdot, \epsilon)$ is closed-valued and locally bounded since each $\nabla_x \varphi_j$ is continuous and $Y(\cdot)$ is locally bounded. We proceed by presenting technical results parallel to Lemma \ref{lem:bound_opt_val}. Let $y^\epsilon: \R^n \to \R^m$ be a single-valued map such that $y^{\epsilon}(x) \in Y^{\epsilon}(x)$ for all $(x, \epsilon) \in \R^n \times (0, +\infty)$. By the optimality of $y^{\epsilon}(x)$, we obtain the following estimate for the objective value:
    \begin{equation}\label{eq:(a)Obj_error}
        f(x) - \frac{\epsilon}{2} \inf_{y \in Y^{\epsilon}(x)} \|\nabla_x \mathcal{L}(x, y; \lambda(x))\|^2 \leq \varphi_0(x, y^{\epsilon}(x)) \leq f(x),
    \end{equation}
    along with a bound on the gradient norm:
    \begin{equation}\label{eq:(a)Gradient_norm}
        \|\nabla_x \mathcal{L}(x, y^{\epsilon}(x); \lambda(x))\|
        \leq \inf_{y \in Y^{\ast}(x)} \|\nabla_x \mathcal{L}(x, y; \lambda(x))\|
        = \inf_{v \in \partial f(x)} \|v\|.
    \end{equation}
    To verify property (G1), consider any sequences $\{x^k\} \to \bx$ and $\{\epsilon_k\} \downarrow 0$. Note that the sequence $\big\{\nabla_x \mathcal{L}\big(x^k, y^{\epsilon_k}(x^k); \lambda(x^k)\big)\big\}$ is bounded due to \eqref{eq:(a)Gradient_norm} and the local Lipschitz continuity of $f$. Since $y^{\epsilon_k}(x^k) \in Y(x^k)$, assumption (ii) ensures that $\{y^{\epsilon_k}(x^k)\}$ is bounded. Moreover, any accumulation point of $\{y^{\epsilon_k}(x^k)\}$ must lie in $Y^\ast(\bx)$ because $\varphi_0(x^k, y^{\epsilon_k}(x^k)) \to f(\bx)$ by \eqref{eq:(a)Obj_error}. In addition, the multiplier sequence $\{\lambda(x^k)\}$ is bounded due to the MFCQ, and its accumulation point must be $\lambda(\bx)$ by the continuity of each $\nabla_y \varphi_j(\cdot, \cdot)$. Hence,
    \[
        \limsup_{k \to +\infty} \nabla_x \mathcal{L}\big(x^k, y^{\epsilon_k}(x^k); \lambda(x^k)\big) 
        \subset \bigcup \big\{\nabla_x \mathcal{L}(\bx, y; \lambda(\bx)) \mid y \in Y^\ast(\bx)\big\} 
        \,=\, \partial f(\bx),
    \]
    which implies property (G1); that is, $\limsup_{x \to \bx, \epsilon \downarrow 0} G(x, \epsilon) \subset \partial f(\bx)$.

    To show $\limsup_{x \to \bx} G(x, \epsilon) = G(\bx, \epsilon)$ for any fixed $\epsilon > 0$, consider any sequence $\{x^k\} \to \bx$. The argument proceeds similarly to the one used above. The key difference is that, for fixed $\epsilon > 0$, any accumulation point of the sequence $\{y^{\epsilon}(x^k)\}$ must belong to $Y^{\epsilon}(\bx)$, which follows from the continuity of $\varphi(\cdot, \cdot)$ and each $\nabla_x \varphi_j(\cdot, \cdot)$.

    Finally, to establish $\lim_{\epsilon \downarrow 0} G(\bx, \epsilon) = \argmin_{v \in \partial f(\bx)}\|v\|$, we observe that property (G1) ensures $\limsup_{\epsilon \downarrow 0} G(\bx, \epsilon) \subset \partial f(\bx)$, while \eqref{eq:(a)Gradient_norm} guarantees that all vectors in $G(\bx, \epsilon)$ have norm at most $\inf_{v \in \partial f(\bx)} \|v\|$. Combining these two results and using the uniqueness of the minimal norm subgradient, we can deduce that $\lim_{\epsilon \downarrow 0} G(\bx, \epsilon) = \argmin_{v \in \partial f(\bx)} \|v\|$.

    So far, we have verified properties (G1) and (G2$^\prime$), and thus conclude that the mapping $G$ defined in \eqref{eq:G_LICQ} is a descent-oriented subdifferential.

    \vspace{0.05in}
    \noindent(b) The proof proceeds analogously under the assumption that $\varphi_0(x, \cdot)$ is strongly concave for each fixed $x$. In this case, both the error estimate for the objective value and the bound on the gradient norm follow from the same reasoning as in part (a). Properties (G1) and (G2$^\prime$) can be verified by applying analogous boundedness and convergence arguments. Therefore, the mapping $G$ defined in \eqref{eq:G_MFCQ} is also a descent-oriented subdifferential in this case.
\end{proof}

\subsection*{B Omitted Proofs in Section \ref{sec:local_linear_converge}}

To proceed, consider the proximal mapping parameterized by $x$:

\begin{equation}\label{eq:prox-linear-subprob}
    z^{\epsilon}(x) \triangleq \argmin_{z \in \R^n} \left\{ h\big(c(x) + \nabla c(x)^\top (z-x)\big) + \frac{1}{2 \epsilon}\|z - x\|^{2}\right\},
    \quad \epsilon > 0.
\end{equation}
The prox-linear update can then be expressed as $x^{k+1} = z^{\epsilon}(x^k)$. By Proposition \ref{prop:equivalent_prox-linear}, this yields
\[
    x^{k+1}= z^{\epsilon}(x^k) = x^k - \epsilon \, G(x^k, \epsilon) \quad\Rightarrow\quad G(x^k, \epsilon) = \frac{x^k- z^{\epsilon}(x^k)}{\epsilon}.
\]
It is not difficult to see that any point $\bx \in (\partial f)^{-1}(0)$ is the unique optimal solution of the strongly convex problem \eqref{eq:prox-linear-subprob} parameterized by $\bx$, i.e., $z^{\epsilon}(\bx) = \bx$. To understand the behavior of $G(x, \epsilon) = (x - z^{\epsilon}(x)) / \epsilon$ near $\bx$, we study a uniform stability of the solution map $z^{\epsilon}(x)$ around a neighborhood of $(\partial f)^{-1}(0)$, as established in the lemma below.

\begin{lemma}[Uniform Holder stability near stationary points]
\label{lem:Holder_stability}
    Consider problem \eqref{eq:cvx_composite}, and denote its set of stationary points by $(\partial f)^{-1}(0)$. Let $z^{\epsilon}(x)$ be defined as in \eqref{eq:prox-linear-subprob}.
    For any compact set $X \subset \R^{n}$ and any $\bar\epsilon > 0$, there exist positive constants $\delta$ and $D$ such that
    \[
        \sup_{\epsilon \in (0,\,\bar\epsilon)}\|z^{\epsilon}(x) - x\| 
        \leq D \cdot {\dist\big(x, (\partial f)^{-1}(0) \cap X\big)}^{1/2}
    \]
    for all $x$ with $\dist\big(x, (\partial f)^{-1}(0) \cap X\big) \leq \delta$.
\end{lemma}

\begin{remark}[Comparison with the pointwise stability]
    Applying a general perturbation result for optimal solutions~\cite[Proposition 4.37]{bonnans2013perturbation} to problem \eqref{eq:prox-linear-subprob}, for each fixed stationary point $\bx$, we have $\|z^{\epsilon}(x) - z^{\epsilon}(\bx)\| = \|z^{\epsilon}(x) - \bx\| = \mathcal{O}\big(\|x - \bx\|^{1/2}\big)$ and thus $\|z^{\epsilon}(x) - x\| = \mathcal{O}\big(\|x - \bx\|^{1/2}\big)$. In contrast, our result holds uniformly for all $x$ around $(\partial f)^{-1}(0) \cap X$, rather than a single point. Moreover, our bound depends on $\dist(x, (\partial f)^{-1}(0) \cap X)$, which is not only tighter than $\|x - \bx\|$ for stationary point $\bx \in X$ but can also be further controlled by metric subregularity in the proof of Proposition \ref{prop:constant_stepsize} (cf. Appendix B.2).
\end{remark}

\begin{proof}
    Define $\Phi(x) \triangleq \big\{(z,u) \in \R^{n}\times \R^{m}\mid u = c(x) + \nabla c(x)^\top (z-x)\big\}$ as the feasible region of problem \eqref{eq:prox-linear-subprob} parameterized by $x$, and rewrite problem \eqref{eq:prox-linear-subprob} as
    \begin{equation}\label{eq:prox-linear-subprob-lifted}
    \big(z^{\epsilon}(x), u^{\epsilon}(x)\big) \triangleq \displaystyle\argmin_{(z, u) \in \Phi(x)}\left\{
        h(u) + \frac{1}{2\epsilon}\|z-x\|^{2}
        \right\},
        \qquad\epsilon > 0.
    \end{equation}
    For any $\bx \in (\partial f)^{-1}(0)$, it is immediate that $\big(z^{\epsilon}(\bx), u^{\epsilon}(\bx)\big) = \big(\bx, c(\bx)\big)$.

    \vspace{0.05in}
    \noindent\textbf{Step 1} (Uniform stability of the feasible region $\Phi(x)$). Note that the feasible set $\Phi(x)$ is defined by an equality constraint for which the linear independence constraint qualification holds. By the stability theorem in~\cite[Theorem 2.87]{bonnans2013perturbation}, for any stationary point $\bx \in (\partial f)^{-1}(0) \cap X$, there exists a constant $C_{\bx}$ such that for all $(z, u, x^\prime)$ near $(\bx, c(\bx), \bx)$,
    \begin{equation}\label{eq:stability}
        \dist\big((z, u), \Phi(x^{\prime})\big) \leq C_{\bx} \big\|c(x^{\prime}) + \nabla c(x^{\prime})^\top (z-x^{\prime}) - u\big\|.
    \end{equation}
    Since $\partial f$ is osc and $X$ is compact, the set $(\partial f)^{-1}(0) \cap X$ is compact. By the continuity of $c(\cdot)$, the set $S \triangleq \{(\bx, c(\bx), \bx) \mid \bx \in (\partial f)^{-1}(0) \cap X \}$ is also compact. By the compactness of $S$ and the bound in \eqref{eq:stability}, we can find a uniform constant $C$ and a neighborhood $U(S)$ of $S$ such that
    \begin{equation}\label{eq:stability_uniform}
        \dist\big((z, u), \Phi(x^{\prime})\big) \leq
        C \big\|c(x^{\prime}) + \nabla c(x^{\prime})^\top (z-x^{\prime}) - u\big\|,
        \qquad\forall\, (z, u, x^{\prime}) \in U(S).
    \end{equation}
    For any $x$, let $\hx \in (\partial f)^{-1}(0) \cap X$ be such that $\|x - \hx\| = \dist\big(x, (\partial f)^{-1}(0) \cap X\big)$. 
    For all $x$ in a sufficiently small neighborhood of $(\partial f)^{-1}(0) \cap X$, we have $(\hx, c(\hx), x) \in U(S)$, so applying \eqref{eq:stability_uniform} with $(z,u,x^{\prime}) = (\hx, c(\hx), x)$ yields
    \begin{equation}\label{eq:dist_bx2x}
        \dist\big( (\hx, c(\hx)), \Phi(x)\big) 
        \leq C \big\|c(x) + \nabla c(x)^\top (\hx - x) - c(\hx) \big\| 
        \leq {C\beta} \|x - \hx\|^{2} / 2
        \leq C_1 \|x - \hx\|
    \end{equation}
    for some constant $C_1$, where the second inequality follows from the $\beta$-Lipschitz continuity of $\nabla c$.
    Next, for all $x$ sufficiently close to $(\partial f)^{-1}(0) \cap X$, take a neighborhood $\mathcal{B}(\bx, c(\bx))$ around $(\bx, c(\bx))$ such that $\mathcal{B}(\bx, c(\bx)) \times \{\hx\} \subset U(S)$ and $(z^{\epsilon}(x), u^{\epsilon}(x)) \in \mathcal{B}(\bx, c(\bx))$ for $\epsilon \in (0, \bar\epsilon)$. For such $x$, we apply \eqref{eq:stability_uniform} again with $(z, u) \in [\,\Phi(x) \cap \mathcal{B}(\bx, c(\bx))]$ and $x^{\prime} = \hx$ to obtain
    \[
    \begin{array}{rl}
        \dist \big((z,u), \Phi(\hx)\big) \leq & C \big\|c(\hx) + \nabla c(\hx)^\top (z - \hx) - c(x) - \nabla c(x)^\top (z - x) \big\|                  \\[0.07in]
        \leq                                  & C \big\|c(\hx) + \nabla c(\hx)^\top (x - \hx) - c(x)\big\| + C \big\|(\nabla c(\hx) - \nabla c(x))^\top(z - x)\big\| \\[0.07in]
        \leq                                  & {C \beta} \|x - \hx\|^2 / 2 + C \beta \|x - \hx\| \|z - x\|,
    \end{array}
    \]
    where the last inequality uses the $\beta$-Lipschitz continuity of $\nabla c$. Consequently, for all $x$ in a neighborhood of $(\partial f)^{-1}(0) \cap X$,
    \begin{equation}\label{eq:dist_x2bx}
    \begin{split}
        \D\Big(\Phi(x) \cap \mathcal{B}(\bx, c(\bx)),\, \Phi(\hx)\Big)
        &\leq \sup_{(z,u) \in [\Phi(x) \cap \mathcal{B}(\bx, c(\bx))]} \dist \big((z,u), \Phi(\hx)\big)
        \,\leq\; C_{2} \|x - \hx\|
    \end{split}
    \end{equation}
    for some constant $C_2$. So far, we have derived two bounds \eqref{eq:dist_bx2x} and \eqref{eq:dist_x2bx}, measuring the local distance between $\Phi(x)$ and $\Phi(\hx)$, from the stability theorem.
    
    \vspace{0.05in}
    \noindent\textbf{Step 2} (A quadratic inequality to bound $\|z^{\epsilon}(x) - x\|$).
    From now on, we fix $x$ in a neighborhood of $(\partial f)^{-1}(0) \cap X$ such that \eqref{eq:dist_bx2x} and \eqref{eq:dist_x2bx} hold. Recall that $\hx \in (\partial f)^{-1}(0) \cap X$ is chosen so that $\|x - \hx\| = \dist(x, (\partial f)^{-1}(0) \cap X)$. Since $\Phi(\hx)$ is a closed set, let $(z, u) \in \Phi(\hx)$ be such that $\big\|(z^{\epsilon}(x), u^{\epsilon}(x)) - (z, u)\big\| = \dist\big((z^{\epsilon}(x), u^{\epsilon}(x)),\, \Phi(\hx)\big)$.
    Then,
    \begin{equation}\label{eq:dist_gamma}
        \|u^{\epsilon}(x) - u\| 
        \leq \dist\Big((z^{\epsilon}(x), u^{\epsilon}(x)),\, \Phi(\hx)\Big) 
        \overset{\eqref{eq:dist_x2bx}}{\leq}C_{2}\|x - \hx\|.
    \end{equation}
    Since $\hx$ is a stationary point of problem \eqref{eq:cvx_composite}, $0 \in \nabla c(\hx) \, \partial h (c(\hx))$. By the convexity of $h$, the subgradient inequality yields $h(u) = h \big(c(\hx) + \nabla c(\hx)^\top(z-\hx)\big) \geq h (c(\hx))$ for any $x$. Hence,
    \begin{equation}\label{eq:func_val_lowerbound}
    \begin{array}{rl}
        \displaystyle h(u^{\epsilon}(x)) + \frac{1}{2\epsilon} \|z^{\epsilon}(x) - \hx\|^2 - h (c(\hx))
        &\geq \displaystyle h(u^{\epsilon}(x)) - h(u) + \frac{1}{2\epsilon} \|z^{\epsilon}(x) - \hx\|^2 \\[0.1in]
        &\overset{\eqref{eq:dist_gamma}}\geq \displaystyle -L C_2 \|x - \hx\| + \frac{1}{2\epsilon} \|z^{\epsilon}(x) - \hx\|^2 .
    \end{array}
    \end{equation}
    To bound the same quantity from above, we fix any pair $(z^{\prime}, u^{\prime}) \in \Phi(x)$. Since $(z^{\epsilon}(x), u^{\epsilon}(x))$ is the optimal solution to \eqref{eq:prox-linear-subprob-lifted}, we have
    $
        h(u^{\epsilon}(x)) + \frac{1}{2\epsilon} \|z^{\epsilon}(x) - x\|^{2}
        \leq h(u^{\prime}) + \frac{1}{2\epsilon} \|z^{\prime} - x\|^{2}
    $. 
    Subtracting $h(c(\hx))$ from both sides and rearranging terms, we obtain
    \begin{align*}
         & \qquad h(u^{\epsilon}(x)) + \frac{1}{2\epsilon} \|z^{\epsilon}(x) - \hx\|^2 - h (c(\hx)) \\
         & \leq h(u^{\prime}) - h (c(\hx)) + \frac{1}{2\epsilon} \big(\|z^{\prime}-x\|^{2} - \|\hx-x\|^{2}\big) + \frac{1}{2\epsilon} \big(\|\hx - x\|^{2}- \|z^{\epsilon}(x) - x\|^{2}+ \|z^{\epsilon}(x) - \hx\|^{2}\big) \\
         & \leq L \|u^{\prime}-c(\hx)\| + \frac{1}{2\epsilon} \|z^{\prime}+\hx-2x\| \|z^{\prime}-\hx\| + \frac{1}{\epsilon} \|x -\hx\| \|z^{\epsilon}(x)-\hx\|  \\[0.03in]
         & \leq \frac{1}{2\epsilon} \max\left\{2\epsilon L, \|z^{\prime}+\hx-2x\| \right\} \cdot \left\|(z^{\prime}, u^{\prime}) - (\hx, c(\hx))\right\| + \frac{1}{\epsilon} \|x - \hx\| \|z^{\epsilon}(x) - \hx\|.
    \end{align*}
    Using the estimate $\|z^{\prime}+ \hx - 2x\| \leq \|z^{\prime}- \hx\| + 2\|\hx - x\| \leq \left\|(z^{\prime}, u^{\prime}) - (\hx, c(\hx))\right\| + 2\|\hx - x\|$ and minimizing the right-hand side over $(z^{\prime}, u^{\prime}) \in \Phi(x)$, we get
    \begin{equation}
    \label{eq:func_val_upperbound}
        \begin{array}{rl}
             & \displaystyle\quad h(u^{\epsilon}(x)) + \frac{1}{2\epsilon} \|z^{\epsilon}(x) - \hx\|^2 - h (c(\hx)) \\[0.1in]
             & \;\leq\; \displaystyle \frac{1}{2\epsilon} \max\Big\{2\epsilon L,\, \dist\big((\hx, c(\hx)), \Phi(x) \big) + 2\|\hx - x\| \Big\} \cdot \dist\big((\hx, c(\hx)), \Phi(x) \big) \\[0.1in]
             &\qquad\qquad \displaystyle 
             + \frac{1}{\epsilon} \|x - \hx\| \|z^{\epsilon}(x) - \hx\| \\[0.08in]
             & \overset{\eqref{eq:dist_bx2x}}{\leq} \displaystyle \frac{1}{2\epsilon}\max\big\{2\epsilon L, (C_{1}+ 2) \|x - \hx\| \big\} \cdot C_1 \|x - \hx\| + \frac{1}{\epsilon} \|x - \hx\| \|z^{\epsilon}(x) - \hx\|.
        \end{array}
    \end{equation}
    Combining the lower bound \eqref{eq:func_val_lowerbound} and the upper bound \eqref{eq:func_val_upperbound}, we obtain the quadratic inequality:
    \[
        \beta_{1}+ \frac{1}{2\epsilon} \|z^{\epsilon}(x) - \hx\|^2 
        \leq \beta_{2} + \frac{1}{\epsilon} \|x - \hx\| \|z^{\epsilon}(x) - \hx\|,
    \]
    where $\beta_{1}\triangleq -L C_{2}\|x - \hx\|$ and $\beta_{2} \triangleq (2\epsilon)^{-1} \max\left\{2\epsilon L, (C_{1}+ 2) \|x - \hx\| \right\} \cdot C_{1} \|x - \hx\|$. Solving this inequality in terms of $\|z^{\epsilon}(x) - \hx\|$ gives
    \[
    \begin{array}{rl}
        \|z^{\epsilon}(x) - \hx\| 
        &\leq \|x - \hx\| + \sqrt{\|x - \hx\|^{2}- 2 \epsilon (\beta_{1}- \beta_{2}) }\leq 2 \|x - \hx\| + \sqrt{2\epsilon(\beta_{2}- \beta_{1})} \\[0.08in]
        &= 2 \|x - \hx\| + \sqrt{C_{1}\max\big\{2\epsilon L, \left(C_{1}+ 2\right) \|x - \hx\| \big\} + 2\epsilon L C_{2}} \cdot \sqrt{\|x - \hx\|}.
    \end{array}
    \]
    Finally, since $\|z^{\epsilon}(x) - x\| \leq \|z^{\epsilon}(x) - \hx\| + \|x - \hx \|$, taking supreme over $\epsilon \in (0, \bar\epsilon)$, we can find a constant $D > 0$ such that
    \[
        \sup_{\epsilon \in (0,\,\bar\epsilon)}\|z^{\epsilon}(x) - x\| 
        \leq D \|x - \hx\|^{\frac{1}{2}}
        = D \cdot \dist\big(x, (\partial f)^{-1}(0) \cap X \big)^{\frac{1}{2}}
    \]
    for all $x$ in a neighborhood of $(\partial f)^{-1}(0) \cap X$.
\end{proof}

\subsubsection*{B.1 Proof of Lemma \ref{lem:from_metric_subregular_to_error_bound}} 

\begin{lemma}[{\cite[Theorem 5.3]{drusvyatskiy2018error}}]
\label{lem:Variational_principle}
    Consider problem \eqref{eq:cvx_composite} with $G$ defined in \eqref{eq:G_cvx-composite}, and let $z^{\epsilon}(x)$ be the unique solution to the prox-linear subproblem \eqref{eq:prox-linear-subprob}. For any $\epsilon > 0$, there exists a point $\hx$ such that $\|z^{\epsilon}(x) - \hx\| \leq \|z^{\epsilon}(x) - x\|$ and $\dist(0, \partial f(\hx)) \leq (3L\beta + 2/\epsilon) \|z^{\epsilon}(x) - x\|$.
\end{lemma}

We now prove Lemma \ref{lem:from_metric_subregular_to_error_bound} using Lemmas \ref{lem:Holder_stability} and \ref{lem:Variational_principle}.
\begin{proof}[Proof of Lemma \ref{lem:from_metric_subregular_to_error_bound}]
    Let $\tau > 0$ be such that the inequality of metric subregularity holds for all $x \in \B_{\tau}(\bx)$. Applying Lemma \ref{lem:Holder_stability} with $X = \{\bx\}$, there exist positive scalars $\delta$ and $D$ such that
    \begin{equation}
    \label{eq:stability_delta_close}
        \sup_{\epsilon \in (0,\,\bar\epsilon)}\|z^{\epsilon}(x) - x\| \leq D \|x - \bx\|^{\frac{1}{2}},
        \qquad\forall\,x \in \B_{\delta}(\bx).
    \end{equation}
    Fix any $\varepsilon \in (0, \delta)$ such that ${\varepsilon + 2D \sqrt{\varepsilon}} \leq \tau$, and consider a point $x \in \B_{\varepsilon}(\bx) \subset \B_{\delta}(\bx)$. By Lemma \ref{lem:Variational_principle}, there exists a point $\hx$ satisfying $\|z^{\epsilon}(x) - \hx\| \leq \|z^{\epsilon}(x) - x\|$ and $\dist(0, \partial f(\hx)) \leq (3L\beta + 2/\epsilon) \|z^{\epsilon}(x) - x\|$. Combining this with the bound \eqref{eq:stability_delta_close}, we obtain
    \[
    \begin{array}{rl}
        \|\bx - \hx\| &\;\leq\; \|\bx - x\| + \|x - z^{\epsilon}(x)\| + \|z^{\epsilon}(x) - \hx \| \leq \|\bx - x\| + 2\|z^{\epsilon}(x) - x\| \\
        &\overset{\eqref{eq:stability_delta_close}}{\leq} \|\bx - x\| + 2D \|x - \bx\|^{\frac{1}{2}}
        \;\leq\; \varepsilon + 2D \sqrt{\varepsilon} 
        \;\leq\; \tau.
    \end{array}
    \]
    Hence $\hx \in \B_{\tau}(\bx)$, and $\dist(\hx, (\partial f)^{-1}(0)) \leq \kappa \cdot \dist\big(0, \partial f(\hx)\big)$ by metric subregularity. Then,
    \[
    \begin{array}{rl}
        \dist\big(x, (\partial f)^{-1}(0)\big)
        & \leq \|x - z^{\epsilon}(x)\| + \|z^{\epsilon}(x) - \hx\| + \dist\big( \hx, (\partial f)^{-1}(0) \big) \\[0.08in]
        & \leq 2 \|z^{\epsilon}(x) - x\| + \kappa \cdot \dist(0, \partial f(\hx)) 
        \,\leq\, \big(2 + (3L\beta + 2/\epsilon) \kappa\big) \cdot \|z^{\epsilon}(x) - x\|.
    \end{array}
    \]
    Finally, since $\|z^{\epsilon}(x) - x\| = \epsilon \|G(x, \epsilon)\|$, the result follows.
\end{proof}

\subsubsection*{B.2 Proof of Proposition \ref{prop:constant_stepsize} (Eventual constant stepsize)}
\begin{proof}
    Suppose, for the sake of contradiction, that $\epsilon_{k,0}\downarrow 0$. Then there exists an index set $N \in \N^{\sharp}_{\infty}$ such that $\|g^{k, i_k}\| \to_{N} 0$, and the ratio test fails for all $k \in N$, i.e.,
    \begin{equation}\label{eq:ratio_test_fail}
        \frac{\tilde\epsilon_{k}\|\tilde g^k\|}{\|g^{k, i_k}\|^{1/2}} > \frac{(\epsilon_{k,i_k})^{1/2}}{\epsilon_{k,0}}.
    \end{equation}
    By Corollary \ref{cor:uniform_descent_direc} and given $\alpha \in (0,1)$, any $\epsilon_{k,i} \in (0, 1/(2L \beta) ]$ with stepsize $\eta = \epsilon_{k,i}$ ensures the descent condition. Thus, $\epsilon_{k,i_k}\geq \min\left\{\epsilon_{k,0}, (4 L \beta)^{-1}\right\}$, which implies that the right-hand side of \eqref{eq:ratio_test_fail} diverges to infinity as $k (\in N) \to +\infty$. Since $f(x^k) \leq f(x^0)$ for all $k$, the sequence $\{x^k\}_{k \in N}\subset X$ is bounded. Let $\{x^k\}_{k \in N^\prime}$ be a convergent subsequence such that $x^k \to_{N^\prime} \bx$ for some $N^\prime \in \N^{\sharp}_{\infty}$ with $N^{\prime}\subset N$. Because $\|g^{k, i_k}\| \to_{N}0$, it follows from \eqref{eq:Def_G1} that $0 \in \partial f(\bx)$. By the continuity of $f$, we have $f(x^k) \to_{N^\prime} f(\bx)$ and $f(\bx) \leq f(x^0)$, so $\bx \in (\partial f)^{-1}(0) \cap X$.

    Since $\partial f$ is metrically subregular at $\bx$ with some modulus $\kappa_{\bx} > 0$, Lemma \ref{lem:from_metric_subregular_to_error_bound} with $\bar\epsilon = \epsilon_{0,0}$ guarantees the existence of a positive number $\tau$ such that the error bound property holds:
    \[
        \dist\big(x, (\partial f)^{-1}(0)\big) \leq ( (3L \beta \epsilon + 2) \kappa_{\bx}+ 2\epsilon ) \cdot \|G(x, \epsilon)\|,
        \qquad\forall\, x \in \B_{\tau}(\bx),\, \epsilon \in (0, \epsilon_{0,0}).
    \]
    In particular, for all sufficiently large $k \in N^{\prime}$, we have
    \begin{equation}
    \label{eq:ErrorBound_specify}
        \dist\big(x^k, (\partial f)^{-1}(0) \big) \leq \big( (3L \beta \epsilon_{k,i_k}+ 2) \kappa_{\bx}+ 2\epsilon_{k,i_k} \big) \cdot \|g^{k, i_k}\| .
    \end{equation}
    For each $k \in N^{\prime}$, let $\hx^k\in (\partial f)^{-1}(0)$ be such that $\|x^k- \hx^k\| = \dist\left(x^k, (\partial f)^{-1}(0) \right)$. Observing that the right-hand side of \eqref{eq:ErrorBound_specify} converges to $0$ as $k(\in N^{\prime}) \to +\infty$, we have $\|x^k- \hx^k\| \to_{N^\prime} 0$, and hence $\hx^k\to_{N^\prime}\bx$. This yields the existence of a compact set $\widehat{X}$ containing $\{\hx^k\}_{k \in N^\prime}$. By Lemma \ref{lem:Holder_stability} with $\bar\epsilon=(a_0)^{1/4}$, there exists $D > 0$ such that for all sufficiently large $k \in N^{\prime}$,
    \[
        \big\|z^{\tilde\epsilon_k} (x^k) - x^k\big\| 
        \leq D \cdot \dist\Big(x^k, (\partial f)^{-1}(0) \cap \widehat{X}\Big)^{\frac{1}{2}}
        = D \|x^k- \hx^k\|^{\frac{1}{2}}
        = D \cdot \dist\big(x^k, (\partial f)^{-1}(0) \big)^{\frac{1}{2}}.
    \]
    Combining this with the error bound \eqref{eq:ErrorBound_specify},
    we obtain
    \[
        \tilde\epsilon_{k}\|\tilde g^k\| 
        = \big\|z^{\tilde\epsilon_k}(x^k) - x^k\big\| 
        \leq D \big((3L \beta \epsilon_{k, i_k}+ 2) \kappa_{\bx} + 2 \epsilon_{k, i_k}\big)^{\frac{1}{2}} \cdot \|g^{k, i_k}\|^{\frac{1}{2}} \,.
    \]
    Dividing both sides by $\|g^{k, i_k}\|^{1/2}$ gives
    \begin{align*}
        \frac{\tilde\epsilon_{k}\|\tilde g^k\|}{\|g^{k, i_k}\|^{1/2}}\leq D \big((3L \beta \epsilon_{k, i_k}+ 2) \kappa_{\bx} + 2 \epsilon_{k, i_k}\big)^{\frac{1}{2}}
        \longrightarrow_{N^\prime} 0,
    \end{align*}
    which contradicts the divergence ${\tilde\epsilon_{k}\|\tilde g^k\|}/{\|g^{k, i_k}\|^{1/2}} \to +\infty$ in \eqref{eq:ratio_test_fail}. 
    Therefore, the sequence $\{\epsilon_{k,0}\}$ cannot converge to zero. There exists $\underline\epsilon > 0$ such that $\epsilon_{k,0}= \underline
    \epsilon$ for all sufficiently large $k$. Using Corollary
    \ref{cor:uniform_descent_direc} again, we conclude that
    \[
        \eta_{k}
        \geq \epsilon_{k,i_k}\geq \min\left\{\epsilon_{k,0}, ({4 L \beta})^{-1}\right\} 
        \geq \min\left\{\underline\epsilon, ({4 L \beta})^{-1}\right\},
        \qquad\forall\, k \in \N.
    \]
\end{proof}

\end{document}